\documentclass[11pt,a4paper]{article}
\usepackage[]{cmbright} 
\usepackage{fullpage}
\usepackage{tikz, pgfplots}
\usetikzlibrary{calc,backgrounds}
\usepackage{graphicx}
\usepackage{subfigure}
\usepackage{amsmath, amsfonts, amssymb, mathrsfs, amsthm, dsfont, color}
\usepackage{fancyhdr}
\usepackage[squaren]{SIunits}

\definecolor{webred}{rgb}{0.85,0,0}
\usepackage[citecolor=blue,colorlinks=true,linkcolor=webred]{hyperref}

\date{}

\makeatletter
\def\@listI{
    \leftmargin\leftmargini
    \parsep 1.5pt plus 1pt minus 1pt
    \topsep 2pt plus 1pt minus 1pt
    \itemsep \parsep}
\let\@listi\@listI
\makeatother

\renewcommand{\geq}{\geqslant} 
\renewcommand{\ge}{\geqslant} 
\renewcommand{\leq}{\leqslant} 
\renewcommand{\le}{\leqslant} 
\renewcommand{\Re}{\mathop{\textup {Re}}}
\renewcommand{\Im}{\mathop{\textup {Im}}}

\newtheorem{thm}{Theorem}[section]
\newtheorem{cor}[thm]{Corollary}
\newtheorem{lem}[thm]{Lemma}
\newtheorem{prop}[thm]{Proposition}

\theoremstyle{definition}
\newtheorem{defn}[thm]{Definition}

\theoremstyle{remark}
\newtheorem{remark}[thm]{Remark}
\newtheorem{exm}[thm]{Example}

\newcommand  {\C}{{\mathbb C}}
\newcommand  {\N}{{\mathbb N}}

\newcommand  {\R}{{\mathbb R}}
\newcommand  {\T}{{\mathbb T}}
\newcommand  {\Z}{{\mathbb Z}}

\newcommand  {\Sbb}{{\mathbb S}}
\newcommand  {\TSbb}{{\mathbb J}}

\newcommand {\re}{\mathrm {e}}
\newcommand {\ri}{\mathrm {i}}
\newcommand {\rd}{\mathrm {d}}
\newcommand {\srm}{\mathrm {s}}
\newcommand {\krm}{\mathrm {k}}

\newcommand {\Sfrak}{\mathfrak {S}}
\newcommand {\Kfrak}{\mathfrak {K}}
\newcommand {\Ufrak}{\mathfrak {U}}

\newcommand {\ffrak}{\mathfrak {f}}
\newcommand {\gfrak}{\mathfrak {g}}
\newcommand {\kfrak}{\mathfrak {k}}
\newcommand {\sfrak}{\mathfrak {s}}
\newcommand {\ufrak}{\mathfrak {u}}

\newcommand {\Bcal}{\mathcal {B}}
\newcommand {\Gcal}{\mathcal {G}}
\newcommand {\Jcal}{\mathcal {J}}
\newcommand {\Mcal}{\mathcal {M}}
\newcommand {\Acal}{\mathcal {A}}
\newcommand {\Ccal}{\mathcal {C}}
\newcommand {\Scal}{\mathcal {S}}
\newcommand {\Ocal}{\mathcal {O}}

\newcommand {\Lscr}{\mathscr{L}}
\newcommand {\Mscr}{\mathscr{M}}

\newcommand {\Zscr}{\mathscr{Z}}

\newcommand {\Ssf}{\mathsf{S}}
\newcommand {\Tsf}{\mathsf{T}}

\newcommand{\di}{\displaystyle}

\newcommand{\jump}[1]{\left[ #1 \right]}
\newcommand{\potA}{\mathcal{A}}
\newcommand{\Span}{\mathsf{Span}}
\newcommand{\Res}{\mathop{\textup {Res}}}
\newcommand{\ee}{\hskip 0.15ex}
\newcommand{\indic}{\mbox{\large$\mathds{1}$}}
\newcommand{\angsec}{45}

\DeclareMathOperator*{\equi}{\sim}
\DeclareMathOperator*{\equa}{=}
\DeclareMathOperator{\sgn}{sgn}

\let\un\underline

\usepackage{color}

\definecolor{gr}{rgb}   {0.,   0.69,   0.23 }
\definecolor{yl}{rgb}   {0.5,   0.5,   0.1 }
\definecolor{bl}{rgb}   {0.,   0.5,   1. }
\definecolor{bf}{rgb}   {0.,   0.,   0.75}
\definecolor{mg}{rgb}   {0.85,  0.,    0.85}

\title{Corner asymptotics of the magnetic potential in the
  eddy-current model}

\author{%
  M. Dauge$^1$, P. Dular$^2$, L. Kr\"ahenb\"uhl$^3$, V. P\'eron$^{4}$,
  R. Perrussel$^5$ and C. Poignard$^6$
  \\ [1mm]
  {\small $^1$ IRMAR CNRS UMR6625, Rennes, France}\\[1mm]
  {\small $^2$ F.R.S.-FNRS, ACE research unit, Li\`ege, Belgium}\\[1mm]
  {\small $^3$ Laboratoire Amp\`ere CNRS UMR5005, Lyon, France}\\[1mm]
  {\small $^4$ LMAP CNRS UMR5142 \& Team MAGIQUE3D INRIA, Universit\'e de
    Pau, France}
  \\[1mm]
  {\small $^5$ LAPLACE CNRS UMR5213, Toulouse, France}\\[1mm]
  {\small $^6$ Team MC2, INRIA Bordeaux-Sud-Ouest \& CNRS UMR 5251,
    Bordeaux, France}
  \\[1mm]
}

\begin{document}
\maketitle
\setcounter{tocdepth}{2}
\tableofcontents
\clearpage

\begin{abstract}
  In this paper, we describe the magnetic potential in the vicinity
  of a corner of a conducting  body embedded in a dielectric medium in
  a  bidimensional setting.   We make  explicit the  corner asymptotic
  expansion for this  potential as the distance to  the corner goes to
  zero.   This  expansion  involves  singular functions  and  singular
  coefficients.   We introduce  a method  for the calculation   of the
  singular  functions near  the  corner and  we  provide two methods
  to compute the  singular
  coefficients: the method of moments and the method of quasi-dual
  singular functions. Estimates for the convergence of both
  approximate  methods are proven.   
  We eventually  illustrate the  theoretical  results with
  finite element computations.  
  The specific non-standard feature of this problem lies in the
  structure of its singular functions: They have the form of series
  whose first terms are harmonic polynomials and further terms are
  genuine non-smooth functions generated by the piecewise constant zeroth order
  term of the operator.
\end{abstract}

\section{Introduction}

Accurate  knowledge of  eddy currents  is  of great  interest for  the
design   of   many   electromagnetic   devices   such   as   used   in
electrothermics.   Taking  advantage  of  the  fast  decrease  of  the
electromagnetic field  inside the conductor,  impedance conditions are
usually considered to reduce  the computational domain.  Impedance
conditions  were first proposed  by Leontovich~\cite{Leo48} and by
Rytov~\cite{Rytov40}   in the 1940's and  then extended by Senior and
Volakis \cite{SeVo95}.  These
conditions  can  be derived  up  to  the  required precision  for  any
conductor with  a smooth  surface: we refer  to Haddar {\it et
  al.}~\cite{HJN08} for the mathematical  justification of
conditions  of order  $1$, $2$  and $3$.
Note however that such conditions are not valid near corners. Few authors
have proposed  heuristic impedance modifications close  to the corners
\cite{deeley1990,  ida2001}, but  these  modifications   are  neither
satisfactory  nor  proved.    In  particular  in  \cite{ida2001},  the
modified impedance  appears  to blow up  near the corner, which
does  not  seem  valid   for  non-magnetic  materials  with  a  finite
conductivity as presented in \cite{buret2011}.

In  the  present paper,  we  do not  consider  the  derivation of  the
impedance  condition,  which  is   an  asymptotic  expansion  in  high
frequency  (or high  conductivity). We  are rather  interested  in the
description of  the magnetic potential  in the vicinity of  the corner
 of  a conducting  body embedded  in a dielectric  medium   in a
bidimensional  setting, considering  the conductivity  $\sigma$, as
well as  the angular frequency  $\kappa$, as given  parameters.  We
emphasize that  these considerations will be  helpful in understanding
the behavior of the impedance condition near corner singularities that
will be considered in a forthcoming paper.

Roughly  speaking, the  aim of  the present  paper is  to  provide the
asymptotic expansion of the magnetic  potential as the distance to the
corner goes to zero.
This notion of corner asymptotics generalizes the Taylor expansion, which
holds in smooth  domains.  Such asymptotics involve two main
ingredients:
\begin{itemize}
\item  The singular  functions,  also called {\em singularities}, which  belong to  the  kernel of  the
  considered operator.
\item  The  singular   coefficients,  whose  calculation  requires  the
  knowledge of (quasi-) dual singular functions.
\end{itemize}
In  the
present paper, we consider the magnetic potential
in  a non-magnetic domain  composed of  a conducting  material $\Omega_-$  with one
corner surrounded by a dielectric material $\Omega_+$. Thus the
operator acting on the magnetic potential is not homogeneous and has a
discontinuous piecewise  constant coefficient  in front of  its zeroth
order part:
$$
- \Delta +\ri\kappa\mu_0 \sigma \indic_{\Omega_-}.
$$
 Though  pertaining to  the wide class  of elliptic  boundary value
problems or transmission problems in conical domains, see for instance
the papers  \cite{kondratev1967,KondratevOleinik83} and the monographs
\cite{Grisvard85,Dauge88,Nicaise93,KozlovMazyaRossmann97b},        this
problem  has  specific  features  which  make the  derivation  of  the
asymptotics  not  obvious---namely  the  fact that  singularities  are
generated by a non-principal term.  Despite its great interest for the
applications, this problem has not yet been explicitly analyzed.

Another  disturbing factor is  the nature of
the  limit problem when  the product  $\kappa\mu_0 \sigma$  tends  to infinity
(large frequency/high  conductivity limit). This limit is  simply the
homogeneous Dirichlet problem for  the Laplace operator set in
the dielectric  medium $\Omega_+$, whose corner singularities  are well
known   \cite{kondratev1967,Grisvard85}.  In   particular,   when  the
conductor  $\Omega_-$  has a  convex  corner,  its surrounding  domain
$\Omega_+$ has a non-convex corner, thus the Dirichlet problem has non
$\Ccal^1$ singularities, in opposition to the problem for any finite
$\kappa\mu_0 \sigma$. This apparent paradox can be solved by
a   delicate  multi-scale  analysis, whose heuristics are exposed
in~\cite{buret2011}. 
 Roughly speaking, there  exists \emph{profiles} in  $\R^2$ which
have  the singular  behavior of  the operator
$-\Delta  +2 \ri  \indic_{\Omega_-}$ near the corner, and which connect at infinity with
the singular functions of the Dirichlet  problem in $\Omega_+$, 
as described by equation (13) in \cite{buret2011}. 
This is why
the  knowledge of  the  singularities of  $-  \Delta +\ri  \kappa\mu_0
\sigma  \indic_{\Omega_-}$  at  given   $\kappa$  and  $\sigma$  is  a
milestone in the full multi-scale analysis.

Besides   the  mere  description   of  singular   functions,  the
computation of singular
coefficients is considered in  this paper.  Various works concern
the extraction  of singular coefficients associated  with the solution
to  an elliptic problem set in  a domain with a corner singularity, 
see  \cite{MazyaPlamenevskii84c,DNBL90} for theoretical  formulas, and
\cite{Mo85,SzaboYosibash2,CoDaYo04} for  more practical methods.  Here,
we choose  to extend the  quasi-dual function method initiated in
\cite{CoDaYo04} to the  case of resonances, which was  discarded in the
latter reference.

Let us now present the problem considered throughout the paper.

\subsection{Statement of the problem}

Denote by  $\Omega_-\subset \R^2$ the bounded  domain corresponding to
the conducting  medium, and by $\Omega_+\subset  \R^2$ the surrounding
dielectric   medium  (see   Figure~\ref{fig:modeldom}).    The  domain
$\Omega$   with   a   smooth   boundary   $\Gamma$   is   defined   by
$\Omega=\Omega_-\cup\Omega_+\cup\Sigma$,   where   $\Sigma$   is   the
boundary of $\Omega_-$.  For the sake of simplicity, we assume that:
\begin{enumerate}
\item $\Sigma$  has only one  geometric singularity, and we  denote by
  $\mathbf{c}$  this  corner.   The  angle  of the  corner  (from  the
  conducting  material, see  Figure~\ref{fig:modeldom}) is  denoted by
  $\omega$.
\item The current source term  $J$, located in $\Omega_+$, is a smooth
  function, which vanishes in a neighborhood of $\mathbf{c}$.
\end{enumerate}

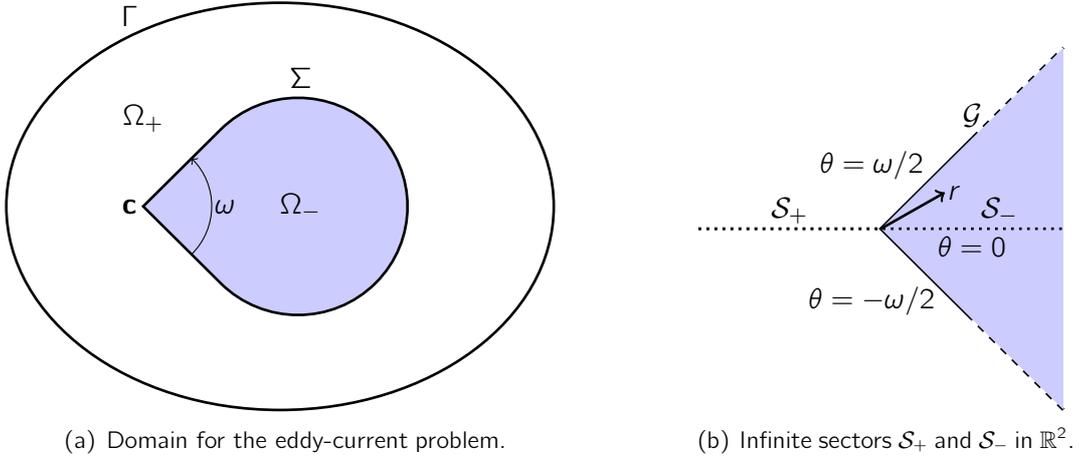
\begin{figure}[ht!]
  \centering%
  \subfigure[Domain for the eddy-current problem.]
  {%
    \begin{tikzpicture}[scale=.9]
      \draw[line width=1pt] (2., .0) circle (4. and 3.);%
      \fill[draw=black, fill=blue!20, line width=1pt] %
      (-45.:1.6 cm) arc (-135:135:1.6 cm) -- (.0,.0) -- (-45:1.6 cm);%
      \node at (-.2, 2.8) {$\Gamma$};%
      \node at (.0, 1.3) {$\Omega_+$};%
      \node at (2.3, 1.9) {$\Sigma$};%
      \node at (-.2, .0) {$\mathbf{c}$};%
      \node at (2.3, 0.) {$\Omega_-$};%
      \draw[->] (-45:1. cm) arc(-45:45:1. cm);%
      \node at (1.2, 0.) {$\omega$};%
    \end{tikzpicture}
    \label{fig:modeldom}
  }%
  \hfil
  \subfigure[Infinite sectors $\Scal_+$ and $\Scal_-$ in $\R^2$.]
  {%
    \begin{tikzpicture}[scale=1.2]
      \draw[line width=1.2pt] (1.,-1.)--(0.,0.)--(1.,1.);%
      \draw[dashed, line width=1.2pt] (2.,-2.)--(1.,-1.);%
      \draw[dashed, line width=1.2pt] (1.,1.)--(2.,2.);%
      \draw[dotted, line width=1pt] (-2, 0)--(2, 0);%
      \fill[fill=blue!20!white] (2,-2)--(.0,.0)--(2,2)--cycle;%
      \draw[dotted, line width=1pt] (-2, 0)--(2, 0);%
      \draw[line width=1.pt, ->] (0.,0.)--(.7,.4);%
      \node at (.8, .4) {$r$};%
      \node at (1.3, .2) {$\Scal_-$};%
      \node at (-1., .2) {$\Scal_+$};%
      \node at (1., 1.25) {$\Gcal$};%
      \node at (1., -.2) {$\theta = 0$};%
      \node at (-0.1, .7) {$\theta = \omega/2$};%
      \node at (-0.1, -.8) {$\theta = -\omega/2$};%
    \end{tikzpicture}
    \label{fig:infsectors}
  }
  \caption{Geometries of the considered problems.}
  \label{fig:fig1}
\end{figure}

The  magnetic vector  potential $\potA$  (reduced to  a  single scalar
component in 2d) satisfies the following problem:
\begin{equation*}
  \left\{
    \begin{aligned}
      -\Delta \potA^{+} & = \mu_0 J \text{ in } \Omega_{+},
      \\
      - \Delta \potA^{-} +\ri \kappa\mu_0\sigma \potA^{-} %
      & = 0 \text{ in } \Omega_{-},
      \\
      \potA^+ & = 0 \text{ on } \Gamma,
    \end{aligned}
  \right. \quad
  \begin{aligned}
    \jump{\potA}_\Sigma & = 0, \text{ on } \Sigma,
    \\
    \jump{\partial_{n} \potA}_\Sigma & = 0, \text{ on }
    \Sigma,
  \end{aligned}
\end{equation*}
where $\ri$  is the imaginary  unit, $\mu_0$ the  vacuum permeability,
$\kappa$  the angular  frequency of  the phenomenon  and  $\sigma$ the
electric  conductivity.  Here, $\potA^\pm$  denotes the  restriction of
$\potA$     to     $\Omega_\pm$     and    $\jump{\potA}_\Sigma     :=
\potA^+\lvert_\Sigma  - \potA^-\lvert_\Sigma$ is  the jump  of $\potA$
across   the   interface   $\Sigma$.   Similarly,   $\jump{\partial_{n}
  \potA}_\Sigma$ is the jump of the normal derivative of $\potA$, where
$n$  is the  unit normal  vector to  $\Sigma$. Since  we  consider the
parameters  as  given  in  the   present  purpose,  for  the  sake  of
normalization we introduce the positive parameter $\zeta$ as
$$
\zeta^2=\kappa\mu_0\sigma/{4},
$$
so that the problem satisfied by $\potA$ writes
\begin{equation}
  \label{eq:Edelta}
  \left\{
    \begin{aligned}
      -\Delta \potA^{+} &= \mu_0 J \text{ in } \Omega_{+},
      \\
      - \Delta \potA^{-} + 4 \ri\zeta^2
      \potA^{-} &= 0 \text{ in } \Omega_{-},
      \\
      \potA^+ &=  0 \text{ on } \Gamma,
    \end{aligned}
  \right.\quad
  \begin{aligned}
    \jump{\potA}_\Sigma & = 0, \text{ on } \Sigma,
    \\
    \jump{\partial_{n} \potA}_\Sigma & = 0, \text{ on }
    \Sigma.
  \end{aligned}
\end{equation}
The factor $4$ is present  in order to simplify the forthcoming calculations
(see        equation~\eqref{eq:Deltacomplex})        in       Section
\ref{s:4} and in Appendix-\ref{appendix}.

The variational formulation of problem \eqref{eq:Edelta} is \quad
\[
\left\{
  \begin{aligned}
    &\mbox{\sl Find $\potA\in H^1_0(\Omega)$ such that
      $\forall \potA_* \in H^1_0(\Omega)$,}
    \\
    & \int_{\Omega_+} \nabla\potA^+\cdot\nabla\overline{\potA^+_*} \
    \rd x\rd y %
    + \int_{\Omega_-} \nabla\potA^-\cdot\nabla \overline{\potA^-_*} %
    + 4 \ri\zeta^2 \potA^-\, \overline{\potA^-_*} \ \rd x\rd y %
    = \mu_{0}\int_\Omega J\, \overline{\potA_*}\ \rd x\rd y.
  \end{aligned}
\right.
\]
Since this form is coercive  on $H^1_0(\Omega)$, there exists a unique
solution  $\potA$ in  $H_{0}^1(\Omega)$ to  problem \eqref{eq:Edelta}.
In addition, note that $\potA$ is solution to the Dirichlet problem
\begin{equation}
  \label{eq:EDelta}
  \left\{
    \begin{aligned}
      -\Delta \potA &= F \text{ in } \Omega,
      \\
      \potA &=  0 \text{ on } \Gamma,
    \end{aligned}
  \right.
\end{equation}
where         $F=\mu_0         J\,\indic_{\Omega_+}        -         4
\ri\zeta^2\potA\,\indic_{\Omega_-}$,           the           functions
$\indic_{\Omega_\pm}   $  being   the   characteristic  functions   of
$\Omega_\pm$.   Since  $J$ is  smooth  and  since  $\potA$ belongs  to
$H^1(\Omega)$,  then  $F$  belongs   to  $L^2(\Omega)$,  and  even  to
$H^{\frac12-\varepsilon}(\Omega)$  for any  $\varepsilon>0$,  --- here
the bound on the Sobolev  exponent comes from the discontinuity of $F$
across $\Sigma$. Standard elliptic  regularity of the Laplace operator
implies   that  $\potA$   belongs  to   $H^2(\Omega)$,  and   even  to
$H^{\frac52-\varepsilon}(\Omega)$   for   any   $\varepsilon>0$.    In
particular, $\potA$ belongs to $\mathcal{C}^1(\overline\Omega)$ thanks
to Sobolev imbeddings in two dimensions.

\subsection{Corner asymptotics}

In  general,  the  magnetic   potential  $\potA$  does  not  belong  to
$\mathcal{C}^2(\overline\Omega)$, but the  function $\potA$ possesses a
corner asymptotic  expansion as the distance $r$  to $\mathbf{c}$ goes
to zero.  

In contrast with problems involving only {\em homogeneous}
operators, problem \eqref{eq:Edelta} involves  the lower order
term $4 \ri\zeta^2$ in $\Omega_-$.  As a consequence, according to the
seminal paper of Kondratev~\cite{kondratev1967}, the singularities are
not  homogeneous  functions, but  infinite  sums of  quasi-homogeneous
terms of the general form
\[
r^{\lambda+\ell}\log^n\!r \,\Phi(\theta),%
\quad \lambda\in\C,\ \ell\in\N,\ n\in\N
\]
in  polar coordinates  $(r,\theta)$ centered  at $\mathbf{c}$.  In the
present situation,  the {\em leading  exponents} $\lambda$ can  be made
precise: they are determined by  the principal part of the operator at
$\mathbf{c}$,  which is  nothing but  the Laplacian  $-\Delta$  at the
interior  point  $\mathbf{c}$. Thus,  the  leading  exponents are  {\em
  integers}  $k\in\N$ corresponding to  leading singularities  in the
form of {\em harmonic polynomials}, written in polar coordinates as
\begin{equation}
  \label{eq:kp}
  r^k\cos(k\theta  {-p\pi/2}), \quad k\in\N,\ p\in\{0,1\},
\end{equation}
--- only  $p=0$ is  involved when  $k=0$. This is  why $\potA$  can be
expanded close to the corner as
\begin{align}\label{cornerasympt}
  \potA(r,\theta) \ \equi_{r\rightarrow 0} \ \ 
  \Lambda^{0, 0} \Sfrak^{0, 0} (r,\theta) + 
  \sum_{k\geq 1}  \sum_{p \in \{0, 1\}} \Lambda^{k, p} \Sfrak^{k, p} (r,\theta),
\end{align}
where the terms $\Sfrak^{k, p}$ are the so-called {\em primal singular
  functions},  which belong  to  the formal  kernel  of the  considered
operator in  $\mathbb{R}^2$, and the numbers $\Lambda^{k,  p}$ are the
{\em singular coefficients}.

Therefore, the derivation of the singular functions $\Sfrak^{k,p}$
and the determination of the singular coefficients $\Lambda^{k,p}$ 
are key points in understanding the behavior of the magnetic potential
in the vicinity of the corner. This  is the double aim of this
work.

\subsection{Outline of the paper}

In this  paper, we  provide a constructive  procedure to  determine the
singular functions $\Sfrak^{k, p}$ as well as two different methods to
compute the  singular coefficients $\Lambda^{k,p}$.  Generally speaking,
the singular functions  are sums of the {\em  kernel functions} of the
Laplacian  plus {\em  shadow terms}.  Since  the leading  part of  the
operator  is the  Laplacian in  the whole  plane $\R^2$,  its singular
functions are the harmonic  polynomials \eqref{eq:kp} and we show that
for    any    $(k,p)\in\N\times\{0,1\}$,    the   singular    function
$\Sfrak^{k,p}$ writes as formal series
\begin{equation*}
   \Sfrak^{k,p}(r,\theta) =r^k\sum_{j\geq 0}  (\ri\zeta^2)^j\, r^{2j}
   \sum_{n=0}^j \log^n\!r \, \Phi^{k,p}_{j,n}(\theta),
\end{equation*}
where  the  angular  functions  $\Phi^{k,p}_{j,n}$ are   $\Ccal^1$
functions  of  $\theta$   globally  in  $\R/(2\pi\Z)$  (and  piecewise
analytic).

The  paper is  organized as  follows: In Section  \ref{s:2}, we
consider the  case when  $\zeta=0$ and show  how the theory  of corner
expansions applies to the Taylor  expansion of $\potA$ in the interior
point   $\mathbf{c}$.    For   the   determination   of  the coefficients
$\Lambda^{k,p}$, we introduce  the method of moments and  the method of
dual singular functions, which are both exact methods in this case. In
Section \ref{s:3},  for the case $\zeta\neq0$, we  provide the problem
satisfied by the singular functions $\Sfrak^{k, p}$ in the whole plane
$\R^2$  split into  two  infinite  sectors $\Scal_\pm$  (cf.\
Figure~\ref{fig:infsectors}). In order to determine the coefficients $\Lambda^{k,p}$, we generalize the method of moments, and
we  introduce  the  method  of quasi-dual  singular  functions,  after
\cite{CoDaYo04,  ShYoCoDa12}.   These   methods  are  now  approximate
methods.   We  prove  estimates  for their  convergence.   In  section
\ref{s:4}, we  introduce a  formalism using complex  variables $z_\pm$
 along  the lines of \cite{CostabelDauge93c},  which simplifies
the   analytic   calculations  of   the   primal  singular   functions
$\Sfrak^{k,p}$ and  the dual singular  functions $\Kfrak^{k,p}$, 
 required  by the quasi-dual  function method. 
 These dual singularities are  composed of  the kernel
functions of the Laplacian with negative integer exponents, plus their
shadows, quite  similarly to the  $\Sfrak^{k,p}$.  In this fourth section,
only the first order
shadows of the kernel functions of the Laplacian are given, while Appendix~\ref{appendix}
 provides tools of calculus of the shadows at any order.  In the
 concluding Section~\ref{s:6}, numerical
 simulations by finite elements methods illustrate the theoretical
results. 

\section{Laplace operator}
\label{s:2}

We  first  consider the  case  when  $\zeta=0$,  which means  that  we
consider    the   solution   $\potA$    to   the    Laplace   equation
\eqref{eq:EDelta} with a smooth  right-hand side $F$, whose support is
located outside  the interior point  $\mathbf{c}$, and we look  at the
Taylor expansion of $\potA$ at $\mathbf{c}$. Since $\potA$ is harmonic
in  a  neighborhood of  $\mathbf{c}$,  all  the  terms of  its  Taylor
expansion are  harmonic polynomials.  A basis  of harmonic polynomials
can be written in polar coordinates as
\begin{equation}
  \label{eq:skp}
  \sfrak^{k,p}(r,\theta) = 
  \begin{cases}
    1, & \mbox{if} \ \ k=0 \ \mbox{ and } \ p=0,\\
    r^k\cos(k\theta  {-p\pi/2}) ,   
    &\mbox{if} \ \ k\ge1  \ \mbox{ and } \ p=0,1,
  \end{cases}
\end{equation}
and $\potA$ admits the Taylor expansion
\begin{equation}
  \label{eq:Taylexpansion}
  \potA(r,\theta) \ \equi_{r\rightarrow 0} \ \ 
  \Lambda^{0, 0} \sfrak^{0, 0} (r,\theta) + \ 
  \sum_{k\geq 1}  \sum_{p \in \{0, 1\}} \Lambda^{k, p} \sfrak^{k, p} (r,\theta).
\end{equation}
There  are several  methods to  extract the  coefficients $\Lambda^{k,
  p}$. Besides taking point-wise  values of partial derivatives (which
would be  unstable from the numerical approximation  viewpoint), we can
consider two methods:
\begin{enumerate}
\item  The method  of moments,  which  uses the  orthogonality of  the
  angular   parts  $\cos(k\theta   {-p\pi/2})$   of  the   polynomials
  $\sfrak^{k,p}$.
\item The dual function method, which is the application to the present
  smooth  case  of  a  general  method valid  for  corner  coefficient
  extraction \cite{MazyaPlamenevskii84c}.
\end{enumerate}

Both methods  use non-polynomial  dual harmonic functions  singular at
$r=0$. Let us set
\begin{equation}
  \label{eq:kkp}
  \kfrak^{k,p}(r,\theta) = 
  \begin{cases}
    \di-\frac{1}{2\pi}\, \log r,&\mbox{if} \ \
    k=0,\ p=0, \\[1.5ex]
    \di\frac{1}{2k\pi}\, r^{-k} \cos(k\theta  {-p\pi/2}),&\mbox{if} \ \
    k\ge 1,\ p=0,1.
  \end{cases}
\end{equation}

\subsection{The method of moments}

 To extract  the singular coefficients with the  method of moments,
we introduce the symmetric bilinear form $\Mcal_R$ defined for $R>0$
by 
\begin{equation}
  \label{eq:mom}
  \Mcal_R(K,A) := \frac{1}{R} \int_{r=R} K\,A \ R\, \rd\theta.
\end{equation}
We extract  the scalar  product of  $\potA$ versus  $1$ to
compute   $\Lambda^{0,0}$  and   versus   $\kfrak^{k,p}$  to   compute
$\Lambda^{k,p}$ when $k\ge1$.  

\begin{prop}
  \label{P:mom}
  Let   $\potA$   be   the    solution   to   the   Laplace   equation
  \eqref{eq:EDelta}  with a smooth  right-hand  side $F$  with support
  outside the ball $\Bcal(\mathbf{c},R)$. Then
  \begin{equation}
    \label{eq:momz0}
    \Mcal_R(1,\potA) = 2\pi\Lambda^{0,0}, \quad \mbox{and} \quad
    \Mcal_R(\kfrak^{k,p},\potA) = \frac{1}{2k} \, \Lambda^{k,p}, %
     \qquad \text{for}\ \ k\ge1,\ p=0,1.
  \end{equation}
\end{prop}

\begin{proof}
  The first equality is straightforward since 
  $$
  \Mcal_R(  1,  \sfrak^{l,p}) = 0 , \qquad l \ge1,\ p=0,1.
  $$
  Using the  orthogonality between the  polynomials $\sfrak^{l,q}$ and
  the dual functions $ \kfrak^{k,p}$ we infer
  \begin{equation}
    \label{eq:momks}
    \Mcal_R(  \kfrak^{k,p},  \sfrak^{l,q}) = \frac1{2k} \delta_{l}^k \delta_{q}^p, 
    \qquad k \ge1 ,\ l \ge 0 ,\ p,q =0,1, 
  \end{equation}
  where  $\delta_{i}^j$  denotes the  Kronecker  symbol. Then,  Taylor
  expansion  \eqref{eq:Taylexpansion}   of $\potA$ leads to  the second
  equality \eqref{eq:momz0}.
\end{proof}

\subsection{The method of dual functions}

 Let us  introduce the anti-symmetric bilinear form  $\Jcal_R$ defined over a
circle of radius $R>0$ by
\begin{equation}
  \label{eq:dual}
  \Jcal_R(K,A) := 
  \int_{r=R} \left(K\partial_rA-A\partial_rK\right) \ R \rd\theta .
\end{equation}

\begin{prop}
  Let   $\potA$   be   the    solution   to   the   Laplace   equation
  \eqref{eq:EDelta}  with a smooth  right-hand  side $F$, whose support
  is located outside the ball $\Bcal(\mathbf{c},R)$. Then
  \begin{equation}
    \label{eq:dualz0}
    \Jcal_R(\kfrak^{0,0},\potA) = \Lambda^{0,0}, \quad\mbox{and}\quad
    \Jcal_R(\kfrak^{k,p},\potA) = \Lambda^{k,p}, \qquad \text{for}\ \ k\ge1,\ p=0,1.
  \end{equation}
\end{prop}

\begin{proof}
  Since $ \Jcal_R(\kfrak^{0,0},\sfrak^{k,p})=0 $ for $k\neq 0$, there holds
  $$
  \Jcal_R(\kfrak^{0,0},\potA) = %
  \Jcal_R(\kfrak^{0,0},\Lambda^{0,0})= \Lambda^{0,0} \ .
  $$
  Straightforward computations lead to 
  \begin{equation*}
    \Jcal_R(\kfrak^{k,p},  \sfrak^{l,q}) = 
    l\, \Mcal_R(  \kfrak^{k,p},  \sfrak^{l,q})  
    + k\, \Mcal_R(  \kfrak^{k,p},  \sfrak^{l,q})  , \qquad k ,\ l \ge 1 ,\ p,q =0,1.
  \end{equation*}
  Then, using \eqref{eq:momks} and  \eqref{eq:Taylexpansion}, we obtain  
  $$
  \Jcal_R(\kfrak^{k,p},\potA) = \Lambda^{k,p}, \qquad \text{for}\ \ k\ge1,\ p=0,1\,,
  $$
  which ends the proof.
\end{proof}

\section{The problem with angular conductor}
\label{s:3}

We  now   consider  our  problem  of  interest,   {\sl  i.e.}  problem
\eqref{eq:Edelta} with $\zeta\neq 0$.  Let $\omega\in(0,2\pi)$ be the
opening of  the conductor part  $\Omega_-$ near its corner.  We define
the  origin  of the  angular  variable,  so  that $\theta=0$  cuts  the
conductor   part  by   half  (see   Figure~\ref{fig:infsectors}).  Let
$\Scal_+$ and $\Scal_-$ be the two infinite sectors of $\R^2$ defined as
$$
\Scal_-=\left\{(x,y)=(r\cos\theta,r\sin\theta)\in\R^2:\ \ r>0,\,
\theta\in(-\omega/2,\,\omega/2) \right\}, %
\quad \Scal_+:=\R^2\setminus \overline{\Scal_-}\,.
$$
We  denote  by  $\Gcal$  the common  boundary  of  $\Scal_-$  and
$\Scal_+$:
$$
\Gcal=\partial \Scal_- =\partial \Scal_+ %
=\left\{(x,y)=(r\cos\theta,r\sin\theta)\in\R^2: \ \
  r>0,\,\theta=-\frac{\omega}{2},\, \frac{\omega}{2} \right\}.
$$
Finally, introduce the unit circle $\T=\R/(2\pi\Z)$ and  set
\[
\T_- = (-\omega/2,\,\omega/2) \quad\mbox{and}\quad
\T_+ = \T \setminus \overline{\T_-}\,.
\]
Thus,  the  configuration   $(\Omega_-,\Omega_+)$  coincides  with  the
configuration $(\Scal_-,\Scal_+)$  inside a ball $\Bcal(\mathbf{c},R)$
for a sufficiently small $R$.

The   positive  parameter   $\zeta$   being  chosen,   we  denote   by
$\Lscr_\zeta$ the operator defined on $\R^2$ by
\begin{align}
  \label{Ldelta}
  \Lscr_\zeta (u)=
  \begin{cases}
    -\Delta u,\quad &\text{in $\Scal_+$},\\
    -\Delta u+4\ri\zeta^2 u,\quad &\text{in $\Scal_-$}.
  \end{cases}
\end{align}
acting    on    functions    $u$    such    that    $\partial^q_\theta
u_-|_{\Gcal}=\partial^q_\theta   u_+|_{\Gcal}$   for  $q=0,\,1$.   
The
singularities  of problem  \eqref{eq:Edelta}  are these  of the  model
operator $\Lscr_\zeta$, with continuous Dirichlet and Neumann traces
across $\Gcal$.

As described above, the asymptotics of $\potA$ near the corner involve
the  {\em singularities}  of $\Lscr_\zeta$  which, by  convention, are
formal solutions $\Ufrak$  to the equation $\Lscr_\zeta(\Ufrak)=0$. It
is therefore crucial to make explicit these singularities.

Note that  the operator $\Lscr_\zeta$ is  the sum of  its leading part
$\Lscr_0$  which is  the Laplacian  $-\Delta$  in $\R^2$,  and of  its
secondary part $4\ri\zeta^2\Lscr_1$ where $\Lscr_1$ is the restriction
to  $\Scal_-$.   According  to  the  general
principles    of     Kondratev's    paper~\cite{kondratev1967},    the
singularities $\Ufrak$ of $\Lscr_\zeta=\Lscr_0+4\ri\zeta^2\Lscr_1$ can
be described as formal sums
\begin{equation}
  \label{eq:uj}
  \Ufrak = \sum_j(\ri\zeta^2)^j\ufrak_j,
\end{equation}
where each term $\ufrak_j$ is  derived through an inductive process by
solving recursively
\begin{equation}
  \label{md1}
  \Lscr_0 \ufrak_{0} = 0,\quad 
  \Lscr_0 \ufrak_{1} = -4\Lscr_1 \ufrak_{0},\quad \ldots\,,\quad
  \Lscr_0 \ufrak_{j} = - 4\Lscr_1 \ufrak_{j-1} ,
\end{equation}
in  spaces  of quasi-homogeneous  functions  $\Ssf^\lambda$ defined  for
$\lambda\in\C$ by
\begin{equation}
  \label{eq:Slam}
  \Ssf^\lambda = \Span
  \left\{r^{\lambda} \log^q\!r\,\Phi(\theta),\quad q\in \N,\ \ 
    \Phi\in\Ccal^1(\T),\ \Phi^\pm\in\Ccal^\infty(\overline\T_\pm)\right\}.
\end{equation} 
The  numbers  $\lambda$  are  essentially determined  by  the  leading
equation  $\Lscr_0  \ufrak_{0}  =  0$,  whose  solutions  are  harmonic
functions in $\R^2\setminus\{\mathbf{c}\}$.  The first term $\ufrak_0$
is  the {\em leading  part} of  the singularity  while the  next terms
$\ufrak_j$, for $j\ge1$, are called the {\em shadows}.

The existence of the terms $\ufrak_j$ relies on the following result.
\begin{lem}
  \label{L:inv}
  For $\lambda\in\C$, let $\Ssf^\lambda$ be defined by \eqref{eq:Slam}
  and let $\Tsf^\lambda$ be defined as
  \begin{equation}
    \label{eq:Tlam}
    \Tsf^\lambda = \Span
    \left\{r^{\lambda} \log^q\!r\,\Psi(\theta),\quad q\in \N,\ \ 
      \Psi\in L^2(\T),\ \Psi^\pm\in\Ccal^\infty(\overline\T_\pm)\right\}.
  \end{equation} 
  For  an element  $\gfrak$ of  $\Ssf^\lambda$ or  $\Tsf^\lambda$, the
  degree $\deg\gfrak$ of $\gfrak$ is its degree as polynomial of $\log
  r$.

  Let  $\lambda\in\C$   and  $\ffrak\in\Tsf^{\lambda-2}$.  Then,  there
  exists      $\ufrak\in\Ssf^\lambda$      such     that      $\Lscr_0
  \ufrak=\ffrak$. Moreover
  \begin{itemize}
  \item[{\em (i)}] If $\lambda\not\in\Z$, $\deg\ufrak = \deg\ffrak$,
  \item[{\em  (ii)}] If $\lambda\in\Z\setminus\{0\}$,  $\deg\ufrak \le
    \deg\ffrak+1$,
  \item[{\em (iii)}] If $\lambda=0$, $\deg\ufrak \le \deg\ffrak+2$.
  \end{itemize}
\end{lem}

\begin{proof}
  We follow the technique  of \cite[Ch.\,4]{Dauge88}. We introduce the
  holomorphic  family of  operators $\Mscr$  (also called  {\em Mellin
    symbol}) associated with $\Lscr_0$ in polar coordinates
  \[
  \Lscr_0 = -\Delta = - r^{-2} ( (r\partial_r)^2 + \partial^2_\theta)
  \quad\mbox{and}\quad
  \Mscr(\mu) = -(\mu^2+\partial^2_\theta),\ \theta\in\T.
  \]
  Note the following facts
  \begin{itemize}
  \item[i)] The poles of $\Mscr(\mu)^{-1}$ are the integers,
  \item[ii)]   If   $k$   is   a   nonzero  integer,   the   pole   of
    $\Mscr(\mu)^{-1}$ in $k$ is of order $1$,
  \item[iii)] The pole of $\Mscr(\mu)^{-1}$ in $0$ is of order $2$.
  \end{itemize}
  Let $\ffrak\in\Tsf^{\lambda-2}$. Then
  \[
  \ffrak = \sum_{q=0}^{\deg\ffrak} r^{\lambda-2} \log^q\!r\,\Psi_q(\theta).
  \]
  We have the residue formula\footnote{For a meromorphic function $f$ of
    the complex variable $\mu$ with pole at $\lambda$ and for
  any simple contour
    $\mathcal{C}$ surrounding $\lambda$ and no other pole of $f$, there holds $ \di \Res_{\mu=\lambda} \,
    \left\{ f(\mu) \right\} =\frac{1}{2 i \pi}
    \int_{\mathcal{C}} f(\mu)\, \rd \mu $ .}
  \[
  \ffrak = \Res_{\mu=\lambda} \, \Big\{ r^{\mu-2} \sum_{q=0}^{\deg\ffrak} 
  \frac{q!\, \Psi_q}  {(\mu-\lambda)^{q+1}}\Big\} \, .
  \]
  Setting 
  \[
  \ufrak = \Res_{\mu=\lambda} \, \Big\{ r^\mu \Mscr(\mu)^{-1} \sum_{q=0}^{\deg\ffrak} 
  \frac{q!\, \Psi_q}  {(\mu-\lambda)^{q+1}}\Big\}\, ,
  \]
  we  check  that  $\Lscr_0   \ufrak=\ffrak$  and  the  assertions  on
  $\deg\ufrak$  are consequences of  the above  properties $\ri)$, $\ri\ri  )$ and
  $\ri\ri\ri)$.%
\end{proof}

At  this point, we  distinguish the  {\em primal  singular functions},
which  belong  to  $H^1$   in  any  bounded  neighborhood  $\Bcal$  of
$\mathbf{c}$,  from the {\em  dual singular  functions}, which  do not
belong to $H^1$ in such  a neighborhood. The primal singular functions
$\Sfrak$  appear in  the  expansion  as $r\to0$  of  the solutions  to
problem~\eqref{eq:Edelta}, while  the dual singular functions $\Kfrak$
are  needed  for  the  determination  of  the  coefficients  $\Lambda$
involved in the asymptotics.

\subsection{Principles of the construction of the primal singularities}
\label{subsec:constrprimal}

The  primal singular  functions start  with  (quasi-homogeneous) terms
$\ufrak_0\in  H^1(\Bcal)$  satisfying $\Delta\ufrak_0=0$.   Therefore,
$\ufrak_0$ is a homogeneous harmonic  polynomial. We are looking for a
basis    for    primal    singular    functions,    so    we    choose
$\ufrak_0=\sfrak^{k,p}$    (for     $(k,p)=(0,0)$    and    for    any
$(k,p)\in\N^*\times\{0,1\}$) according  to notation \eqref{eq:skp}. In
order  to   have  coherent  notation,  we   set  $\sfrak_{0}^{k,p}  :=
\sfrak^{k,p}$.  Each $\sfrak_{0}^{k,p}$ belongs to $\Ssf^k$ and is the
leading part of the singular function $\Sfrak^{k,p}$ defined by
\begin{equation}
  \label{eq:Skp}
  \Sfrak^{k,p}(r,\theta)=
  \sum_{j\geq 0}(\ri\zeta^2)^j\sfrak^{k,p}_j(r,\theta).
\end{equation}
For $j\geq1$,  the terms $\sfrak^{k,p}_j$  in the previous  series are
the  {\em shadow terms  of order  $j$}. According  to the  sequence of
problems  \eqref{md1}, the  function $\sfrak^{k,p}_j$  is  searched in
$\Ssf^{k+2j}$ as a particular solution to the following problem:
\begin{align}
  \label{shadowterm}
  \left\{
    \begin{aligned}
      & \Delta \sfrak_{j}^{k,p\,+} = 0, \quad\text{ in } \Scal_+\,,
      \\
      & \Delta \sfrak_{j}^{k,p\,-} = %
      4 \sfrak_{j-1}^{k,p\,-}, \quad\text{
        in } \Scal_-\,.
    \end{aligned}
  \right.
 \intertext{ Belonging to $\Ssf^{k+2j}$ includes the following transmission conditions on $\Gcal$:}
  \label{transmi}
  \left\{
    \begin{aligned}
      & \sfrak_{j}^{k,p\,+}|_{\Gcal}
      = \sfrak_{j}^{k,p\,-}|_{\Gcal}\,,
      \\
      & \partial_\theta \sfrak_{j}^{k,p\,+}|_{\Gcal} =\partial_\theta
      \sfrak_{j}^{k,p\,-}|_{\Gcal}\,.
    \end{aligned}
  \right.
\end{align}

\begin{lem}
  \label{L:skpj}
  Let     $(k,p)=(0,0)$     or    $(k,p)\in\N^*\times\{0,1\}$.     Let
  $\sfrak^{k,p}_0$ be defined as
  \[
  \sfrak^{k,p}_0(r,\theta) = 
  \begin{cases}
    1, & \mbox{if} \ \ k=0 \ \mbox{ and } \ p=0,
    \\
    r^k\cos(k\theta {-p\pi/2}) , &\mbox{if} \ \ k\ge1 \ \mbox{ and } \
    p=0,1\, .
  \end{cases}
  \]
  Then,  for any  $j\ge1$,  there exists  $\sfrak^{k,p}_j\in\Ssf^{k+2j}$
  satisfying       \eqref{shadowterm}-\eqref{transmi}.        Moreover
  $\deg\sfrak^{k,p}_j\le j$. When $p=0$, we choose $\sfrak^{k,p}_j$ as
  an even function\footnote{We understand that $u$ is an even function
    of     $\theta$      if     $u(-\theta)=u(\theta)$     for     all
    $\theta\in\R/(2\pi\Z)$.}     of   $\theta$    and    when   $p=1$,
  $\sfrak^{k,p}_j$ is chosen to be odd.
\end{lem}

\begin{proof}
  We   apply   recursively   Lemma   \ref{L:inv}.    We   start   with
  $\sfrak^{k,p}_0$,   which   belongs   to  $\Ssf^k$   and   satisfies
  $\deg\sfrak^{k,p}_0=0$.   The right-hand side  of \eqref{shadowterm}
  for   $j=1$  is   $4  \sfrak_{0}^{k,p}\,\indic_{\Scal_-}$.    It  is
  discontinuous  across $\Gcal$,  belongs to  $\Tsf^k$, and  is even
  (resp.  odd) for  $p=0$ (resp.  $p=1$). By  Lemma  \ref{L:inv} there
  exists       $\sfrak^{k+2,p}_1\in\Ssf^{k+2}$       solution       to
  \eqref{shadowterm}-\eqref{transmi}        for       $j=1$,       and
  $\deg\sfrak^{k+2}_1\le1$. Since $\Delta$  commutes with the symmetry
  with respect  to $x$ axis,  there exists a particular  solution with
  the  same parity  as the  right-hand side.   The proof  for  any $j$
  follows from a recursion argument.
\end{proof}

These  definitions  being set,  we  can  write  the expansion  of  the
solution to problem \eqref{eq:Edelta} as
\begin{equation}
  \label{eq:singexp}
  \potA(r,\theta) \ \equi_{r\rightarrow 0} \ \ 
  \Lambda^{0, 0} \Sfrak^{0, 0} (r,\theta) + 
  \sum_{k\geq 1}  \sum_{p \in \{0, 1\}} \Lambda^{k, p} \Sfrak^{k, p} (r,\theta).
\end{equation}
Moreover, using the results of Lemma~\ref{L:skpj} and \eqref{eq:Slam}
in \eqref{eq:Skp}, it is straightforward that for any
$(k,p)\in\N^\ast\times\{0,1\}$ or $(k, p) = (0, 0)$, the singular
function $\Sfrak^{k,p}$ writes as formal series
\begin{equation}
   \Sfrak^{k,p}(r,\theta) =r^k\sum_{j\geq 0}  (\ri\zeta^2)^j\, r^{2j}
   \sum_{n=0}^j \log^n\!r \, \Phi^{k,p}_{j,n}(\theta),
   \label{eq:Skp_angular}
\end{equation}
where the angular functions $\Phi^{k,p}_{j,n}$ are $\Ccal^1$ functions
of $\theta$ globally in $\R/(2\pi\Z)$ (and piecewise analytic).

\subsection{Principles of the construction of the dual singularities}
\label{subsec:constrdual}

We start with the dual  singularities of the interior Laplace operator
introduced in \eqref{eq:kkp} .  Setting $\kfrak^{k,p}_0=\kfrak^{k,p}$,
the dual singularities of $\Lscr_\zeta$ are given by series
\begin{equation}
  \label{eq:Kkq}
  \Kfrak^{k,p}(r,\theta) =
  \sum_{j\geq 0}(\ri\zeta^2)^j\kfrak^{k,p}_j(r,\theta),
\end{equation}
where, for  $j\ge1$, the shadow terms $\kfrak^{k,p}_j$  are solutions in
$\Ssf^{-k+2j}$ to
\begin{equation}
  \label{shadowk}
  \left\{
    \begin{aligned}
      & \Delta \kfrak_{j}^{k,p\,+} = 0, \quad\text{ in } \Scal_+\,,
      \\
      & \Delta \kfrak_{j}^{k,p\,-} = %
      4 \kfrak_{j-1}^{k,p\,-}, \quad\text{
        in } \Scal_-\,,
    \end{aligned}
  \right.
\end{equation}
with the same transmission  conditions as ~\eqref{transmi}.  We deduce
from Lemma \ref{L:inv}:
\begin{lem}
  \label{L:kkpj}
  Let     $(k,p)=(0,0)$     or    $(k,p)\in\N^*\times\{0,1\}$.     Let
  $\kfrak^{k,p}_0$ be defined as
  \[
  \kfrak^{k,p}_0(r,\theta) = 
  \begin{cases}
    \di-\frac{1}{2\pi}\, \log r, & \mbox{if} \ \
    k=0,\ p=0,
    \\[1.5ex]
    \di\frac{1}{2k\pi}\, r^{-k} \cos(k\theta  {-p\pi/2}),&\mbox{if} \ \
    k\ge 1,\ p=0,1\,.
  \end{cases}
  \]
  Then,  for any  $j\ge1$, there  exists $\kfrak^{k,p}_j\in\Ssf^{-k+2j}$
  satisfying \eqref{shadowk}.  Moreover, if  $j$ is odd,  or if  $j$ is
  even  and $2j<k$,  then $\deg\kfrak^{k,p}_j\le  j$.  Otherwise ({\sl
    i.e.} if  $j$ is  even  and $2j\ge  k$) then  $\deg\kfrak^{k,p}_j\le
  j+1$.  When $p=0$, we choose $\kfrak^{k,p}_j$ as an even function of
  $\theta$ and when $p=1$, $\kfrak^{k,p}_j$ is chosen to be odd.
\end{lem}

Using the results of Lemma~\ref{L:kkpj} and \eqref{eq:Slam} in
\eqref{eq:Kkq}, it is straightforward that for any
$(k,p)\in\N^\ast\times\{0,1\}$ or $(k, p) = (0, 0)$, the singular
function $\Kfrak^{k,p}$ writes as formal series
\begin{equation}
    \Kfrak^{k,p}(r,\theta) =
    \begin{cases}
      \displaystyle r^{-k}\sum_{j\geq 0} (\ri\zeta^2)^j\, r^{2j}
      \sum_{n=0}^j \log^n\!r \, \Psi^{k,p}_{j,n}(\theta) & \text{($k$
        odd)},\hskip-3mm
      \\
      \displaystyle r^{-k} \bigg( \sum_{0 \leq j < k/2} (\ri\zeta^2)^j\,
        r^{2j} \sum_{n=0}^j \log^n\!r \, \Psi^{k,p}_{j,n}(\theta) +
        \sum_{j\geq k/2} (\ri\zeta^2)^j\, r^{2j} \sum_{n=0}^{j+1}
        \log^n\!r \, \Psi^{k,p}_{j,n}(\theta) \bigg) & \text{($k$
        even)}.\hskip-3mm
    \end{cases}
  \label{eq:Kkp_angular}
\end{equation}
where the angular functions $\Psi^{k,p}_{j,n}$ are $\Ccal^1$ functions
of $\theta$ globally in $\R/(2\pi\Z)$ (and piecewise analytic).

\subsection{A first try to extract the singular  coefficients: the method of moments}

A naive approach to extract the coefficients consists in following the
heuristics of the case when $\zeta=0$,  cf.\ \eqref{eq:mom}-\eqref{eq:momz0}: compute
$\Mcal_R(1/(2\pi),\potA)$   in  order   to  extract   the  coefficient
$\Lambda^{0,0}$    and    $\Mcal_R(2k\kfrak^{k,p},\potA)$   for    the
coefficient  $\Lambda^{k,p}$  of  expansion \eqref{eq:singexp}.   This
method  makes  possible  a  quite  rough approximation  of  the  first
coefficients only, as described below. For a systematic computation of
all the coefficients, the method of quasi-dual functions, described in
the next section is more accurate.

Let us set
\[
M^{k,p}_R(\potA) =
\begin{cases} 
  \frac{1}{2\pi}\, \Mcal_R(1,\potA) \quad& \mbox{if} \ \ k=0 \ \mbox{
    and } \ p=0,
  \\[1ex]
  2k\, \Mcal_R(\kfrak^{k,p}_0,\potA) \quad& \mbox{if} \ \ k\ge1 \
  \mbox{ and } \ p=0,1.
\end{cases} 
\]
We input  the expansion \eqref{eq:singexp}  of $\potA$ in $M^{k,p}_R $
to  obtain the
formal expression
\[
\underline {M\!}_R(\potA) = \underline{\Mcal\!}\ee_R\,\underline \Lambda(\potA),
\]
with $\underline{M\!}_R(\potA)$  and $\underline  \Lambda(\potA)$ the column vectors
$(M^{0,0}_R(\potA),M^{1,0}_R(\potA),M^{1,1}_R(\potA),\ldots)^\top$
and
$(\Lambda^{0,0},\Lambda^{1,0},\Lambda^{1,1},\ldots)^\top$, respectively, and
$\underline{\Mcal\!}\ee_R$ the matrix with coefficients
\[
\Mcal^{k,p\,;\,k',p'}_R = M^{k,p}_R(\Sfrak^{k',p'}).
\]
Hence, we deduce the formal expression of the coefficients $\Lambda^{k,p}$
\[
\underline\Lambda(\potA) = 
\big(\underline{\Mcal\!}\ee_R\big)^{-1}\,\underline {M\!}_R(\potA) .
\]
In order to give a sense to this expression, and to design a method for
the  determination of  the  coefficients $\Lambda^{k,p}$,  we have  to
truncate the matrix $\underline{\Mcal\!}\ee_R$ by taking its first $N$
rows and columns, $N=1,2,\ldots$

Let us denote for shortness
\begin{equation}
  \label{eq:Rzeta}
  \framebox{$R_0 = \zeta R\, (1+\sqrt{|\log R|})$}
\end{equation}
and  assume   that  $R\le1$.    Note  that,  relying   on  Proposition
\ref{P:mom}, Lemma \ref{L:skpj} and formula \eqref{eq:Skp}, we find
\[
\Mcal^{k,p\,;\,k,p}_R \equa_{R\to0} 1 + \Ocal\big(R_0^2\big) 
\quad\mbox{and}\quad
\Mcal^{k,p\,;\,k',p'}_R = 0 \ \ \mbox{if}\ \ p\neq p'.
\]
We also immediately check that
\[
\Mcal^{k,p\,;\,k',p'}_R \equa_{R\to0} \Ocal\big(R^{-k+k'}R_0^2 \big) \quad
\mbox{if}\ \ k\neq k'. 
\]
Let us  emphasize the blow up of this  coefficient as $R\to0$ if $k\ge k'+2$, which
prevents  the   use  of  this  method  to   approximate  the  singular
coefficients at any level of accuracy. This drawback is avoided by the
method of the quasi-dual functions presented in the next subsection.

However, due to its simplicity, the method of  moments can be useful if
a low  order of accuracy of the 
first 3 coefficients  is sufficient.  In the following, we  list what can be  deduced from
the consideration of the lowest values for the cut-off dimension $N$.

\medskip\noindent$\bullet\;$ $N=1$.  We consider the  sole approximate
equation
\[
M^{0,0}_R(\potA) = M^{0,0}_R(\Sfrak^{0,0})\,\Lambda^{0,0} + \Ocal\big(RR_0^2\big).
\]
The  remainder $\Ocal\big(RR_0^2\big)$  corresponds  to the  neglected
terms in the asymptotics of  $\potA$. Using Lemma \ref{L:skpj} for the
structure of $\Sfrak^{0,0}$, we find
\begin{equation}
  \label{eq:L00b}
  \Sfrak^{0,0}(r,\theta)  \equa_{r\to0} 1 
  + \ri (\zeta r)^2\sum_{q=0}^1\log^q\!r\,\Phi^{0,0}_{1,q}(\theta) 
  + \Ocal((\zeta r)^4 \log^2r).
\end{equation}
Therefore we deduce
\begin{itemize}
\item[i)]   Using  only  the   first  term   in  the   asymptotics  of
  $\Sfrak^{0,0}$, we find
  \begin{equation}
    \label{eq:L00_1}
    \Lambda^{0,0} \equa_{R\to0} M^{0,0}_R(\potA) + \Ocal\big(R_0^2\big).
  \end{equation}
\item[ii)]       If        we       compute       $\frac1{2\pi}\int_\T
  \Phi^{0,0}_{1,q}(\theta)\,\rd\theta=:m^{0,0\,;\,0,0}_{1,q}$, we find
  \begin{equation}
    \label{eq:L00_2}
    \Lambda^{0,0} \equa_{R\to0} \ \frac{M^{0,0}_R(\potA)}{1 
      + \ri (\zeta R)^2\sum_{q=0}^1\,m^{0,0\,;\,0,0}_{1,q}\log^q\!R} 
    + \Ocal\big(RR_0^2\big) + \Ocal\big(R_0^4\big).
  \end{equation}
\item[iii)]  If   we  compute  more   terms  in  the   asymptotics  of
  $\Sfrak^{0,0}$, we improve the piece of error $\Ocal\big(R_0^4\big)$,
  but leave the piece $\Ocal\big(RR_0^2\big)$ unchanged.
\end{itemize}
\medskip\noindent  $\bullet\;$ $N=3$.  We consider  the  3 approximate
equations   concerning   $M^{0,0}_R(\potA)$,  $M^{1,0}_R(\potA)$   and
$M^{1,1}_R(\potA)$. For parity reasons, we have a $2\times2$ system
\[
\begin{aligned}
  M^{0,0}_R(\potA) & = M^{0,0}_R(\Sfrak^{0,0})\,\Lambda^{0,0} 
  + M^{0,0}_R(\Sfrak^{1,0})\,\Lambda^{1,0} + 
  \Ocal\big(R^2R_0^2\big),
  \\
  M^{1,0}_R(\potA) & = M^{1,0}_R(\Sfrak^{0,0})\,\Lambda^{0,0} 
  + M^{1,0}_R(\Sfrak^{1,0})\,\Lambda^{1,0} + 
  \Ocal\big(RR_0^2 \big),
\end{aligned}
\]
and an independent equation
\[
M^{1,1}_R(\potA) = M^{1,1}_R(\Sfrak^{1,1})\,\Lambda^{1,1} + \Ocal\big(RR_0^2 \big).
\]
For  this  latter  equation,  we   have  a  similar  analysis  as  for
$\Lambda^{0,0}$ alone ($N=1$): Using  again Lemma \ref{L:skpj} for the
structure of $\Sfrak^{1,1}$, we find
\begin{equation}
  \label{eq:L11}
  \Sfrak^{1,1}  (r,\theta) \equa_{r\to0} r \sin\theta
  + \ri (\zeta r)^2 r \sum_{q=0}^1\log^q\!r\,\Phi^{1,1}_{1,q}(\theta) 
  + \Ocal((\zeta r)^4 r \log^2r),
\end{equation}
and with the first term in the asymptotics of $\Sfrak^{1,1}$, we obtain 
\begin{equation}
  \label{eq:L11_1}
  \Lambda^{1,1} \equa_{R\to0} M^{1,1}_R(\potA) + \Ocal\big(R_0^2\big) , 
\end{equation}
whereas,  using two terms  in the  asymptotics of  $\Sfrak^{1,1}$, and
computing               $\frac1{2\pi}\int_\T              \sin\theta\,
\Phi^{1,1}_{1,q}(\theta)\,\rd\theta=:m^{1,1\,;\,1,1}_{1,q}$, we find
\begin{equation}
  \label{eq:L11_2}
  \Lambda^{1,1} \equa_{R\to0} \ \frac{M^{1,1}_R(\potA)}{1 
    + \ri (\zeta R)^2\sum_{q=0}^1\,m^{1,1\,;\,1,1}_{1,q}\log^q\!R} 
  + \Ocal\big(RR_0^2\big) + \Ocal\big(R_0^4\big).
\end{equation}
For the  $2\times2$ system, according to  the number of  terms 
considered  in  the  asymptotics  \eqref{eq:Skp} of  $\Sfrak^{0,0}$  and
$\Sfrak^{1,0}$ we find the following formulas.
\begin{itemize}
\item[i)] If  we take one  term for $\Sfrak^{0,0}$  in \eqref{eq:L00b}
  and one term for $\Sfrak^{1,0}$ in
  \begin{equation*}
    \Sfrak^{1,0}(r,\theta)  \equa_{r\to0} r\cos \theta   + \Ocal((\zeta r)^2 r \log^2r) ,
  \end{equation*}
  we obtain a diagonal system and
  \[
  \Lambda^{0,0} \equa_{R\to0} M^{0,0}_R(\potA) + \Ocal\big(R_0^2\big)
  \quad\mbox{and}\quad
  \Lambda^{1,0} \equa_{R\to0} M^{1,0}_R(\potA) + \Ocal\big(R^{-1}R_0^2\big).
  \]
\item[ii)]  If  we take  two  terms  for  $\Sfrak^{0,0}$ and  one  for
  $\Sfrak^{1,0}$,  we  obtain  \eqref{eq:L00_1} again,  and  computing
  $m^{1,0\,;\,0,0}_{1,q}$    as    $\frac1{\pi}\int_\T    \cos\theta\,
  \Phi^{0,0}_{1,q}(\theta)\,\rd\theta$, we find
  \[
  \Lambda^{1,0} \equa_{R\to0} \ M^{1,0}_R(\potA) - M^{0,0}_R(\potA)\,
  \frac{\ri R^{-1}(\zeta R)^2\sum_{q=0}^1\,m^{1,0\,;\,0,0}_{1,q}\log^q\!R}{1 
    + \ri (\zeta R)^2\sum_{q=0}^1\,m^{0,0\,;\,0,0}_{1,q}\log^q\!R} 
  + \Ocal\big(RR_0^2\big) + \Ocal\big(R^{-1}R_0^4\big).
  \]
\item[iii)] Taking more terms, we improve some parts of the remainder,
  we  still have  $\Ocal\big(R^2R_0^2\big)$  for $\Lambda^{0,0}$,  and
  $\Ocal\big(RR_0^2 \big)$  for $\Lambda^{1,0}$, but we  always have a
  term containing $R^{-1}$, preventing the good accuracy of the method.
\end{itemize}

\subsection{Extraction of singular coefficients by the method of quasi-dual
  functions}
\label{sec:quasidual}
When $\zeta\neq0$, and in contrast to  the case $\zeta=0$, we do not use 
exact dual functions  satisfying $\Lscr_\zeta\Kfrak=0$ inside the anti-symmetric bilinear form $\Jcal_R$ \eqref{eq:dual}.
Instead we use the {\em quasi-dual}
functions   $\Kfrak^{k,p}_m$, which   are   the truncated    series   of
\eqref{eq:Kkq}:
\begin{equation}
  \label{eq:Kkpm}
  \Kfrak^{k,p}_m(r,\theta) :=
  \sum_{j=0}^m(\ri\zeta^2)^j\kfrak^{k,p}_j(r,\theta).
\end{equation}
 Here, $m$ is a nonnegative integer, which is the order of the quasi-dual function. 
By construction
\begin{equation}
  \label{eq:LzKk}
  \Lscr_\zeta\Kfrak^{k,p}_m = %
  4\ri\zeta^2(\ri\zeta^2)^m\kfrak^{k,p}_m \, \indic_{\Scal_-},
\end{equation}
which is  not zero,  but smaller  and smaller as  $r\to0$ when  $m$ is
increased. The extraction of coefficients $\Lambda^{k,p}$ in expansion
\eqref{eq:singexp} is performed through the evaluation of quantities
\[
\Jcal_R(\Kfrak^{k,p}_m,\potA),\quad k=0,1,2,\ldots
\]
and    corresponding   $p\in\{0,1\}$,    for   suitable    values   of
$m\in\{0,1,\ldots\}$. The quasi-dual function method was introduced in
\cite{CoDaYo04} for straight  edges and developed in \cite{ShYoCoDa12}
for   circular   edges  and   homogeneous   operators  with   constant
coefficients.   The  expansions  considered   there  do   not  contain any
logarithmic terms. Here, we revisit this theory in our framework where,
on the  contrary, we  have an accumulation  of logarithmic  terms. The
main result follows.

\begin{thm}
  \label{T:quasi}
  Let $\potA$ be the  solution to problem \eqref{eq:Edelta}, under the
  assumptions  of the  introduction.  Let  $k\in\N$  and $p\in\{0,1\}$
  ($p=0$  if  $k=0$).  Let  $m$  such  that  $2m+2> k$.  
  For the extraction quantity $\Jcal_R(\Kfrak^{k,p}_m,\potA)$ defined through
  \eqref{eq:dual} and \eqref{eq:Kkpm}, there  exist 
  coefficients $\Jcal^{k,p;k',p'}$ independent of $R$ and $\potA$ such
  that
  \begin{equation}
    \label{eq:JRKm}
    \Jcal_R(\Kfrak^{k,p}_m,\potA) \equa_{R\to0}\ \Lambda^{k,p} 
    + \sum_{\ell=1}^{[k/2]} \Jcal^{k,p\,;\,k-2\ell,p} \Lambda^{k-2\ell,p}
    +\Ocal(R^{-k} R_0^{2m+2} \log R)\,,
  \end{equation}
  where $R_0$ is defined  in \eqref{eq:Rzeta}. If $p=1$, the remainder
  is improved  to $\Ocal(R^{1-k} R_0^{2m+2}  \log R)$. The  extra term
  $\log R$ disappears if $k$ is odd.
\end{thm}

\begin{remark}
  \label{R:quasi}
  The  collection  of  equations  \eqref{eq:JRKm} for  $p\in\{0,1\}$,
  and  $k$ belonging to $\{0,2,\ldots,2L\}$ or to $\{1,3,\ldots,2L+1\}$,
  with  $L\in\N$,  forms a  lower  triangular  system with  invertible
  diagonal.
\end{remark}

\begin{exm}
  \label{ex:quasi}
  {\em i)} \ We have the following simple formulas for $m=0$:\\
  $\bullet$ \ For $k=0$
  \[
  \Lambda^{0,0} \equa_{R\to0}\ \Jcal_R(\Kfrak^{0,0}_0,\potA) 
  + \Ocal(R_0^{2} \log R)\,.
  \]
  $\bullet$ \ For $k=1$
  \[
  \Lambda^{1,0} \equa_{R\to0}\ \Jcal_R(\Kfrak^{1,0}_0,\potA) +
  \Ocal(R^{-1}R_0^{2}) \quad\mbox{and} \quad \Lambda^{1,1}
  \equa_{R\to0}\ \Jcal_R(\Kfrak^{1,1}_0,\potA) + \Ocal(R_0^{2}).
  \]
  If  we know  $m$  more terms  in  the quasi-dual  functions, then 
  $\Ocal(R_0^{2})$ is replaced by $\Ocal(R_0^{2+2m})$ in the
  remainder terms.

  \noindent
  {\em ii)} \ For $k=2$, we need $m\ge1$. We have
  \[
  \Lambda^{2,1}  \equa_{R\to0}\ \Jcal_R(\Kfrak^{2,1}_m,\potA)
  + \Ocal(R^{-1}R_0^{2+2m} \log R),
  \]
  for the odd singularity, and
  \[
  \Lambda^{2,0}  \equa_{R\to0}\ \Jcal_R(\Kfrak^{2,0}_m,\potA) 
  - \Jcal^{2,0\,;\,0,0} \Lambda^{0,0} 
  + \Ocal(R^{-2}R_0^{2+2m} \log R),
  \]
  for the even singularity. An explicit formula for the coefficient
  $\Jcal^{2,0\,;\,0,0}$ is given later on, see \eqref{eq:JR7explicit}, as a particular case of the general formula \eqref{eq:Jcal_angular}.
\end{exm}

\begin{proof}[Proof of Theorem \ref{T:quasi}]
  Introduce the truncated series of singularities \eqref{eq:Skp}:
  \[
  \Sfrak^{k,p}_m(r,\theta) :=
  \sum_{j=0}^m(\ri\zeta^2)^j\sfrak^{k,p}_j(r,\theta).
  \]
  We have
  \begin{equation}
    \label{eq:LzSk}
    \Lscr_\zeta \Sfrak^{k,p}_m = %
    4\ri\zeta^2(\ri\zeta^2)^m\sfrak^{k,p}_m \,\indic_{\Scal_-}.
  \end{equation}
  In  order   to  prove  formula  \eqref{eq:JRKm},   let  us  evaluate
  $\Jcal_R(\Kfrak^{k,p}_m,\Sfrak^{k',p'}_{m'})$ with  any chosen $k'$,
  $p'$      and     $m'\ge      m$.       For     parity      reasons,
  $\Jcal_R(\Kfrak^{k,p}_m,\Sfrak^{k',p'}_{m'})$  is   zero  if  $p\neq
  p'$. Therefore, from now on we take $p'=p$.

  For $\varepsilon\in(0,R)$, denote by $\Ccal(\varepsilon,R)$ be the annulus
  $\{(x,y):\ r\in(\varepsilon,R)\}$. Thanks to Green formula the
  following equality holds:
  \[
  \Jcal_R(\Kfrak^{k,p}_m,\Sfrak^{k',p}_{m'}) 
  - \Jcal_\varepsilon(\Kfrak^{k,p}_m,\Sfrak^{k',p}_{m'}) =
  \int_{\Ccal(\varepsilon,R)} 
  \Big(\Kfrak^{k,p}_m \,\Lscr_\zeta \Sfrak^{k',p}_{m'}
  - \Lscr_\zeta \Kfrak^{k,p}_m\,\Sfrak^{k',p}_{m'}\Big) \, r\rd r\rd\theta,
  \]
 hence, by formulas \eqref{eq:LzKk} and \eqref{eq:LzSk}, we deduce
  \[
  \Jcal_R(\Kfrak^{k,p}_m,\Sfrak^{k',p'}_{m'}) -
  \Jcal_\varepsilon(\Kfrak^{k,p}_m,\Sfrak^{k',p}_{m'}) = 4\ri\zeta^2
  \int_{\Ccal(\varepsilon,R)} \Big(\Kfrak^{k,p}_m \,(\ri\zeta^2)^{m'}
  \sfrak^{k',p}_{m'} - (\ri\zeta^2)^m \kfrak^{k,p}_m\,\Sfrak^{k',p}_{m'}
  \Big) \,\indic_{\Scal_-} \, r\rd r\rd\theta.
  \]
  Using Lemmas \ref{L:skpj} and \ref{L:kkpj}, and taking the condition
  $-k+2m+2>0$ into account, we find
  \[
  \Jcal_R(\Kfrak^{k,p}_m,\Sfrak^{k',p}_{m'}) 
  - \Jcal_\varepsilon(\Kfrak^{k,p}_m,\Sfrak^{k',p}_{m'}) \equa_{R\to0} 
  \Ocal(R^{-k+k'}R_0^{2m'+2}) + \Ocal(R^{-k+k'}R_0^{2m+2}\log R),
  \]
  independently  of $\varepsilon<R$,  ---  the term  $\log  R$ can  be
  omitted if $k$ is odd. Hence, since $m'\ge m$ it yields:
  \begin{align}
    \label{eq:JR1}
    \Jcal_R(\Kfrak^{k,p}_m,\Sfrak^{k',p}_{m'}) -
    \Jcal_\varepsilon(\Kfrak^{k,p}_m,\Sfrak^{k',p}_{m'})
    &\equa_{R\to0} \Ocal(R^{-k+k'}R_0^{2m+2}\log R)
    \\
    \label{eq:JR2}
    &\equa_{R\to0} \zeta^{2m+2}\Ocal(R^{-k+k'+2m+2}\log^{m+2} R).
  \end{align} 
  We use Lemmas \ref{L:skpj} and \ref{L:kkpj} to find the general form
  of     $\Jcal_\rho(\Kfrak^{k,p}_m,\Sfrak^{k',p}_{m'})$    for    any
  $\rho>0$. We obtain
  \begin{align}
    \label{eq:JR3}
    \Jcal_\rho(\Kfrak^{k,p}_m,\Sfrak^{k',p}_{m'}) &= \sum_{j=0}^m
    \sum_{j'=0}^{m'} \sum_{q=0}^{j+1} \sum_{q'=0}^{j'}
    \Jcal^{k,p\,;\,k'}_{j,q,j'\!,q'}\, (\ri\zeta^2)^{j+j'}
    \rho^{-k+k'+2j+2j'} \log^{q+q'}\!\rho\\
    \label{eq:JR4}
    &= \sum_{\un j=0}^{m+m'} \sum_{\un q=0}^{\un j+1}
    \un{\Jcal\!}\,^{k,p\,;\,k'}_{\un j,\un q}\, (\ri\zeta^2)^{\un j}\,
    \rho^{-k+k'+2\un j} \log^{\un q}\rho.
  \end{align} 
 Fix  $R>0$ and  let $\varepsilon$  go to  $0$.  The expression
  \eqref{eq:JR2}                      yields                      that
  $\Jcal_R(\Kfrak^{k,p}_m,\Sfrak^{k',p}_{m'})                         -
  \Jcal_\varepsilon(\Kfrak^{k,p}_m,\Sfrak^{k',p}_{m'})$   is   bounded
  independently  of $\varepsilon<R$,  which means  in  particular that
  $\Jcal_\varepsilon(\Kfrak^{k,p}_m,\Sfrak^{k',p}_{m'})$   is  bounded
  independently of $\varepsilon<R$. Considering \eqref{eq:JR4} for
  $\rho=\varepsilon$, we    deduce    that     the    coefficients
  $\un{\Jcal\!}\,^{k,p\,;\,k'}_{\un  j,\un q}$  attached  to unbounded
  terms  $\varepsilon^{-k+k'+2\un j} \log^{\un  q}\varepsilon$, namely
  such that $-k+k'+2\un  j<0$, or $-k+k'+2\un j=0$ and  $\un q>0$, are
  zero. Hence, setting $\nu=\frac12(k-k')$, we find
  \begin{equation}
    \label{eq:JR5}
    \Jcal_\rho(\Kfrak^{k,p}_m,\Sfrak^{k',p}_{m'}) = 
    \begin{cases}
      (\ri \zeta^2)^\nu \di\un{\Jcal\!}\,^{k,p\,;\,k'}_{\nu,0}
      + \sum_{\un j=\nu+1}^{m+m'} \sum_{\un q=0}^{\un j+1}
      \un{\Jcal\!}\,^{k,p\,;\,k'}_{\un j,\un q}\, (\ri\zeta^2)^{\un j}\,
      \rho^{-k+k'+2\un j} \log^{\un q}\rho
      & \mbox{if $\nu\in\N$,} \\
      \di\sum_{\un j=[\nu]+1}^{m+m'} \sum_{\un q=0}^{\un j+1}
      \un{\Jcal\!}\,^{k,p\,;\,k'}_{\un j,\un q}\, (\ri\zeta^2)^{\un j}\,
      \rho^{-k+k'+2\un j} \log^{\un q}\rho
      & \mbox{if $\nu\not\in\N$} \ .
    \end{cases}
  \end{equation}
  We                             note                             that
  $\lim_{\varepsilon\to0}\Jcal_\varepsilon(\Kfrak^{k,p}_m,\Sfrak^{k',p}_{m'})$
  exists          and          is          $(\ri          \zeta^2)^\nu
  \un{\Jcal\!}\,^{k,p\,;\,k'}_{\nu,0}$   if  $\nu\in\N$  and   $0$  if
  not.  Letting $\varepsilon\to0$ in  \eqref{eq:JR1}-\eqref{eq:JR2} we
  find in all cases
  \[
  \sum_{\un j=[\nu]+1}^{m+m'} \sum_{\un q=0}^{\un j+1}
  \un{\Jcal\!}\,^{k,p\,;\,k'}_{\un j,\un q}\, (\ri\zeta^2)^{\un j}\,
  R^{-k+k'+2\un j} \log^{\un q}R
  \equa_{R\to0} \zeta^{2m+2}\Ocal(R^{-k+k'+2m+2}\log^{m+2} R)\ .
  \]
  Therefore  the coefficients  $\un{\Jcal\!}\,^{k,p\,;\,k'}_{\un j,\un
    q}$  are zero if  $-k+k'+2\un j<-k+k'+2m+2$,  {\sl i.e.}  if $\un
  j<m+1$. Finally
  \begin{equation}
    \label{eq:JR6}
    \Jcal_\rho(\Kfrak^{k,p}_m,\Sfrak^{k',p}_{m'}) = 
    \begin{cases}
      (\ri \zeta^2)^\nu \di\un{\Jcal\!}\,^{k,p\,;\,k'}_{\nu,0}
      + \sum_{\un j=m+1}^{m+m'} \sum_{\un q=0}^{\un j+1}
      \un{\Jcal\!}\,^{k,p\,;\,k'}_{\un j,\un q}\,(\ri\zeta^2)^{\un j}\,
      \rho^{-k+k'+2\un j} \log^{\un q}\rho
      & \mbox{if $\nu\in\N$}, \\
      \di\sum_{\un j=m+1}^{m+m'} \sum_{\un q=0}^{\un j+1}
      \un{\Jcal\!}\,^{k,p\,;\,k'}_{\un j,\un q}\,(\ri\zeta^2)^{\un j}\,
      \rho^{-k+k'+2\un j} \log^{\un q}\rho
      & \mbox{if $\nu\not\in\N$}.
    \end{cases}
  \end{equation}
  Taking the largest term in the remainders of equality \eqref{eq:JR6}, we find
  \begin{equation}
    \label{eq:JR7}
    \Jcal_R(\Kfrak^{k,p}_m,\Sfrak^{k',p}_{m'}) \equa_{R\to0} 
    \begin{cases}
      (\ri \zeta^2)^\nu \un{\Jcal\!}\,^{k,p\,;\,k'}_{\nu,0} +
      \Ocal\big( (\ri\zeta^2)^{m+1} R^{-k+k'+2m+2} \log^{m+2}\!R
      \big)
      & \mbox{if $\nu\in\N$},
      \\[1ex]
      \Ocal\big( (\ri\zeta^2)^{m+1} R^{-k+k'+2m+2} \log^{m+2}\!R
      \big) & \mbox{if $\nu\not\in\N$}.
    \end{cases}
  \end{equation}
  Recall that $\nu=\frac12(k-k')$.  The case $\nu\in\N$ amounts to
  considering $k'=k-2\ell$ with $\ell\le[k/2]$. Setting
  \[
  \Jcal^{k,p\,;\,k-2\ell,p} = %
  (\ri \zeta^2)^\ell \un{\Jcal\!}\,^{k,p\,;\,k-2\ell}_{\ell,0},
  \]
  and $\Jcal^{k,p\,;\,k',p'}=0$ otherwise, we have proved the theorem.
\end{proof}

\begin{prop}
  \label{P:coeffJcal}
  Let $k\in\N\setminus \{0, 1\}$ and $p\in\{0,1\}$.  For any $\ell \in
  1, \ldots, [k/2] $, the coefficients $\Jcal^{k,p\,; k-2\ell,p}$
  introduced in \eqref{eq:JRKm} are given by
  \begin{equation}
    \label{eq:Jcal_angular}
    \Jcal^{k,p\,;\,k-2\ell,p} = 
   (\ri \zeta^2)^{\ell}
    \sum_{j = 0}^{\ell} \int_0^{2 \pi}
    \Psi_{j, 0}^{k, p} (\theta) \left[2 (k-2j)
      \Phi_{\ell-j, 0}^{k-2\ell, p} (\theta)
      + \Phi_{\ell-j, 1}^{k-2\ell, p} (\theta)\right] - \Psi_{j, 1}^{k, p} (\theta)
    \Phi_{\ell-j, 0}^{k-2\ell, p} (\theta) \, d\theta,
  \end{equation}
  with the convention that $\Phi^{k-2\ell, p}_{0, 1} = 0$ and
  $\Psi^{k, p}_{0, 1} = 0$.  
\end{prop}

\begin{proof}
  Recall that $ \Jcal^{k,p\,;\,k-2\ell,p} = (\ri \zeta^2)^\ell
  \un{\Jcal\!}\,^{k,p\,;\,k-2\ell}_{\ell,0}$ where from \eqref{eq:JR7}
  \begin{equation*}
    \Jcal_R\left(\Kfrak^{k,p}_m,\Sfrak^{k-2\ell, p}_{m'} \right) \equa_{R\to0} 
    (\ri \zeta^2)^\ell \un{\Jcal\!}\,^{k,p\,;\,k-2\ell}_{\ell,0} +
    \Ocal\big( (\ri\zeta^2)^{m+1} R^{-2\ell+2m+2} \log^{m+2}\!R
    \big) ,   
  \end{equation*}
  with $m' \geq m \geq l$; moreover
  \begin{equation*}
    \Jcal_R\left(\Kfrak^{k,p}_m,\Sfrak^{k-2\ell, p}_{m'} \right) =
    \int_{r=R} \left(\Kfrak^{k,p}_m \partial_r\Sfrak^{k-2\ell, p}_{m'}
      -\Sfrak^{k-2\ell, p}_{m'}\partial_r\Kfrak^{k,p}_m\right) \ R \rd\theta.
  \end{equation*}
  Comparing \eqref{eq:JR3} and \eqref{eq:JR4}, it is straightforward
  that the useful terms for computing $ \Jcal^{k,p\,;\,k-2\ell,p}$ are
  the terms whose powers in $r$ and in $\log r$ equal $0$ in the
  product $r \Kfrak^{k,p}_m \partial_r\Sfrak^{k-2\ell, p}_{m'}$ and in
  the product $r \Sfrak^{k-2\ell, p}_{m'}\partial_r\Kfrak^{k,p}_m$.

  From \eqref{eq:Skp_angular} and \eqref{eq:Kkp_angular},
  straightforward calculi lead to
  \begin{align}
    \partial_r \Sfrak^{k, p} (r, \theta) = & r^{k-1} \sum_{j \geq 0}
    (\ri \zeta^2)^j r^{2j} %
    \sum_{n = 0}^j \log^n r ((k+2j)\Phi^{k, p}_{j, n} (\theta) +
    (n+1)\Phi^{k, p}_{j, n+1}(\theta)),
    \label{eq:drSkp}
    \\
    \partial_r \Kfrak^{k, p} (r, \theta) = & r^{-k-1} \sum_{j \geq 0}
    (\ri \zeta^2)^j r^{2j} %
    \sum_{n = 0}^{j+1} \log^n r ((-k+2j)\Psi^{k, p}_{j, n} (\theta) +
    (n+1)\Psi^{k, p}_{j, n+1}(\theta)),
    \label{eq:drKkp}
  \end{align}
  with the convention that $\Phi^{k, p}_{j, n} = 0, \, \forall n > j$
  and $\Psi^{k, p}_{j, n} = 0$, if $(n>(j+1))$ or $((n>j) \text{ and }
  k \text{ odd})$ or $((n>j) \text{ and } k \text{ even and }
  (j<k/2))$. From \eqref{eq:Kkp_angular} and \eqref{eq:drSkp},
  the terms of the product $r \Kfrak^{k,p}_m \partial_r\Sfrak^{k-2\ell,
    p}_{m'}$, whose powers in $r$ and $\log r$ equal 0 are given by
  \begin{equation*}
    (\ri \zeta^2)^\ell \sum_{j = 0}^l \int_0^{2 \pi} \Psi_{j, 0}^{k, p} (\theta)
    \left[(k-2j) \Phi_{\ell-j, 0}^{k-2\ell, p} (\theta)
      + \Phi_{\ell-j, 1}^{k-2\ell, p} (\theta)\right] \, d \theta,
  \end{equation*}
  and similarly, from \eqref{eq:Skp_angular} and \eqref{eq:drKkp},   the terms of the product $r \Sfrak^{k-2\ell,
    p}_{m'}\partial_r\Kfrak^{k,p}_m$, whose powers in $r$ and $\log r$ equal 0 are
  \begin{equation*}
    (\ri \zeta^2)^\ell \sum_{j = 0}^l \int_0^{2 \pi} \Phi_{l-j, 0}^{k-2\ell, p} (\theta)
    \left[(-k+2j) \Psi_{j, 0}^{k, p} (\theta)
      + \Psi_{j, 1}^{k, p} (\theta)\right] \, d \theta.
  \end{equation*}  
  Making the difference of both contributions, we obtain
  \eqref{eq:Jcal_angular}.
\end{proof}

\section{Leading singularities and their first shadows in complex variables}
\label{s:4}

According to  the previous section,  the explicit calculation  of both
primal  and  dual  singularities  is  essential to  obtain  a  precise
description   of  the   magnetic  potential   $\potA$   for  numerical
 purposes.  The heuristics  of such  a calculus  have been  described in
subsections~\ref{subsec:constrprimal}  and~\ref{subsec:constrdual}. It
consists  in  deriving  the  shadow   at  any  order  of  the  leading
singularities (primal or dual depending on which kind of singularities
is being considered).  The difficulty  of these calculations lies in the
transmission conditions  (continuity of  the potential $\potA$  and of
its normal  derivative) across  the boundary $\Gcal$,  that have  to be
imposed on  the angular functions $\Phi$ of  the space $\Ssf^\lambda$.
The  use  of  appropriate  complex  variables  like in \cite{CostabelDauge93c}
instead  of  the  polar coordinates  $(r,\theta)$ 
simplifies drastically the  calculations. Indeed, in our case, it will be
sufficient to use 
an Ansatz using two complex variables $z_+$ and $z_-$ (in $\Scal_-$ and $\Scal_+$) involving only integer powers of $z_\pm$,
$\bar{z}_\pm$, $\log z_\pm$, and $\log\bar{z}_\pm$. Thus, it is much less complex
than a general Ansatz using polar coordinates without further information.
The aim of this  section is to
derive  the  first  shadow  terms  of both  primal  and  dual  leading
singularities and to  present a systematic method of  calculus if more
terms are needed.

\subsection{Formalism in complex variables}

Let us present the appropriate formalism in complex variables.
Setting $z=re^{\ri \theta} = x+\ri y$ it is obvious that
\[
\partial_z = \frac12(\partial_x-\ri\partial_y)
\quad\mbox{and}\quad
\partial_{\bar z} = \frac12(\partial_x+\ri\partial_y),
\]
from which we deduce
\begin{align}
  \label{eq:Deltacomplex}
  r\partial_r=z\partial_z+\bar{z}\partial_{\bar{z}},\quad
  \partial_\theta=\ri\left(z\partial_z-\bar{z}\partial_{\bar{z}}\right)
  \quad \text{and} \quad 
  \Delta = 4\partial_z\partial_{\bar z}\,.
\end{align}
The leading terms $\sfrak_{0}^{k,p}=\sfrak^{k,p}$ and $\kfrak^{k,p}_0=\kfrak^{k,p}$
of the primal and dual singularities have a natural expression in complex variable:
Setting for any nonnegative integer $k$:\footnote{For uniformity of presentation we complete the sets of functions $\sfrak_{0}^{k,p}$ and $\kfrak^{k,p}_0$ by the convention that $\sfrak^{0,1}_0=0$ and $\kfrak^{0,1}_0=\theta$.}
$$
   \srm^k_0\big(re^{\ri\theta} \big)=\sfrak^{k,0}_0(r,\theta)+\ri
   \,\sfrak^{k,1}_0(r,\theta)
   \quad\mbox{and}\quad
  \krm^k_0 \big(re^{\ri\theta} \big)=%
  \kfrak^{k,0}_0(r,\theta)-\ri \,\kfrak^{k,1}_0(r,\theta),
$$
we obtain 
\begin{equation}
\label{eq:sk0z}
  \forall k\in\N\quad \srm^k_0(z)=z^k \quad \text{ and }\quad
    \begin{cases}
      \krm^k_0(z)= \dfrac1{2k\pi} z^{-k} &\mbox{if} \quad k\in\N\setminus\{0\},
      \\[1.5ex]
      \krm^0_0(z)= -\dfrac1{2\pi} \log z &\mbox{if} \quad k=0.
    \end{cases}
\end{equation}
The idea is to take advantage of formulas \eqref{eq:Deltacomplex} and \eqref{eq:sk0z} to solve problems \eqref{shadowterm}--\eqref{transmi} and their analogues for dual functions, in order to determine shadow terms.

\paragraph*{Complex variables in $\Scal_-$ and $\Scal_+$.}

Intuitively,  since the  shadows $\sfrak_j^{k,p}$  and $\kfrak_j^{k,p}$
belong   respectively  to   $\Ssf^{k+2j}$   and  $\Ssf^{-k+2j}$,   the
corresponding  function  of  the  complex  variable  $\srm_j^{k}$  and
$\krm_j^{k}$  should  involve terms  in  $\log^q z$, $q\in\N$.   Note
however that  despite $\log(z)$ is  well-defined in $\Scal_-$  with its
branch cut on  $\R^-$, it is not defined  in $\Scal_+$ (which contains
$\R^-$).

To avoid such a problem  we introduce two complex variables to perform
the calculations with the same determination of the complex logarithm.
Since  $\Gcal$  is  the  broken line  $\{z:\,|\arg(z)|=\omega/2\}$  we
define $z_-$ and $z_+$ as
\begin{equation}
\label{eq:z+}
   z_-=z,\quad z_+=-z,
\end{equation}
which  ensures  that  $\log(z_-)$  and $\log(z_+)$  are  well  defined
respectively for $z_-\in\Scal_-$  and $z_+\in\Scal_+$, where $\log$ is
the usual  complex logarithm  with its branch  cut on $\R^-$.  For the
sake  of   simplicity,  we  simply   denote  $z_-$  by  $z$,   and  in
$(r,\theta)$--coordinates we have
\begin{equation}
\label{eq:theta+}
   z=r e^{\ri\theta},\quad z_+=r e^{\ri\theta_+},\quad
   \text{where}\quad\theta_+ =
   \begin{cases}
     \theta-\pi & \mbox{if $\theta\in(0,\pi]$,} \\[0.5ex]
     \theta+\pi & \mbox{if $\theta\in[-\pi,0)$.} 
   \end{cases}
\end{equation}
As  previously  presented,   the  two  primal  leading  singularities,
$\sfrak^{k,0}_0$  and $\sfrak^{k,1}_0$,  are the  respective  real and
imaginary parts of $\srm^k_0(z) = z^k$, for $k\in\N$,
which writes $\srm^k_0=(\chi^k_0,\xi^k_0)$ in the variables $(z,z_+)$ where
\begin{align*}
  & \begin{cases}
    \chi^k_0(z) = z^k, &  \left|\arg z\right|\le \dfrac\omega2,\\[0.5ex]
    \xi^k_0(z_+)    =    (-1)^k(z_+)^k,    &   \left|\arg    z_+\right|\le
    \pi-\dfrac\omega2.
  \end{cases}
\end{align*}

\paragraph{Shadows in complex variables}
For any $k\in\N$ and $j\ge1$, we define the couples $\srm^k_j:=\big(\chi^k_j(z),\xi^k_j(z_+)\big)$ and $\krm^k_j:=\big(\zeta^k_j(z),\eta^k_j(z_+)\big)$ by
\begin{align}
  \label{eq:complexvar}
\begin{cases}
   \ \ \chi^k_j (z)=\sfrak^{k,0}_j(r,\theta)+\ri\,\sfrak^{k,1}_j(r,\theta) ,
   \quad\mbox{and}\quad \ \ 
  \zeta^k_j (z)= \kfrak^{k,0}_j(r,\theta)-\ri \,\kfrak^{k,1}_j(r,\theta),
   &\mbox{if}\ |\theta|\le\dfrac\omega2, 
   \\[1.5ex]
   \xi^k_j(z_+)=\sfrak^{k,0}_j(r,\theta)+\ri \,\sfrak^{k,1}_j(r,\theta) ,
   \quad\mbox{and}\quad
  \zeta^k_j (z_+)= \kfrak^{k,0}_j(r,\theta)-\ri \,\kfrak^{k,1}_j(r,\theta),
   &\mbox{if}\ |\theta|\ge\dfrac\omega2.
\end{cases}
  \end{align}
The interior problem \eqref{shadowterm} translates into $\partial_{z}\partial_{\bar z} \chi^k_j = \chi^k_{j-1}$ in $\Scal_-$ and $\partial_{z_+}\partial_{\bar z_+} \xi^k_j = 0$ in $\Scal_+$, and similarly for \eqref{shadowk}. The transmission conditions \eqref{transmi} write equivalently in polar coordinates 
$\big[r\partial_r\sfrak_{j}^{k,p}\big]_{\Gcal}=0$  and
$\big[\partial_\theta \sfrak_{j}^{k,p}\big]_{\Gcal}=0$ and similarly for $\kfrak_{j}^{k,p}$. As far as primal singularities are concerned, the degree of (quasi) homogeneity of $\sfrak_{j}^{k,p}$ is never $0$ and the transmission conditions are equivalent to 
$\big[\partial_z\srm_{j}^{k}\big]_{\Gcal}=0$  and
$\big[\partial_{\bar z} \srm_{j}^{k}\big]_{\Gcal}=0$.
The computation of the primal shadow terms of order $j\geq 1$ consists now in
finding $(\chi^k_j(z),\xi^k_j(z_+))$ such that
\begin{subequations}
  \label{eq:primalshadow}
  \begin{align}
    \ \partial_{z}\partial_{\bar z} \chi^k_j = \chi^k_{j-1} \quad
    &\text{in $\Scal_-$}\,,
    \label{dzdbarzsh}
    \\
    \ \partial_{z_+}\partial_{\bar z_+} \xi^k_j = 0, \quad &\text{in
      $\Scal_+$}\,,
    \label{dz+dbarz+sh}
    \\
    \ \partial_{z} \chi^k_j + \partial_{z_+} \xi^k_j = 0 \quad
    &\mbox{on $\Gcal$}\,,
    \label{dzdz+sh}
    \\
    \ \partial_{\bar z} \chi^k_j + \partial_{\bar z_+} \xi^k_j = 0
    \quad &\mbox{on $\Gcal$}\,
    \label{dbarzdbarz+sh}.
  \end{align}
\end{subequations}
If  $-k+2j$ is not equal to $0$,  the dual shadow terms of order $j\ge1$ are determined by solutions $(\zeta^k_j(z),\eta^k_j(z_+))$ to problem
\begin{subequations}\label{eq:shadowdual}
  \begin{align}
    \ \partial_{z}\partial_{\bar z} \zeta^k_{j} = \zeta^k_{j-1} \quad
    &\text{in $\Scal_-$}\,,
    \\
    \ \partial_{z_+}\partial_{\bar z_+} \eta^k_{j} = 0,\quad
    &\text{in $\Scal_+$}\,,
    \\
    \ \partial_{z} \zeta^k_{j} + \partial_{z_+} \eta^k_{j} = 0
    \quad &\mbox{on $\Gcal$}\,,
    \\
    \ \partial_{\bar z} \zeta^k_{j} + \partial_{\bar z_+}
    \eta^k_{j} = 0 \quad &\mbox{on $\Gcal$}\,.
  \end{align}  
\end{subequations}
Finally  if $-k+2j$ vanishes,   $(\zeta^k_j(z),\eta^k_j(z_+))$ satisfies
\begin{subequations}
  \label{eq:shadowdualPart}
  \begin{align}
    \ \partial_{z}\partial_{\bar z} \zeta^k_{j} = \zeta^k_{j-1} \quad
    &\text{in $\Scal_-$}\,,
    \\
    \ \partial_{z_+}\partial_{\bar z_+} \eta^k_{j} = 0,\quad
    &\text{in $\Scal_+$}\,,
    \\
    \ z\partial_{z} \zeta^k_{j} -\bar z\partial_{\bar z}
    \zeta^k_{j}- z_+\partial_{z_+} \eta^k_{j} +\bar
    z_+\partial_{\bar z_+} \eta^k_{j}= 0 \quad &\mbox{on $\Gcal$}\,,
    \\
    \ \zeta^k_{j} -\eta^k_{j} = 0 \quad &\mbox{on $\Gcal$}\,.
  \end{align}  
\end{subequations}
Performing the calculations, we will find that the following Ansatz
spaces are the correct ones to solve the above three problems: For
$\lambda\in\Z$ and $(\ell,n)\in\N^2$ we   define  the $\R$--vector
spaces
$\Zscr_{\ell,n}^{\lambda,\pm}$, $ \Zscr^{\lambda}_{\ell,n}$ and $\Sbb^{\lambda}_{\ell,n}$ by
\begin{subequations}
\label{SSlambdaell}
\begin{align}
  \Zscr^{\lambda,\pm}_{\ell,n} & = \Span_\R\left\{\di z_\pm^{\lambda-\mu}
    \bar{z}^\mu_\pm \log^q(z_\pm), \,\text{with} \quad 0\le\mu\leq
    \ell,\,0\leq q\leq n\right\},
  \label{Zscrkpmj}
  \\
  \Sbb^{\lambda,\pm}_{\ell,n}&=\Zscr^{\lambda,\pm}_{\ell,n} 
  + \overline{\Zscr^{\lambda,\pm}_{\ell,n}},
  \label{Zscrkj}
  \\
  \Sbb^{\lambda}_{\ell,n}&=\left\{\di u=(v(z_-),w(z_+)),\,\text{with} \quad
    v\in\Sbb^{\lambda,-}_{\ell,n} \ \mbox{and}\ w\in\Sbb^{\lambda,+}_{\ell,n} \right\},
  \label{Zkj}
\end{align}
\end{subequations}
where $\overline{\Zscr^{\lambda,\pm}_{\ell,n}}$  is the  space of conjugates  $\bar u$
with $u\in{\Zscr}^{\lambda,\pm}_{\ell,n}$. Note that $\Sbb^{\lambda}_{\ell,n}$ is a finite dimensional subspace of the space $\Tsf^\lambda$ introduced in \eqref{eq:Tlam}.

\begin{remark}\label{ZRev}
 We emphasize that the spaces $\Zscr^{\lambda,\pm}_{\ell,n}$
are  $\R$--vector spaces, and not $\C$--vector spaces, which  will be  important  in the  foreseen
calculus.  
\end{remark}
We are going to prove  the following theorem in the next subsections (for $j=1$) and in Appendix~\ref{appendix}
(for $j\ge2$).
\begin{thm}
Let $k\in\N$ and $j\ge1$. Then the primal shadow $\srm_j^k$ in complex variables belongs to $\Sbb^{k+2j}_{j,j}$.
If $-k+2j<0$, the dual shadow $\krm_j^k$ belongs to $\Sbb^{-k+2j}_{j,j}$ and if $-k+2j\ge0$,
$\krm_j^k$ belongs to $\Sbb^{-k+2j}_{j,j+1}$.
The   functions  $\sfrak^{k,p}_j$   and   $\srm_j^k$
(respectively $\kfrak^{k,p}_j$ and $\krm^{k}_j$) are linked by
$$
\begin{cases}
\sfrak_{j}^{k,0}(r,\theta)=\Re\left(\srm_j^k(re^{\ri\theta})\right),\quad
\kfrak_{j}^{k,0}(r,\theta)=\Re\left(\krm_j^k(re^{\ri\theta})\right),
& \mbox{if $k\in\N$} \\[1.ex]
\sfrak_{j}^{k,1}(r,\theta)=\Im\left(\srm_j^k(re^{\ri\theta})\right),\quad
\kfrak_{j}^{k,1}(r,\theta)=-\Im\left(\krm_j^k(re^{\ri\theta})\right)
& \mbox{if $k\in\N\setminus\{0\}$.}
\end{cases}
$$
\end{thm}

\subsection{Calculation of the shadow function generated by $z^k$ and $\log z$}
\label{sec:1rstShadow}

In this section,  we derive the shadow term generated  by $z^k$ in the
space $\Sbb^{k+2}_{1,1}$ for  $k\in\Z\setminus\{-2\}$ and in $\Sbb^{0}_{1,2}$ if $k=-2$.
We will also find the shadow term generated by $\log z$ in $\Sbb^2_{1,2}$. 
The shadows of  these functions are
important  since the  leading term  of  the primal  functions are  the
function  $z^k$ with  $k\in\N$, while  the  leading term  of the  dual
functions   are   $\log z$    and   the   functions   $z^{-k}$   with
$k\in\N\setminus\{0\}$.
\begin{prop}
  \label{prop:kneq-1-2}
  Let    $k\in\Z\setminus\{-2,\,-1\}$.     A    particular    solution
  $u^k=(v^k,w^k)$ to the following problem
  \begin{subequations}
    \label{eq:uk}
    \begin{align}
      \ \partial_{z}\partial_{\bar z} v^k = z^k \quad &\text{in
        $\Scal_-$}\,, \label{dzdbarz}
      \\
      \ \partial_{z_+}\partial_{\bar z_+} w^k = 0,\quad &\text{in
        $\Scal_+$}\,, \label{dz+dbarz+}
      \\
      \ \partial_{z} v^k + \partial_{z_+} w^k= 0 \quad &\mbox{on
        $\Gcal$}\,,\label{dzdz+}
      \\
      \ \partial_{\bar z} v^k + \partial_{\bar z_+} w^k = 0 \quad
      &\mbox{on $\Gcal$}\,\label{dbarzdbarz+} \, ,
    \end{align}
  \end{subequations}
  writes
  \begin{subequations}
    \begin{align}
        v^k(z) &  = \frac{\sin\omega}{\pi(k+2)} \,z^{k+2}\log z 
        + \frac{\sin(k+1)\omega}{\pi(k+1)(k+2)} \,{\bar
          z}^{k+2}\log \bar z
        - \frac{\cos\omega}{k+2} \,z^{k+2} + \frac{z^{k+1}\bar
          z}{k+1}
      \\[0.5ex]
       \hskip-1em (-1)^k w^k(z_+) & = \frac{\sin\omega}{\pi(k+2)}
        \,z_+^{k+2}\log z_+ 
        + \frac{\sin(k+1)\omega}{\pi(k+1)(k+2)}
        \,{\bar z}_+^{k+2}\log \bar z_+
        + \frac{\cos(k+1)\omega}{(k+1)(k+2)}
        \,{\bar z}_+^{k+2}.
    \end{align}
    \label{u10}
  \end{subequations}
\end{prop}

\begin{proof}
  Since the function $z^{k+1}\bar{z}/(k+1)$ satisfies
  $$
  \partial_z\partial_{\bar{z}}\left(z^{k+1}\bar{z}/(k+1)\right)=z^k,
  $$
  we  suppose  that the  solution  $(v,w)=(v^k,w^k)$ to  \eqref{eq:uk}
  writes
  \begin{equation}
    \label{u1}
    \begin{cases} 
      \ v(z) & = a z^{k+2}\log z + a' {\bar z}^{k+2}\log \bar z + b
      z^{k+2} + z^{k+1}\bar z/(k+1),
      \\
      \ (-1)^k w(z_+) & =  a z_+^{k+2}\log z_+ +  a' {\bar
        z}_+^{k+2}\log \bar z_+ +  b'{\bar z}_+^{k+2}.
    \end{cases}
  \end{equation}

  Observe  that equations  \eqref{dzdbarz}  and \eqref{dz+dbarz+}  are
  satisfied by the Ansatz, for  any complex numbers $a$, $b$, $a'$ and
  $b'$.  Jump  conditions will  determine $a$, $b$,  $a'$ and  $b'$ in
  order \eqref{eq:uk} to be satisfied.

  Equations~\eqref{dzdz+}--\eqref{dbarzdbarz+} lead  to
  four conditions for the coefficients  $a$, $b$, $a'$ and $b'$, since
  both   equalities    hold   in   $\arg(z)=\omega/2$    ({\sl   i.e.}
  $\arg(z_+)=\omega/2-\pi$)   and   $\arg(z)=-\omega/2$  ({\sl   i.e.}
  $\arg(z_+)=\pi-\omega/2$).
  For       instance       write       equation~\eqref{dzdz+}       in
  $\arg(z)=\omega/2$. Applying the following equalities
  \begin{align*}
    &\forall z\in \Scal_-, &\quad & \partial_z
    v(z)=a\Bigl((k+2)z^{k+1}\log z +z^{k+1}\Bigr)
    +b(k+2)z^{k+1}+z^k\bar{z},
    \\
    &\forall z_+\in \Scal_+,&\quad&
    \partial_{z_+}
    w(z_+)=(-1)^ka\Bigl((k+2)z_+^{k+1}\log z_+ + z_+^{k+1}\Bigr),
  \end{align*}
  respectively     in     $z=r     e^{\ri\omega/2}$     and     $z_+=r
  e^{\ri\left(\omega/2-\pi\right)}$ and using constraint~\eqref{dzdz+}
  imply
  $$
  \partial_z v|_{\arg(z)=\omega/2}+\partial_{z_+} w|_{\arg(z_+)=\omega/2-\pi}=
  r^{k+1}e^{\ri (k+1)\omega/2}\left(\ri\pi a(k+2)+b(k+2)+e^{-\ri\omega}\right),
  $$
  and    similarly    in    $z=r    e^{-\ri\omega/2}$    and    $z_+=r
  e^{\ri\left(-\omega/2+\pi\right)}$:
  $$
  \partial_z v|_{\arg(z)=-\omega/2}+\partial_{z_+} w|_{\arg(z_+)=-\omega/2+\pi}=
  r^{k+1}e^{-\ri (k+1)\omega/2}\left(-\ri\pi a(k+2)+b(k+2)+e^{\ri\omega}\right).
  $$
  Therefore $\partial_z v+\partial_{z_+} w$ vanishes on $\Gcal$ iff
  \begin{equation*}
    \begin{cases}
      \ \ri\pi a +  b    &= - \re^{-\ri\omega} / (k+2), \\
      \ - \ri\pi a +  b  &= - \re^{\ri\omega} / (k+2), 
    \end{cases}
  \end{equation*}
  hence
  \begin{equation*}
    a = \frac{\sin\omega}{\pi(k+2)} \quad\mbox{and}\quad
    b = -\frac{\cos\omega}{k+2}\,.
  \end{equation*}
  Very similar  computations imply that the jump  $\partial_{\bar z} v
  + \partial_{\bar z_+} w$ vanishes in $\theta=\omega/2$ iff
  \begin{equation*}
    -(k+2) \re^{-\ri(k+1)\omega/2} \ri\pi a' 
    - (k+2) \re^{-\ri(k+1)\omega/2} b' + \re^{\ri(k+1)\omega/2}/(k+1)
    = 0.
  \end{equation*}
 For $\theta=-\omega/2$, we have
  \begin{equation*}
    (k+2) \re^{\ri(k+1)\omega/2} \ri\pi a' 
    - (k+2) \re^{\ri(k+1)\omega/2} b' + \re^{-\ri(k+1)\omega/2}/(k+1)
    = 0.
  \end{equation*}
  Hence,  $\partial_{\bar z}  v +  \partial_{\bar z_+}  w$  vanishes on
  $\Gcal$ iff
  \begin{equation*}
    \begin{cases}
      \ \ri\pi a' + b' &= \re^{\ri(k+1)\omega} /(k+1) (k+2),
      \\
      \ - \ri\pi a' + b' &= \re^{-\ri(k+1)\omega} /(k+1) (k+2),
    \end{cases}
  \end{equation*}
  from which we infer
  \begin{equation*}
    a' = \frac{\sin(k+1)\omega}{\pi(k+1)(k+2)} \quad\mbox{and}\quad
    b' = \frac{\cos(k+1)\omega}{(k+1)(k+2)}\,.
  \end{equation*}
  Therefore, $u^k=(v^k,w^k)$ writes \eqref{u10}.
\end{proof}

 If  $k=-1$,    a particular solution to \eqref{eq:uk}
necessarily involves  the  function  $\bar  z\log  z$ since $\partial_z\partial_{\bar      z}\left(\bar z\log z\right)=z^{-1}$.   We therefore have

\begin{prop}
  \label{prop:k-1}
  For  $k=-1$,  a   particular  solution  $u^{-1}=(v^{-1},w^{-1})$
  to \eqref{eq:uk}
  writes
  \begin{subequations}
    \begin{align}
      \begin{split}
        \ v^{-1}(z) & \di = \frac{\sin\omega}{\pi} \,z\log z + 
        \frac{\omega-\pi}{\pi} \,{\bar z}\log \bar z - 
        {\cos\omega} \,z + {\bar z}\log z, 
      \end{split}
      \\[0.5ex]
      \begin{split}
        \ w^{-1}(z_+) & \di=    -\frac{\sin\omega}{\pi} \,z_+\log z_+ - 
        \frac{\omega}{\pi} \,{\bar z}_+\log \bar z_+ +\bar z_+.   
      \end{split}
    \end{align}
    \label{u10k-1}
  \end{subequations}
\end{prop}

\begin{proof}
Similarly to \eqref{u1}, for $k=-1$  we suppose  that  the
  solution $(v,w)=(v^k,w^k)$ to \eqref{eq:uk}    writes
  \begin{equation*}
    \begin{cases}
      \ v(z) & = a z\log z + a' {\bar z}\log \bar z + b z + \bar z\log
      z,
      \\
      \ w(z_+) & = - a z_+\log z_+ - (a'+1) {\bar z}_+\log {\bar z}_+-
      b'{\bar z}_+.
    \end{cases}
  \end{equation*}
  Very similar computations using the jump conditions \eqref{dzdz+} on
  $\Gcal$ imply necessarily
  \begin{equation*}
    \begin{cases}
      & \ri a\pi +b=-e^{-\ri\omega},
      \\
      & \ri a\pi -b=e^{\ri\omega},
    \end{cases}
  \end{equation*}
  and using \eqref{dbarzdbarz+}, this leads to
  \begin{equation*}
    \begin{cases}
      &\ri a'\pi+b'+1=\ri(\omega-\pi),
      \\
      &\ri a'\pi-b'-1=\ri(\omega-\pi),
    \end{cases}
  \end{equation*}
  hence the proposition.
\end{proof}

The case  $k=-2$ is quite different  since the source  term belongs to
$\Tsf^{-2}$ of  Lemma~\ref{L:inv}, hence the degree of  the shadow (as
polynomial in $\log  (r)$) is bounded by $2$  instead of $1$.  The following proposition
shows that it is actually equal to $2$.

\begin{prop}
  \label{prop:k-2}
  For  $k=-2$,  a   particular  solution  $u^{-2}=(v^{-2},w^{-2})$  to
  \eqref{eq:shadowdualPart} writes
  \begin{subequations}
    \begin{align}
      \begin{split}
        \ v^{-2}(z) & \di = \frac{1}{2\pi}\sin\omega\,\log^2 z +
        \dfrac{1}{2\pi}\sin \omega \,\log^2 \bar z %
        - \frac 1 \pi \left( \sin \omega+ (2\pi-\omega) \cos\omega
        \right) \, \log z -z^{-1}\bar z
      \end{split}
      \\[1.ex]
      \begin{split}
        \ w^{-2}(z_+) & \di=\frac{1}{2\pi}\sin\omega\,\log^2 z_+ +
        \dfrac{1}{2\pi}\sin \omega \,\log^2 \bar z_+ %
        -\frac 1 \pi \left(\sin \omega + (\pi-\omega) \cos\omega
        \right) \, \log z_+
        \\
        & -\cos\omega\, \log \bar z_+ - \cos \omega + (\pi-\omega)
        \sin \omega.
      \end{split}
    \end{align}
  \end{subequations}
\end{prop}

\begin{proof}
  Suppose that the solution $(v,w)=(v^{-2},w^{-2})$ to \eqref{eq:uk}
  for $k=-2$ writes
  \begin{equation*}
    \begin{cases}
      \ v(z) = a\,\log^2 z + a' \,\log^2 \bar z +b\, \log z
      -z^{-1}\bar z
      \\[.5ex]
      \ w(z_+) =a\,\log^2 z_+ + a' \,\log^2 \bar z_+ + \tilde b \log
      z_+ + b'\, \log \bar z_+ + c.
    \end{cases}
  \end{equation*}
  Then
  \begin{align*}
    \begin{cases}
     \ \partial_z v = 2a\,z^{-1}\log z + b\,z^{-1}+z^{-2}{\bar z} ,
      \\[.5ex] 
     \ \partial_{z_+} w = 2a\,z_+^{-1}\log z_+ +{\tilde b}\,z_+^{-1},
    \end{cases}
    \quad 
    \begin{cases} 
     \ \partial_{\bar z} v = 2{a'}\,\bar z^{-1}\log \bar z -z^{-1} ,
      \\[.5ex] 
     \ \partial_{\bar z_+} w = 2{a'}\,\bar z_+^{-1}\log \bar z_+ +{b'}\,{\bar z_+}^{-1} \, .
    \end{cases}
  \end{align*}
  Writing  the  jump  condition  $z\partial_zv  -\bar  z\partial_{\bar
    z}v-z_+\partial_{z_+}w+\bar    z_+\partial_{\bar    z_+}w=0$    on
  $\theta=\pm\omega/2$, and  since $z_+=z e^{\mp\ri\pi}$  we infer the
  system
  \begin{align*}
    \begin{cases}
      \ 2\ri\pi (a+a')+(b-\tilde b+b')=-2e^{-\ri\omega},
      \\
      \ -2\ri \pi(a+a')+(b-\tilde b+b')=-2e^{\ri\omega}.
    \end{cases}
  \end{align*}
  The continuity across $\Gcal$ implies:
  \begin{align*}
    & \pm2\ri\pi (a-a')+(b-\tilde{b}-b')=0,
    \\
    & c-\pi^2 (a+a') + \pi \omega (a+a')
    \mp\ri\pi(\tilde{b}-b^\prime)\mp\ri \frac{\omega}{2} (b+b'-\tilde
    b) +e^{\mp\ri\omega}=0,
  \end{align*}
  hence 
  $$
  a=a'=\frac{1}{2\pi}\sin\omega,\quad b'=-\cos\omega,
  $$
  and then
  $$
  c=-\cos(\omega)+(\pi-\omega)\sin \omega, \quad %
  b=-\frac{1}{\pi}\sin \omega+\frac{\omega-2\pi}{\pi}\cos\omega,
\quad  \tilde b=-\frac{1}{\pi}\sin
  \omega+\frac{\omega-\pi}{\pi}\cos\omega.
  $$
\end{proof}

In order  to obtain  explicit formulas of  any first shadow  terms, it
remains to derive the shadow of $\log z$. The following proposition
can be checked through straightforward calculations. 

\begin{prop}
  \label{prop:k-0}
  A    particular    solution
  $\krm^0_1=(\zeta^0_1,\eta^0_1)$  to  \eqref{eq:shadowdual} with  the
  source  term equal  to $-1/(2\pi)\log  z$  (\textsl{i.e.} $k=0,j=1$)
  writes
    \begin{subequations}
    \begin{align}
      \begin{split}
        \ -2\pi \zeta^{0}_1(z) & \di = 
      \frac{\sin\omega}{4\pi}  \,z^2\log^2 z
      + \frac{\sin\omega}{4\pi} \,\bar z^2\log^2\bar z \\
      &- \frac{\sin\omega+2\pi\cos\omega}{4\pi} \, z^2\log z 
      - \frac{3\sin\omega+2(\pi-\omega)\cos\omega}{4\pi} \, \bar z^2\log \bar z \\
      &- \frac14(\pi\sin\omega-\cos\omega) z^2 
      + \bar z  z\log z  - z\bar z,
      \end{split}
      \\[1.ex]
      \begin{split}
        \ -2\pi \eta^{0}_1(z_+) & \di =\frac{\sin\omega}{4\pi} \,z^2_+\log^2 z_+\,
      +\frac{\sin\omega}{4\pi}\,\bar z^2_+\log^2\bar z_+ \\
      &- \frac{\sin\omega}{4\pi} \, z_+^2\log z_+
       - \frac{3\sin\omega-2\omega \cos\omega}{4\pi} \, \bar z_+^2\log \bar z_+  \\
      &  - \frac 1 4 (3\cos \omega + (2\omega-\pi) \sin \omega)\, \bar z^2_+.
      \end{split}
    \end{align}
    \label{u10k-2}
  \end{subequations}
\end{prop}

\subsection{Explicit expressions of the  first shadows of primal
  singularities.}

To obtain the  expressions of the first shadows  of the primal leading
singularities, we  just have to take  the real and  imaginary parts of
the  functions  $u^k$  given by  Proposition~\ref{prop:kneq-1-2},  for
any nonnegative integer
$k\in\N$:
$$
\sfrak^{k,0}_1=\Re(u^k),\quad \sfrak^{k,1}_1=\Im(u^k).
$$
We recall that $\sfrak^-(r,\theta)$ is defined for $|\theta|\le\dfrac\omega2$ and $\sfrak^+(r,\theta)$ for $|\theta|\ge\dfrac\omega2$. We also recall from \eqref{eq:theta+} that $\theta_+=\theta-\pi\sgn\theta$.

\paragraph{The shadow $\sfrak^{k,0}_1$.}

For any $k\in\N$, the even first order shadow of $z^k$ writes 
\begin{subequations}
  \label{sk0}
  \begin{align}
    \begin{split}
      \sfrak^{k,0\,-}_1(r,\theta)&=
      \frac{(k+1)\sin\omega+\sin(k+1)\omega}{\pi(k+1)(k+2)}\ 
      r^{k+2}\Big(\log r \,\cos(k+2)\theta-\theta\sin(k+2)\theta\Big)\\[.5ex]
      &+r^{k+2}\Big(\frac{\cos k\theta}{k+1}-\frac{\cos \omega\,\cos
          (k+2)\theta}{k+2}\Big),
    \end{split}
    \\[1ex]
    \begin{split}
      \sfrak^{k,0\,+}_1(r,\theta)&=
     \frac{(k+1)\sin\omega+\sin(k+1)\omega}{\pi(k+1)(k+2)}\ 
      r^{k+2}\Big(\log r\,\cos(k+2)\theta
      -\theta_+\sin(k+2)\theta\Big)\\[.5ex]
      &+r^{k+2}\ \frac{\cos(k+1) \omega\, \cos
        (k+2)\theta}{(k+1)(k+2)}.  
    \end{split}
  \end{align}
\end{subequations}

\paragraph{The shadow $\sfrak^{k,1}_1$.}

For any $k\in\N\setminus\{0\}$, the odd first order shadow of $z^k$ writes 
\begin{subequations}
  \label{sk1}
  \begin{align}
    \begin{split}
      \sfrak^{k,1\,-}_1(r,\theta)&=
      \frac{(k+1)\sin\omega-\sin(k+1)\omega}{\pi(k+1)(k+2)}\ 
      r^{k+2}\Big(\log r \,\sin(k+2)\theta+\theta \cos(k+2)\theta\Big)\\[.5ex]
      &+r^{k+2}\Big(\frac{\sin k\theta}{k+1}-\frac{\cos \omega\,\sin
          (k+2)\theta}{k+2}\Big),
    \end{split}
    \\[1ex]
    \begin{split}
      \sfrak^{k,1\,+}_1(r,\theta)&=
     \frac{(k+1)\sin\omega-\sin(k+1)\omega}{\pi(k+1)(k+2)}\ 
      r^{k+2}\Big(\log r\,\sin(k+2)\theta
      +\theta_+\cos(k+2)\theta\Big)\\[.5ex]
      & - r^{k+2}\ \frac{\cos(k+1) \omega\, \sin
        (k+2)\theta}{(k+1)(k+2)}.  
    \end{split}
  \end{align}
\end{subequations}

\subsection{Explicit expressions of the first shadows of dual
  singularities}
\label{S4.4}
The  explicit expressions  of the  first shadows  of the  dual leading
singularities are given for $k\in \N\setminus\{0\}$ by
$$
\kfrak^{k,0}_1=\frac{1}{2k\pi}\Re(u^{-k}),\quad
\kfrak^{k,1}_1=-\frac{1}{2k\pi}\Im(u^{-k}),
$$
where      the       functions      $u^k$      are       given      by
Propositions~\ref{prop:kneq-1-2},  \ref{prop:k-1}  and  \ref{prop:k-2}
and for $k=0$ we have
$$
\kfrak^{0,0}_1= \Re(\krm^0_1),
$$
where $\krm^0_1$ is given by Proposition~\ref{prop:k-0}.

\paragraph{The dual shadow $\kfrak^{k,0}_1$, \ for $k\neq0$.}
\mbox{ }
\\
$\bullet$ \ For $k\geq 3$, the even first order shadow of the dual singularity
$z^{-k}$ writes
\begin{align*}
  \kfrak^{k,0}_1=\frac{1}{2k\pi}\,\sfrak^{-k,0}_1,
\end{align*}
where  by  extension,  we  denote  by  $\sfrak^{-k,0}$  the
formula~\eqref{sk0}, in  which $k$ is replaced by  $-k$.  
\\
$\bullet$ \ For $k=1,2$, we have
\begin{align*}
  2\pi\kfrak^{1,0\,-}_1(r,\theta)&=\frac{\sin\omega +\omega-\pi}{\pi}\,
  r\left(\log r \cos\theta 
    -\theta\sin\theta\right)
-\cos \omega \, r \cos\theta +r\left(\log r \cos\theta +\theta\sin\theta\right),
  \\
  2\pi\kfrak^{1,0\,+}_1(r,\theta)&=\frac{\sin\omega + \omega}{\pi}\,
  r\left(\log r \cos\theta -\theta_+\sin\theta\right)
- r \cos\theta,
  \\[1ex]
  4\pi\kfrak^{2,0\,-}_1(r,\theta)&=\frac{\sin\omega}{\pi}
  \left(\log^2r-\theta^2\right)-\frac 1 \pi (\sin \omega +
  (2\pi-\omega) \cos\omega )\log r-\cos 2\theta,
  \\
  4\pi\kfrak^{2,0\,+}_1(r,\theta)&= \frac{\sin\omega}{\pi}
  \left(\log^2r-\theta_+^2\right)-\frac 1 \pi (\sin
  \omega + (2\pi-\omega) \cos\omega )\log r - \cos \omega + (\pi -
  \omega) \sin \omega.
\end{align*}

\paragraph{The dual shadow $\kfrak^{k,1}_1$, \ for $k\neq0$.}
\mbox{ }
\\
$\bullet$ \ For $k\geq3$, the odd first order shadow of $z^{-k}$ writes
\begin{align*}
  \kfrak^{k,1}_1(r,\theta)=-\frac{1}{2k\pi}\,\sfrak^{-k,1}_1,
\end{align*}
where we denote   by $\sfrak^{-k,1}$
the formula~\eqref{sk1},  in which $k$ is replaced  by $-k$. \\
$\bullet$ \ For $k=1,2$, we have
\begin{align*}
  -2\pi\kfrak^{1,1\,-}_1(r,\theta)&=\frac{\sin\omega - \omega+\pi}{\pi}\,
  r\left(\log r\sin\theta +\theta\cos\theta\right)
 -{\cos \omega \, r \sin\theta}-r\left(\log r \sin\theta
   -\theta\cos\theta\right),
  \\
  -2\pi\kfrak^{1,1\,+}_1(r,\theta)&=\frac{\sin\omega - \omega}{\pi}
  r\left(\log r\sin\theta +\theta_+\cos\theta\right)+r\sin\theta,
  \\[1ex]
  -4\pi\kfrak^{2,1\,-}_1(r,\theta)&=-\frac 1 \pi (\sin \omega +
  (2\pi-\omega) \cos\omega )\,\theta+\sin2\theta,
  \\
  -4\pi \kfrak^{2,1\,+}_1(r,\theta)&=-\frac 1 \pi (\sin \omega -\omega
  \cos\omega )\,\theta_+.
\end{align*}

\paragraph{The shadow $\kfrak^{0,0}_1$.}

\begin{align*}
  - 2 \pi \kfrak^{0,0\,-}_1(r,\theta) & %
  = \frac{\sin\omega}{2\pi} \, r^2\left(\cos2\theta (\log^2
    r-\theta^2)-2\theta\sin2\theta\log r \right) %
  \\
  & +\frac {(\omega-2\pi) \cos \omega -2 \sin \omega} {2\pi}\,
  r^2\left(\cos2\theta\log r-\theta\sin2\theta\right)
  \\
  & +\frac {\cos \omega - \pi \sin \omega} 4\, r^2 \cos2\theta
   +r^2 \log r - r^2,
  \\[1ex]
  - 2 \pi \kfrak^{0,0\,+}_1(r,\theta)& %
  = \frac{\sin\omega}{2\pi} \,r^2\left(\cos2\theta (\log^2
    r-\theta_+^2)-2\theta_+\sin2\theta
    \log r \right)
  \\
  & +\frac {\omega \cos \omega -2 \sin \omega} {2\pi}\,
  r^2\left(\cos2\theta \log
    r-\theta_+\sin2\theta\right)
  \\
  & -\frac {3\cos \omega + (2\omega-\pi) \sin \omega} 4 \,r^2
  \cos2\theta.
\end{align*}

\section{Numerical simulations}
\label{s:6}

In order  to illustrate the calculus of the first shadows, mostly what is
proposed  in   Section~\ref{s:4},  and  the   technique  described  in
subsection~\ref{sec:quasidual} to compute   the  coefficients   of
expansion  \eqref{eq:singexp}, we consider  the  following  problem
(problem  \eqref{eq:Edelta}  with  $J=0$  except  the  non-homogeneous
Dirichlet boundary condition)	
\begin{equation}
  \label{eq:probtest}
  \left\{
    \begin{aligned}
      -\Delta \potA^{+} &= 0 \text{ in } \Omega_{+},
      \\
      - \Delta \potA^{-} + 4 \ri\zeta^2
      \potA^{-} &= 0 \text{ in } \Omega_{-},
      \\
      \potA & =  \frac{\lvert \theta \lvert} {2 \pi }\text{ on } \Gamma,
    \end{aligned}
  \right. \quad
  \begin{aligned}
    \jump{\potA}_\Sigma & = 0, \text{ on } \Sigma,
    \\
    \jump{\partial_{n} \potA}_\Sigma & = 0, \text{ on }
    \Sigma.
  \end{aligned}
\end{equation}
Since the source term is even with respect to $\theta$, the solution $\potA$ of \eqref{eq:probtest} is  $\theta$-even. As a consequence, only the terms with indices $p=0$ are involved in the Kondratev-type expansion  \eqref{eq:singexp}.
The      computational      domain     $\Omega$,     depicted      in
Figure~\ref{fig:geomtest}, is a disk  of radius $50 \milli \meter$. We
consider a conducting  sector for $\omega = \pi/4$  (other values have
been tested and the conclusions  are similar). We particularly focus on the behavior of the solution in the vicinity of the corner $\bf c$. Parameter  $\zeta$  is  equal  to
$1/(5\sqrt{2})  \milli \meter^{-1}$, which  corresponds to  a physical skin depth of $5 \milli  \meter$.  The solution computed by the finite element  method,  using  $P_2$   finite  elements  available in the library \cite{GetDP} and a mesh with $64192$ triangles is  plotted in
Figure~\ref{fig:resulttest}     for      the     real     part     and
Figure~\ref{fig:resulttestim}  for the  imaginary part.
\begin{figure}[ht!]
  \centering%
  \subfigure[Domain $\Omega$.]{
    \begin{tikzpicture}[scale=2.5]
      \draw[line width=1pt] (0., .0) circle (1.cm);%
      \fill[draw=black, fill=blue!20, line width=1pt] %
      (-\angsec/2:1 cm) arc (-\angsec/2:\angsec/2:1 cm) -- (.0,.0) --
      (-\angsec/2:1 cm);%
      \node at (-1., .4) {$\Gamma$};%
      \node at (.0, .5) {$\Omega_+$};%
      \node at (.25, .2) {$\Sigma$};%
      \node at (.8, 0.) {$\Omega_-$};%
      \draw[->] (-\angsec/2:.5 cm) arc(-\angsec/2:\angsec/2:.5 cm);%
      \node at (-.1, .0) {$\mathbf{c}$};%
      \node at (.57, 0.) {$\omega$};%
    \end{tikzpicture}
    \label{fig:geomtest}
  }
  
  \subfigure[Real part of the solution $\Acal$.]{
    \includegraphics[trim = 20mm 0mm 30mm 0mm, clip, width=.44\linewidth]%
    {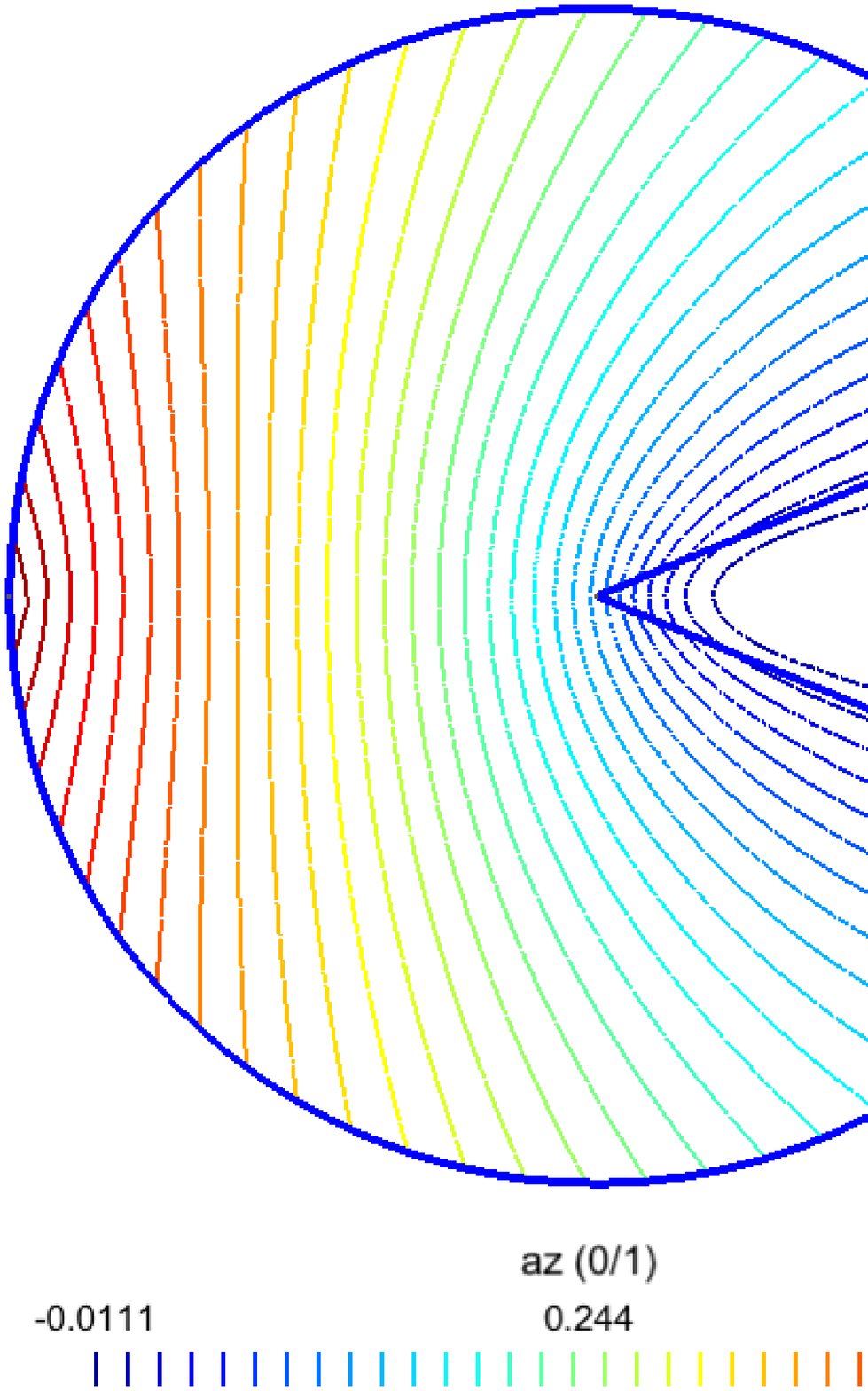}
    \label{fig:resulttest}
  }
  \subfigure[Imaginary part of the solution $\Acal$.]{
    \includegraphics[trim = 20mm 0mm 30mm 0mm, clip, width=.44\linewidth]%
    {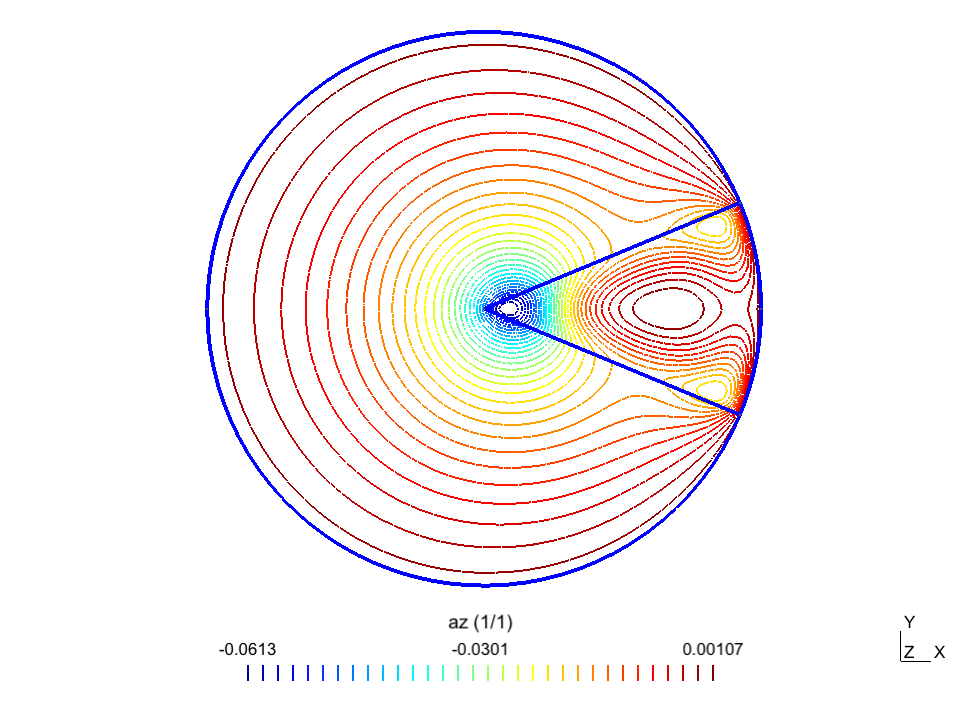}
    \label{fig:resulttestim}
  }
  \caption{Domain $\Omega$ and the computed solution for
    problem~\eqref{eq:probtest}.}
  \label{fig:figtest}
\end{figure}

First, we consider the computation of $\Lambda^{0, 0}$ by the use of
formula \eqref{eq:JRKm} for the case $m = 0$ and $m = 1$. The value of
$\Acal$ computed at the corner $\mathbf c$ is defined as the reference
value for $\Lambda^{0, 0}$; note that it is already a numerically
approximated value.  This reference value is compared to $\Jcal_R
(\Kfrak^{0, 0}_m , \Acal)$ for $m = 0$ and $m = 1$, a quantity which
provides an approximate value of $\Lambda^{0, 0}$ by the method
described in subsection \ref{sec:quasidual}. Note that the convergence
rate as a function of $R$ is related to the first neglected terms.
The results are shown in Figure~\ref{fig:L0} and are consistent with
the theoretical convergence rate (remind $R_0$ defined in \eqref{eq:Rzeta}). Note that concerning $m=1$ and
the smallest values of $R$, we can presume that the discretization
accuracy is attained: This should explain the behavior for the
smallest values of $R$ in Figure~\ref{fig:L0}.

\begin{figure}[!hbtp]
  \centering
  \begin{tikzpicture}
    \begin{loglogaxis}[%
      xlabel=Radius $R$ ($\meter$).,%
      ylabel={$\rm{err}^{0, 0}_m(R)$.},
      grid=both, legend pos = south east] %
      \addplot[thin, mark=x, blue] table[x index=0, y index=3]
      {figures/L000_w45_r2_trace.txt}; %
      \addplot[dashed, blue] table[x index=0, %
      y expr= -ln(x)*(5*sqrt(2)*x*(1+(abs(ln(x)))^0.5))^2*5.5]
      {figures/L000_w45_r2_trace.txt}; %
      \addplot[thin, mark=x, red] table[x index=0, y index=3]
      {figures/L001_w45_r2_trace.txt}; %
      \addplot[dashed, red] table[x index=0, %
      y expr= -ln(x)*(5*sqrt(2)*x*(1+(abs(ln(x)))^0.5))^4*35]
      {figures/L001_w45_r2_trace.txt}; %
      \legend{$m = 0$, $R_0^2 \log (R)$, $m=1$, $R_0^4 \log (R)$}
    \end{loglogaxis}
  \end{tikzpicture}  
  \begin{center}
  \begin{minipage}{0.63\textwidth}\vskip -2ex
  \caption{Accuracy for the computation of $\Lambda^{0, 0}$ as a
    function of $R$.\hfill\break 
    Quantity $\rm{err}^{0, 0}_m(R)$ is the relative
    error $\lvert \Jcal_R (\Kfrak^{0, 0}_m , \Acal) -
    \Acal({\mathbf{c}}) \lvert / \lvert \Acal(\mathbf{c})
    \lvert$.\hfill\break 
    Reference value $\Acal(\mathbf{c})$: $(0.114449904 -\ri \,
    0.0464907336)$.
      \label{fig:L0}
}
\end{minipage}
\end{center}
\end{figure}

Second, we also consider the computation of the coefficients
$\Lambda^{1, 0}$ and $\Lambda^{2, 0}$. The reference values are taken
from the computation of $\Jcal_R$ for the small value $R_{\rm
  small} = 5 \cdot 10^{-5} \meter$ of $R$.  Results for $\Jcal_R
(\Kfrak^{1, 0}_m , \Acal)$, $m = 0$, $1$ are shown in
Figure~\ref{fig:L1}; they are also consistent with the expected
theoretical behaviors.  Results for $\Jcal_R (\Kfrak^{2, 0}_1 ,
\Acal)$ are shown in Figure~\ref{fig:L2}.  In order to deduce from
$\Jcal_R (\Kfrak^{2, 0}_1 , \Acal)$ the approximate value of
$\Lambda^{2, 0}$, we have to use the computed value for $\Lambda^{0,
  0}$ and the coefficient $\Jcal^{2, 0; 0, 0}$. This last coefficient
can be obtained from Proposition \ref{P:coeffJcal}:
\begin{equation}
  \label{eq:JR7explicit}
  \Jcal^{2, 0; 0, 0} = \ri \zeta^2 \int_0^{2 \pi}
  \Psi_{0, 0}^{2, 0} (\theta) \left[4
    \Phi_{1, 0}^{0, 0} (\theta)
    + \Phi_{1, 1}^{0, 0} (\theta) \right] - \Psi_{1, 1}^{2, 0} (\theta)
  \Phi_{0, 0}^{0, 0} (\theta) \, d\theta.
\end{equation}
Angular functions $\Phi_{0, 0}^{0, 0}$, $\Phi_{1, 0}^{0, 0}$ and
$\Phi_{1, 1}^{0, 0}$ can be identified by comparing \eqref{eq:Skp} and
\eqref{eq:Skp_angular} when $k=0$, $p=0$, and $j\in\{0,1\}$
\begin{equation*}
   \sfrak_{j}^{0,0}(r,\theta) =   r^{2j}  \sum_{n=0}^j \log^n\!r \, \Phi^{0,0}_{j,n}(\theta),
\end{equation*}
and using $\sfrak_{0}^{0,0}(r,\theta)=1$ \eqref{eq:skp} and the expression of $\sfrak_{1}^{0,0}$ \eqref{sk0}
\begin{equation*}
  \Phi_{0, 0}^{0, 0} (\theta) = 1,
  \quad
  \Phi_{1, 1}^{0, 0} (\theta) = \frac{\sin\omega}{\pi} \cos 2 \theta,
  \quad
  \Phi_{1, 0}^{0, 0} (\theta) = 
  \begin{cases}
    1 - \dfrac {\cos \omega} 2 \cos 2 \theta - \dfrac {\sin \omega}
    \pi \theta \sin 2 \theta & \mbox{for} \ \  |\theta|\leqslant \dfrac{\omega}{2},
    \\[1.5ex]
    \dfrac {\cos \omega} 2 \cos 2 \theta - \dfrac {\sin \omega} \pi
    \theta_+ \sin 2 \theta  & \mbox{for} \ \  |\theta|\geqslant \dfrac{\omega}{2}  .
  \end{cases}
\end{equation*}
Angular functions $\Psi_{0, 0}^{2, 0}$ and $\Psi_{1, 1}^{2, 0}$ can be
identified by comparing \eqref{eq:Kkq} and \eqref{eq:Kkp_angular} when $k=2$, $p=0$ and
using \eqref{eq:kkp} and $\kfrak_1^{2, 0}$ from Subsection~\ref{S4.4}
\begin{equation*}
  \Psi_{0, 0}^{2, 0} (\theta) = \frac 1 {4 \pi} \cos 2 \theta,
  \quad
  \Psi_{1, 1}^{2, 0} (\theta) = -\frac 1 {4 \pi^2} (\sin \omega + (2 \pi - \omega)
  \cos \omega).
\end{equation*}
Computing integrals, we obtain
\begin{equation*}
  \Jcal^{2, 0; 0, 0} = 
  \ri \zeta^2 \left( \frac{3 \sqrt 2}{4 \pi} + \frac {5 \sqrt 2} 8 \right)
  \approx
  \ri \zeta^2 1.221502.
\end{equation*}

%



\begin{figure}[!hbtp]
  \centering
  \begin{tikzpicture}
    \begin{loglogaxis}[%
      xlabel=Radius $R$  ($\meter$).,%
      ylabel={$\rm{err}^{1, 0}_m(R)$.},
      grid=both, legend pos = south east] %
      \addplot[thin, mark=x, blue] table[x index=0, y index=3]
      {figures/L100_w45_r2_trace.txt}; %
      \addplot[dashed, blue] table[x index=0, %
      y expr= (5*sqrt(2)*x*(1+(abs(ln(x)))^0.5))^2/x*.12]
      {figures/L100_w45_r2_trace.txt}; %
      \addplot[thin, mark=x, red] table[x index=0, y index=3]
      {figures/L101_w45_r2_trace.txt}; %
      \addplot[dashed, red] table[x index=0, %
      y expr= (5*sqrt(2)*x*(1+(abs(ln(x)))^0.5))^4/x*5]
      {figures/L101_w45_r2_trace.txt}; %
      \legend{$m = 0$, $R^{-1} R_0^2$, $m=1$, $R^{-1} R_0^4$}
    \end{loglogaxis}
  \end{tikzpicture}  
  \begin{center}
  \begin{minipage}{0.92\textwidth}\vskip -2ex
  \caption{Accuracy for the computation of $\Lambda^{1, 0}$ as a
    function of $R$.\hfill\break 
    Quantity $\rm{err}^{1, 0}_m(R)$ stands for the relative
    error  
    $\lvert \Jcal_R (\Kfrak^{1, 0}_m , \Acal) - \Jcal_{R_{\rm
        small}} (\Kfrak^{1, 0}_1, \Acal) \lvert / \lvert \Jcal_{R_{\rm
        small}} (\Kfrak^{1, 0}_1, \Acal) \lvert$. \hfill\break 
   Reference value $\Jcal_{R_{\rm small}} (\Kfrak^{1, 0}_1, \Acal)$:
    $(-12.970664-\ri \, 5.40915055)$.
      \label{fig:L1}
}
\end{minipage}
\end{center}
\end{figure}

\begin{figure}[!hbtp]
  \centering
  \begin{tikzpicture}
    \begin{loglogaxis}[%
      xlabel=Radius $R$  ($\meter$).,%
      ylabel={$\rm{err}^{2, 0}_m(R)$.},
      grid=both, legend pos = south east] %
      \addplot[thin, mark=x, red] table[x index=0, y index=3]
      {figures/L201_w45_r2_trace.txt}; %
      \addplot[thin, mark=x, blue] table[x index=0, y index=1]
      {figures/L201_w45_r2_trace.txt}; %
      \addplot[thin, mark=x, green] table[x index=0, y index=2]
      {figures/L201_w45_r2_trace.txt}; %
      \addplot[dashed, red] table[x index=0, %
      y expr= -ln(x)*(5*sqrt(2)*x*(1+(abs(ln(x)))^0.5))^4/x^2/100]
      {figures/L201_w45_r2_trace.txt}; %
      \legend{$m=1$, Real part, Imaginary part, $R^{-2} R_0^4
        \log(R)$}
    \end{loglogaxis}
  \end{tikzpicture}
  \begin{center}
  \begin{minipage}{0.8\textwidth}\vskip -2ex
    \caption{Accuracy for the computation of $\Lambda^{2, 0}$ as a
      function of $R$.\hfill\break Quantity $\rm{err}^{2, 0}_m(R)$ stands for
      the relative error of the real part, of the imaginary part and
      of the modulus of $(\Jcal_R (\Kfrak^{2, 0}_1 , \Acal) -
      \Jcal_{R_{\rm small}} (\Kfrak^{2, 0}_1 , \Acal))$.  \hfill\break
      Reference value $\Jcal_{R_{\rm small}} (\Kfrak^{2, 0}_1 ,
      \Acal)$: $(1406.54919 + \ri \, 4599.19999)$.}
\end{minipage}
\end{center}
  \label{fig:L2}
\end{figure}

\clearpage
In  Figure~\ref{fig:visualcomp},  we  perform  a  qualitative  description
of  the isovalues  of $\Acal$  close  to the  corner comparing  the
finite  element solution and, successively,
\begin{itemize}
\item expansion \eqref{eq:singexp}  restricted to
a composite order 1, {\sl i.e.}
\begin{equation*}
  \Jcal_{R_{\rm small}}(\Kfrak^{0, 0}_1, \mathcal A)   + \Jcal_{R_{\rm small}}(\Kfrak^{1, 0}_1, \mathcal A) \sfrak^{1, 0}_0,
\end{equation*}

\item expansion \eqref{eq:singexp} restricted to a composite order 2, {\sl i.e.}  
\begin{equation*}
  \Jcal_{R_{\rm small}}(\Kfrak^{0, 0}_1, \mathcal A) (1+\ri \zeta^2 \sfrak^{0, 0}_1)
  + \Jcal_{R_{\rm small}}(\Kfrak^{1, 0}_1, \mathcal A) \sfrak^{1, 0}_0
  + \left(\Jcal_{R_{\rm small}}(\Kfrak^{2, 0}_1, \mathcal A) %
  - \Jcal^{2, 0; 0, 0}\Jcal_{R_{\rm small}}(\Kfrak^{0, 0}_1, \mathcal A)\right)
  \sfrak^{2, 0}_0,
\end{equation*}

\item expansion \eqref{eq:singexp} restricted to a composite order 3, {\sl i.e.}  adding to the  expression above the term
\begin{equation*}
  \left(\Jcal_{R_{\rm small}}(\Kfrak^{3, 0}_1, \mathcal A) %
  - \Jcal^{3, 0; 1, 0}\Jcal_{R_{\rm small}}(\Kfrak^{1, 0}_1, \mathcal A)\right)
  \sfrak^{3, 0}_0,
\end{equation*}
and replacing $\sfrak^{1, 0}_0$ by $\sfrak^{1, 0}_0 + \ri \zeta^2 \sfrak^{1, 0}_1$.
Computing integrals, we obtain
\begin{equation*}
  \Jcal^{3, 0; 1, 0}  \approx
  \ri \zeta^2 1.522117 \cdot 10^{-9} \ .
\end{equation*}
The reference value for $\Jcal_{R_{\rm small}}(\Kfrak^{3, 0}_1, \mathcal A) $ is $(93037.6253 - \ri \,  154720.669)$.  
\end{itemize}

For the composite order 1, only the constant and  linear terms with respect to $r$ are collected in the Kondratev-type expansion. For the composite order 2, the terms which behave as $r^2$ and $r^2 \log r$ are added. 
Adding then the terms which behave as $r^3$ and $r^3 \log r$  leads to  the composite order 3.

Both    on   the    real  and imaginary parts,  we  observe in Figure~\ref{fig:visualcomp}  as
expected  that the  increase  of  the order  enables  to increase  the
accuracy.

 \begin{figure}[!hbtp]
  \centering
  \subfigure[Expansion restricted to composite order 1. Real part.]{%
    \includegraphics[width=0.48\linewidth]{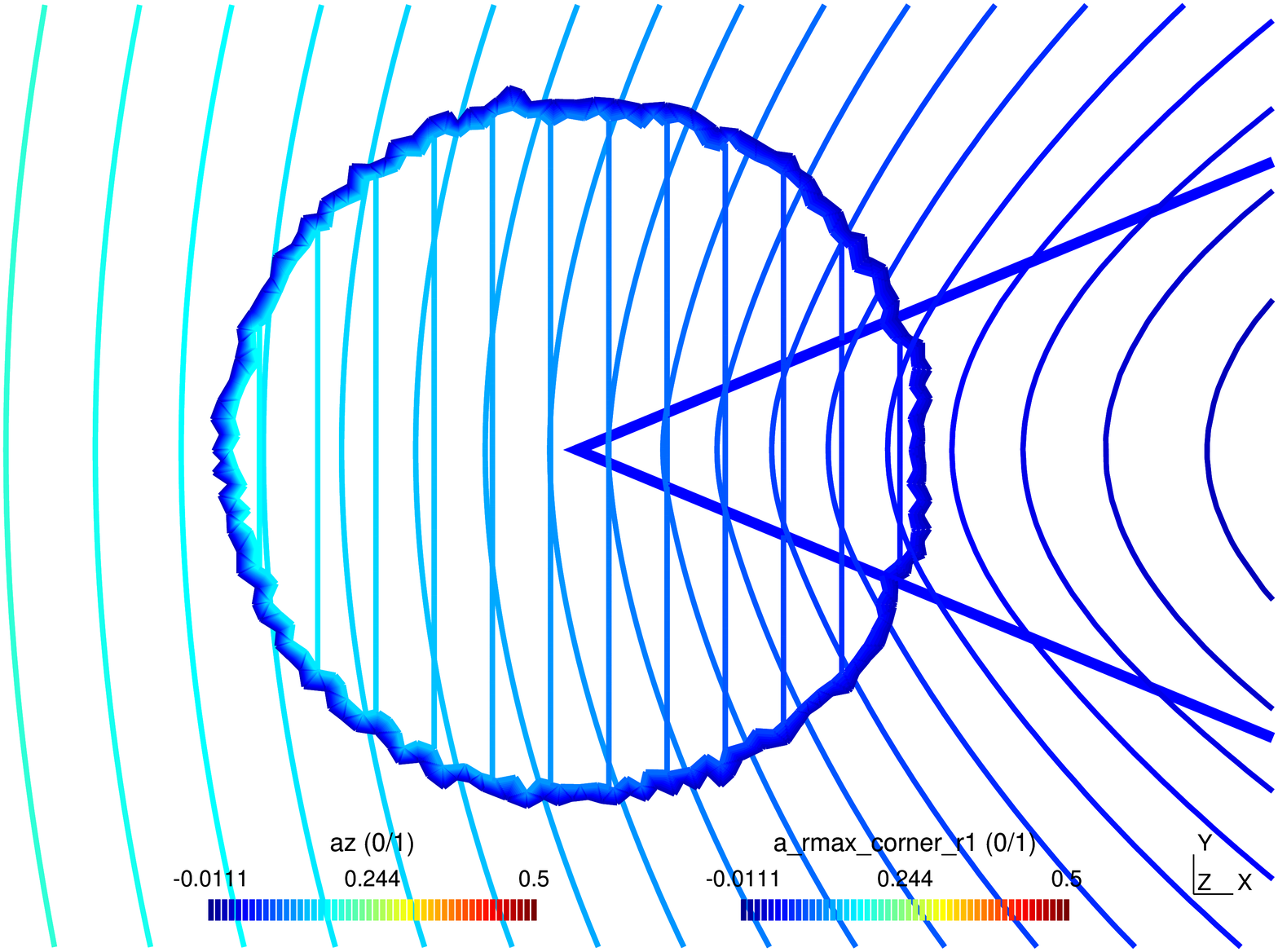}
  }%
   \hfill
   \subfigure[Expansion restricted to composite order 1. Imaginary part.]{%
    \includegraphics[width=0.48\linewidth]{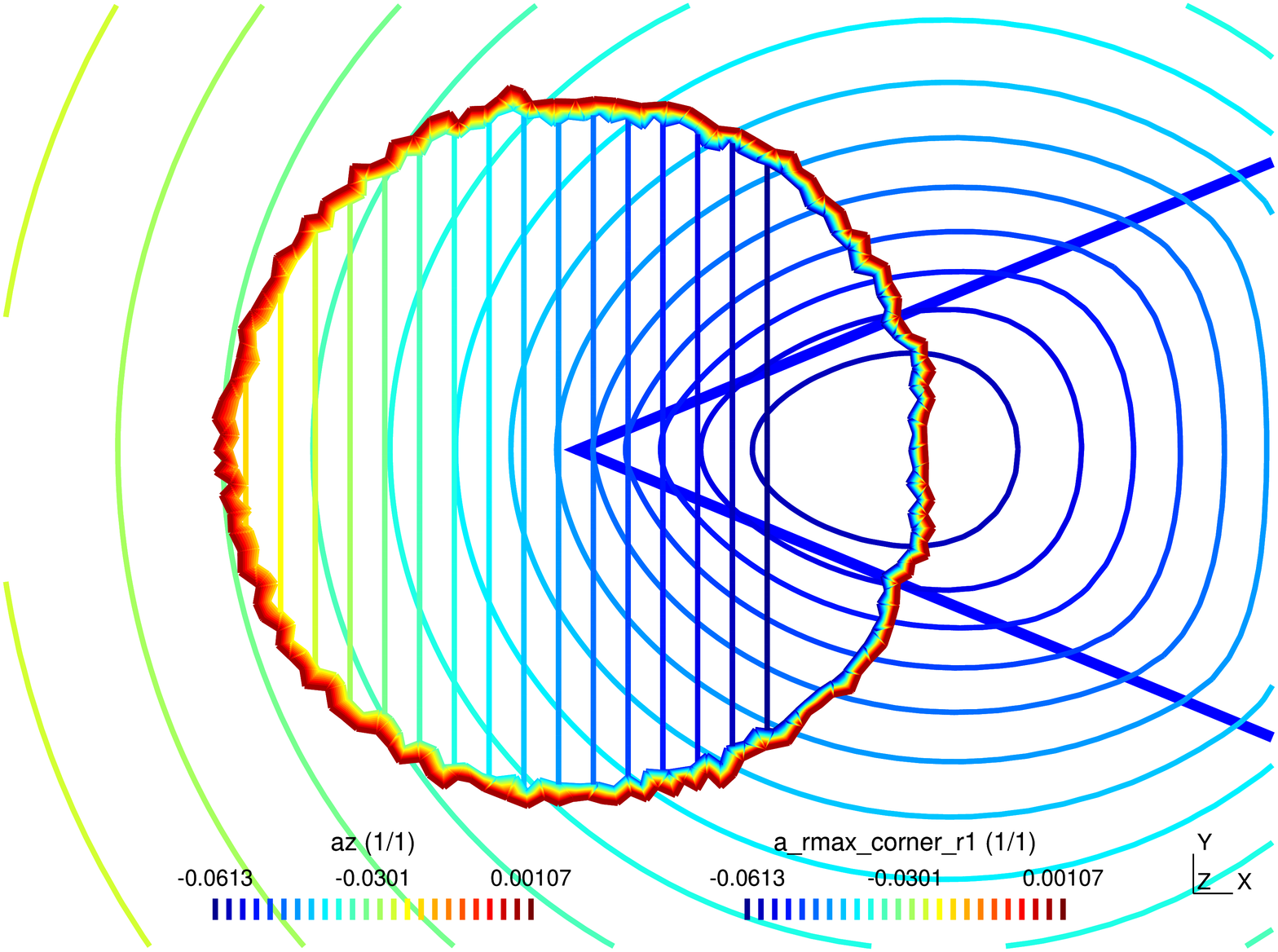}
  }%
  
  \subfigure[Expansion restricted to composite order 2. Real part.]{%
    \includegraphics[width=0.48\linewidth]{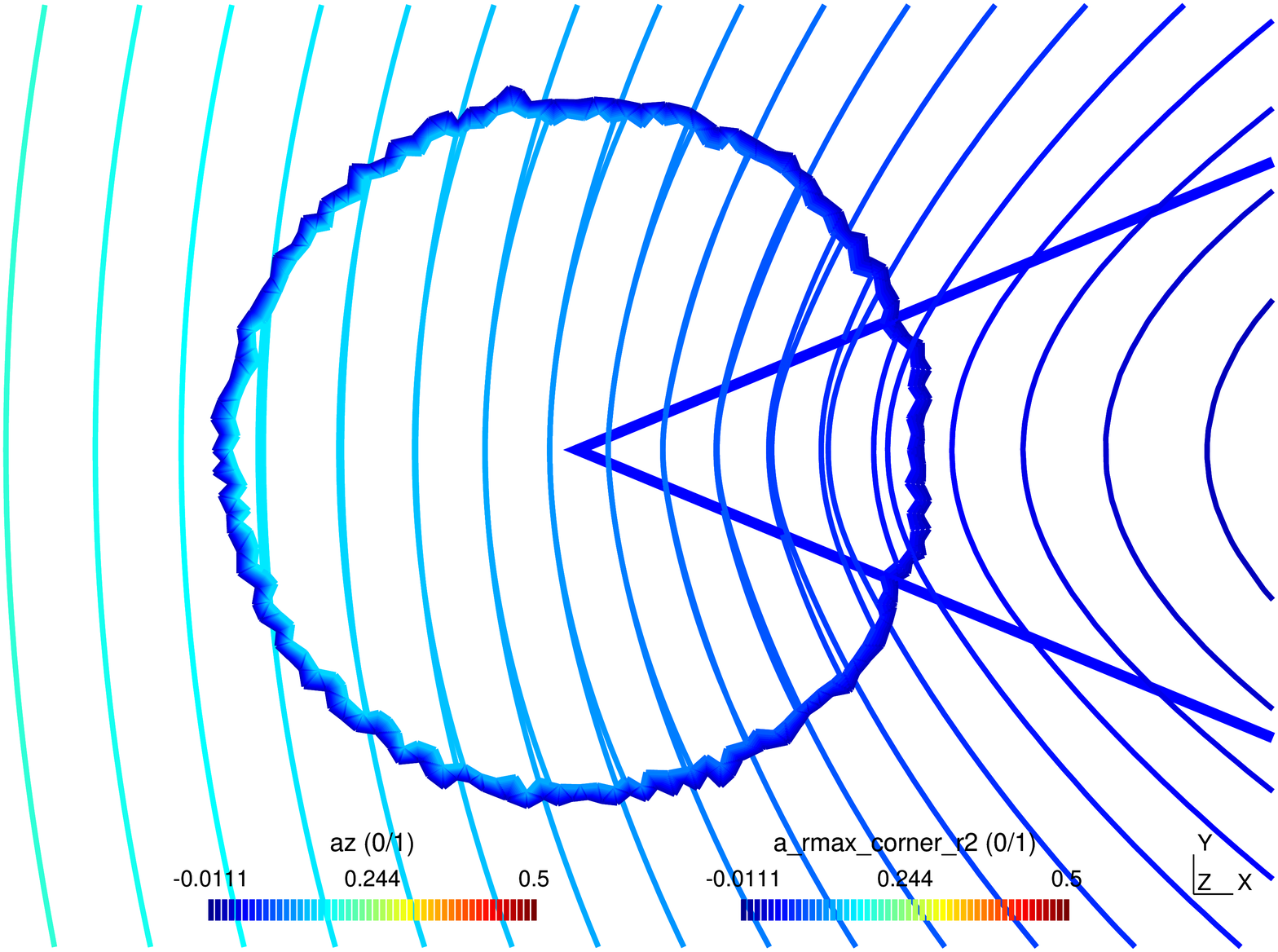}
  }
\hfill
   \subfigure[Expansion restricted to composite order 2. Imaginary part.]{%
    \includegraphics[width=0.48\linewidth]{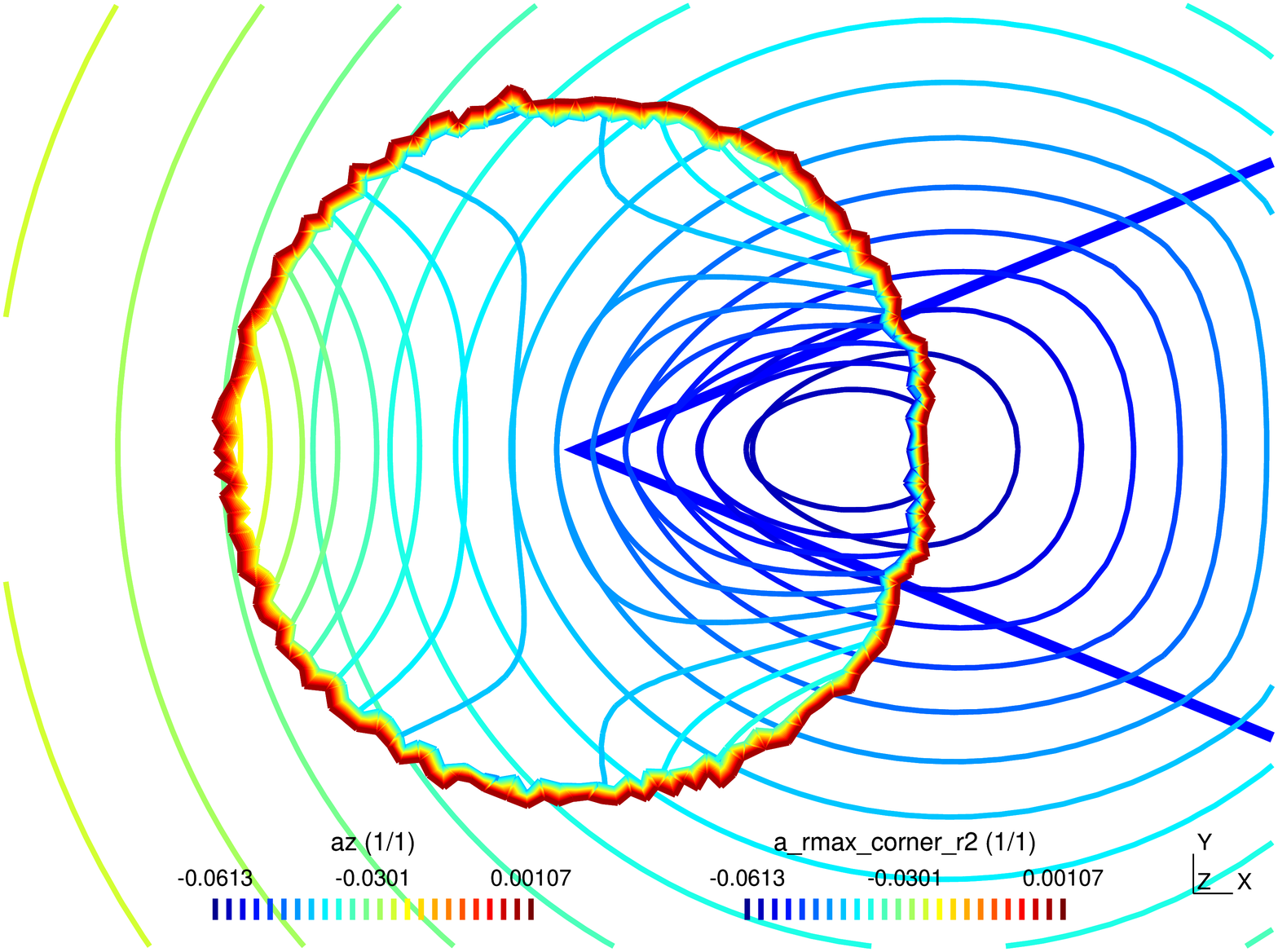}
  }  
    
  \subfigure[Expansion restricted to composite order 3. Real part.]{%
    \includegraphics[width=0.48\linewidth]{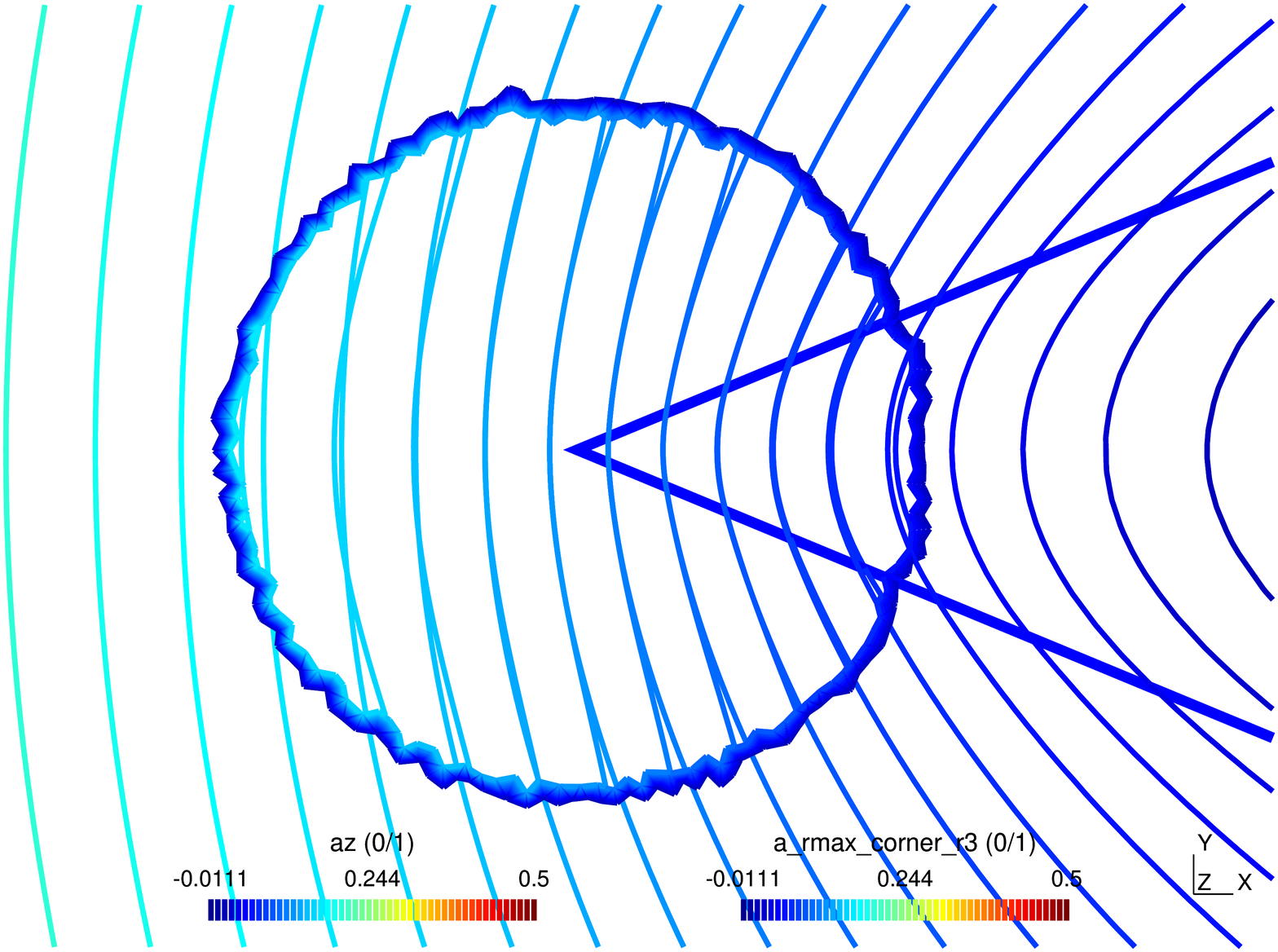}
  }
\hfill
   \subfigure[Expansion restricted to composite order 3. Imaginary part.]{%
    \includegraphics[width=0.48\linewidth]{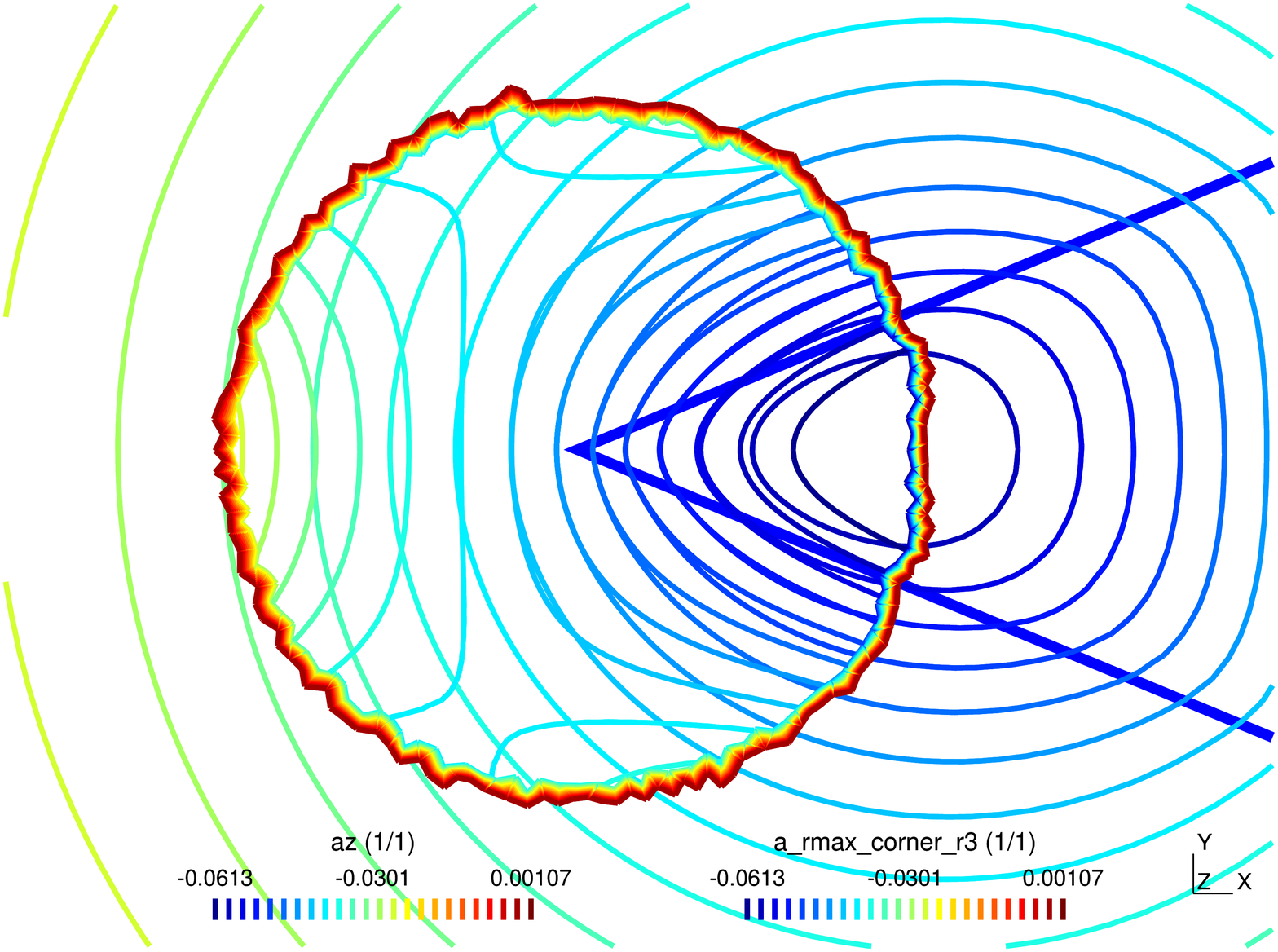}
  }  
  \caption{Comparison of the finite element solution and of the local
    expansion.}
  \label{fig:visualcomp}
\end{figure}
\section{Conclusion}
In this paper, we have provided corner asymptotics of the magnetic
potential for the eddy current problem in a bidimensional
domain. Such expansions involve two ingredients: The calculation of
both primal and dual singularities, and the computation of the
singular coefficients.
Primal and dual singularities of the non-homogeneous operator $-
\Delta +\ri\kappa\mu_0 \sigma \indic_{\Omega_-}$ in $\R^2$
are derived as infinite series, whose coefficients are obtained
recursively.  To compute the singular coefficients, we first tried to
apply the method of moments, straightforwardly derived from the case
of the Laplace operator. This method is limited since it makes possible to obtain only the
first singular coefficients of the corner asymptotics and at a low order of accuracy, therefore
we adapted the method of quasi-dual functions introduced
in~\cite{CoDaYo04} to our specific problem. Numerical simulations of
section~\ref{s:6} corroborate the theoretical order of accuracy obtained in
the previous sections, illustrating the accuracy of the asymptotics.

Forthcoming work will deal with the high conducting case, ({\i.e.} the
case $\sigma$ goes to infinity), for which
two small parameters appear: the distance to the corner and the skin
depth $\sqrt{2/(\kappa\mu_0\sigma)}$. Preliminary results of
the formal derivation of the magnetic potential have been obtained by
Buret~{\it et al} in~\cite{buret2011}, which have to be rigorously
justified and extended.

\appendix

\section{Appendix: Tools for the derivation of the shadows at any order}
\label{appendix}

As  described previously, 
the  singularities of  the operator $\Lscr_\zeta$ \eqref{Ldelta} are  obtained  as series  involving the  shadows  of the  leading  singularities at  any
order.  Moreover, these shadows can be described with the help of the complex variable Ansatz introduced in Section \ref{s:4} and spaces $\Sbb^{\lambda}_{\ell,n}$ defined by equations \eqref{SSlambdaell}. 
However, the derivation of  a generic formula for the shadows is
hardly (and tediously!)  reachable, because of the growth of  the number of
the  terms with the  order of  the shadow.  

For instance,  the leading
singularity $z\in\Sbb^{1}_{0,0}$ generates a primal shadow
$\srm_1^1\in\Sbb^{3}_{1,1}$. According to Proposition~\ref{prop:kneq-1-2}, this shadow $\srm_1^1$ contains 4 terms: $z^3\log z$,  $\bar  z^3\log \bar  z$,  $z^3$ and  $z^2\bar  z$  in the  sector
$\Scal_-$, and we shall prove that each of these terms generates a
shadow belonging to $\Sbb^{5}_{2,2}$.  This shows the complexity of the  generic  expression of  the
singularities.  Nevertheless, it is possible  to obtain the expansion of the
singularities at  any order using a  formal calculus tool.  The aim of
this section is  to provide the elementary results  in order to derive
the singularities, after appropriate use of a formal calculus algorithm.

Actually,  we   see  that, generally speaking, the   shadow  of  the
$k^{\text{th}}$  leading   singularity  at   the  order  $j$   is  a
combination   of   functions   of   the  form   $z^{k+2j-\ell}   \bar
z^{\ell}\log^{n} z$  (or its  conjugate), where  $(n,\ell)\in\N^2$ and
$(k,j)\in\Z\times\N$.  Therefore, if  we provide  a way  to  obtain the
shadow of  this generic function, it  will be possible  to assemble the
terms in order to obtain the full expression of the singularities.

For any $(\lambda,\ell,n)\in\Z\times\N^2$, the idea of the derivation of the shadow generated by $z^{\lambda-\ell}
\bar z^{\ell}\log^{n}z$ 
consists in
performing the same calculus as for the first order shadows, by only
considering the terms of higher degree in $\log z$, the remainder
terms being neglected, and treated then at the next iteration. The
principle of such calculations holds behind the calculus of
Section~\ref{s:4}.

A   congruence   relation  on   the   spaces   $\Sbb^{\lambda}_{\ell,n}$  is   thus
needed. Moreover,  due to  the omission of  the terms of  lower degree,
errors  are  generated on  the  boundary  $\Gcal$,  which have  to  be
corrected and therefore appropriate  spaces of traces and a congruence
on such spaces have also to be given.

\subsection{Definition of the  space  $\TSbb^{\lambda}_{n}$ on $\Gcal$ and the congruence relation $\equiv$ on $\Sbb^{\lambda}_{\ell,n}$ and $\TSbb^{\lambda}_{n}$} 

\begin{defn}
  For $(\lambda,n)\in\Z\times\N$ denote by $\TSbb^{\lambda}_{n}$ the space
  \begin{align}
    \label{eq:defTrace}
    \TSbb^{\lambda}_{n} = 
    \left\{ g: \quad \forall z \in \Gcal, \, %
    g(z) =z^{\lambda}  \di\sum_{q=0}^n  \log^q z
    \begin{cases}
       \alpha_q  , & \text{if} \, \arg
      z = \omega/2
      \\
      \overline{\alpha_q}, &
      \text{if} \, \arg z = -\omega/2
    \end{cases}
    ,\ \ (\alpha_q)_{q=0}^n\in\C^{n+1} \right\}.
  \end{align}
\end{defn}

\begin{remark}
We would obtain the same space $\TSbb^{\lambda}_{n}$ if we replace everywhere $z^{\lambda} \log^q z$ in \eqref{eq:defTrace} by $z^{\lambda} \log^q \bar z$. The same holds with $\bar z^{\lambda} \log^q z$, or $\bar z^{\lambda} \log^q \bar z$, instead of $z^{\lambda} \log^q z$.
\end{remark}

\begin{defn}
The congruence relation on $\Sbb^{\lambda}_{\ell,n}$ \eqref{SSlambdaell} and
  $\TSbb^{\lambda}_n$ \eqref{eq:defTrace} is defined as follows:
  \begin{align*}
    \forall (\lambda,\ell,n)\in\Z\times\N\times\N\setminus\{0\},\quad
  \begin{cases}
    \forall u,u'\in \Sbb^{\lambda}_{\ell,n},&\quad u\equiv u' \quad \mbox{iff} \quad
    u- u'\in \Sbb^{\lambda}_{\ell,n-1}\ ,
    \\
    \forall g,g'\in \TSbb^{\lambda}_{n},&\quad g\equiv g' \quad \mbox{iff} \quad g-
    g'\in \TSbb^{\lambda}_{n-1} \ ,
  \end{cases}
\end{align*}
and for $n=0$ we define $\equiv$ as the usual equality between two
functions.
\end{defn}

Roughly speaking, the congruence relation consists in identifying two
functions with the same coefficient in front of the highest power of
the logarithmic terms.

\subsection{Generic elementary calculation}
\label{PrelimCalc}
Throughout this subsection, we choose an integer $\lambda\in\Z$ (which corresponds to the degree of homogeneity of the solution) and a natural number 
$\ell \in \N$.
Define the sequences of functions of the complex variables
$(f^*_{n},g^*_{n},h^*_{n})_{n\in\N}$ as
\begin{equation}
\forall n \in \N,\qquad  \left\{
    \begin{aligned}
      \forall z\in \Scal_-,\quad
      f^*_{n}(z)&=z^{\lambda-2-\ell}\bar{z}^{\ell}\log^n z,
      \\
      \forall z\in\Gcal,\quad g^*_{n} (z)&=z^{\lambda-1}\log^n z,
      \\
      \forall z\in\Gcal,\quad h^*_{n} (z)&=\bar z^{\lambda-1}\log^n \bar z .
    \end{aligned}
  \right.
\end{equation}
According  to the  previous observations,  the primal  and  dual shadow
terms  at any  order  are  generated by  linear  combinations of  such
$(f^*_{n},g^*_{n},h^*_{n})$. For any $n \in\N$, we also denote  by $v^*_n$  the
function
\begin{equation}
  \label{eq:vn*}
  \forall z\in \Scal_-,\quad  v^*_n(z) =
  \left\{
    \begin{aligned}
      &\frac{z^{\lambda-\ell-1}}{\lambda-\ell-1}\, \frac{\bar
        z^{\ell+1}}{\ell+1}\, \log^n z, \qquad \text{if
        $\ell\neq\lambda-1$},
      \\
      & \frac{1}{n+1}\frac{\bar z^{\ell+1}}{\ell+1}\, \log^{n+1}z, \qquad 
      \text{if $\ell=\lambda-1$}.
    \end{aligned}
  \right.
\end{equation}
Now we choose $n\ge1$ and investigate an induction step.
Let $(\alpha_n,\beta_{n},\gamma_{n})\in\R\times\C^2$, and
let $(f_n,g_{n},h_{n})$ be defined by
\begin{equation}
    \begin{aligned}
      \forall z\in \Scal_-,\quad f_n (z)&=\alpha_n f^*_n,
      \\
      \forall z\in\Gcal,\quad g_{n} (z)&=
      \begin{cases}
        \beta_{n} g^*_{n} (z), \quad \text{for $\theta=\omega/2$,}
        \\
        \overline{\beta_{n}} g^*_{n} (z), \quad \text{for $\theta=-\omega/2$,}
        \\
      \end{cases}
      \\
      \forall z\in\Gcal,\quad h_{n} (z)&=\begin{cases} \gamma_{n} h^*_{n} (z),
        \quad \text{for $\theta=\omega/2$,}
        \\
        \overline{\gamma_{n} }h^*_{n} (z), \quad \text{for $\theta=-\omega/2$.}
      \end{cases}
    \end{aligned}
  \label{genericsource} 
\end{equation}
Therefore, 
$$
   (f_n,g_{n},h_{n})\in
   \Sbb^{\lambda-2,-}_{\ell,n}\times\TSbb^{\lambda-1}_{n}\times\TSbb^{\lambda-1}_{n}.
$$
We are going to show  that the shadow term  $(V_n,W_n)$ generated by  the triple $(
f_n,  g_{n},h_{n})$  can be  obtained  recursively in $\Sbb^{\lambda}_{\ell+1,n+1}$ when $\lambda\neq0$, and in  $\Sbb^{0}_{\ell+1,n+2}$ when $\lambda=0$ by  solving a  partial
differential  equation   problem,  thanks to  the   congruence  relation
$\equiv$.  Actually,  the  shadow  $(V_n,W_n)$  generated  by  $(  f_n,
g_{n},h_{n})$ satisfies
\begin{equation}
  \label{eq:prob4prop}
  \left\{
    \begin{aligned}
      \partial_{z}\partial_{\bar z} V_n & = f_n, \quad & \text{in
        $\Scal_-$}\,,
      \\
      \partial_{z_+}\partial_{\bar z_+} W_n &= 0,\quad & \text{in
        $\Scal_+$}\,,
      \\
      \partial_{z} V_n + \partial_{z_+} W_n &= g_{n}, \quad & \text{on
        $\Gcal$}\,,
      \\
      \partial_{\bar z} V_n + \partial_{\bar z_+} W_n &= h_{n}, \quad &
      \text{on $\Gcal$}\, .
    \end{aligned}
  \right.
\end{equation}
Throughout this subsection, in order to simplify the notations, we set
\begin{align}
  \label{eq:nota}
      &\ell'=\lambda-2-\ell,\quad m=\lambda-1 \, .
\end{align}

\subsubsection{The case $\lambda\neq0$}

We  first consider  the  case $\lambda\neq0$,  which  corresponds to a
source term which does not  belong to $\Tsf^{-2}$, with the notation of
Lemma~\ref{L:inv}.

\begin{prop}
  \label{prop:prelim}
  Let $(\lambda,\ell,n)\in\Z\times\N^2$ and let $(f_n,g_{n},h_{n})$ be
  defined by \eqref{genericsource}. We suppose that $\lambda\neq0$ and 
  \begin{align*}
    \text{either}\quad\Bigl(
    \ell\neq\lambda-1\Bigr), %
    \quad \text{or}\quad %
    \Bigl( \ell=\lambda-1 \quad \text{and}\quad
    \alpha_n=0\Bigr).
  \end{align*}
We define $(v_n,w_n)$ as
  \begin{equation}
    \label{vwn}
    \begin{cases}
      v_n(z) & = a z^{\lambda}\log^{n+1} z + a' {\bar
        z}^{\lambda}\log^{n+1} \bar z + b z^{\lambda}\log^n z \, ,
      \\
       (-1)^{\lambda}  w_n(z_+) & =  a z_+^{\lambda}\log^{n+1} z_+ +
       a' {\bar z}_+^{\lambda}\log^{n+1} \bar z_++ 
      b'{\bar z}_+^{\lambda} \log^n \bar z_+ \, ,
    \end{cases}
  \end{equation}
  where, using notations~\eqref{eq:nota}
  \begin{align*}
    \begin{cases}
      a= \dfrac 1 {\pi \lambda(n+1)} \left( \Im \beta_{n}  + \alpha_n
        \dfrac{ \sin \omega(\ell+1)}{(\ell+1)} \right) \,
      \\[1.5ex]
      b=\dfrac 1 {\lambda} \left( \Re \beta_{n} - \alpha_n
        \dfrac{\cos \omega(\ell+1)}{(\ell+1)} \right) \,
    \end{cases}, %
    \quad
    \begin{cases}%
      a'= \dfrac 1 {\pi \lambda(n+1)} \left(- \Im \gamma_{n} +
        \alpha_n \dfrac{ \sin \omega(\ell'+1)}{(\ell'+1)} \right) \,
      \\[1.5ex]
      b'= \dfrac 1 {\lambda} \left( - \Re \gamma_{n} + \alpha_n
        \dfrac{\cos \omega(\ell'+1)}{(\ell'+1)} \right) \,
    \end{cases}.
  \end{align*}
Then,  $(\alpha_n  v^*_n+v_n,  w_n)$  is  a  particular  solution  to
\eqref{eq:prob4prop} in $\Sbb^{\lambda}_{\ell+1,n+1}$ for the congruence relation $\equiv$.  
This means that 
  the couple of functions $(V_{n-1},W_{n-1})$ defined by
  \begin{align}
    V_{n-1}=V_n-(\alpha_n v^*_n+v_n),\quad
    W_{n-1}=W_n-w_n,\label{VWn-1}
  \end{align}
  satisfies the following problem:
  \begin{align}
    \begin{cases}
      \ \partial_{z}\partial_{\bar z} V_{n-1} = f_{n-1}, \quad
      &\text{in $\Scal_-$}\,,
      \\
      \ \partial_{z_+}\partial_{\bar z_+} W_{n-1} = 0,\quad &\text{in
        $\Scal_+$}\,,
      \\
      \ \partial_{z} V_{n-1} + \partial_{z_+} W_{n-1} ={g}_{n-1}, \quad
      &\mbox{on $\Gcal$}\,,
      \\
      \ \partial_{\bar z} V_{n-1} + \partial_{\bar z_+} W_{n-1} =
      h_{n-1}, \quad &\mbox{on $\Gcal$}\, ,
    \end{cases}
    \label{vwfn-1}
  \end{align}
where 
$$
   (f_{n-1},g_{n-1},h_{n-1})\in
   \Sbb^{\lambda-2,-}_{\ell,n-1}\times\TSbb^{\lambda-1}_{n-1}\times\TSbb^{\lambda-1}_{n-1}.
$$
More precisely we have the explicit formulas
\begin{subequations}
\begin{align}
    f_{n-1}&=-\dfrac{n\alpha_n}{\ell'+1} f^*_{n-1},\\
    g_{n-1}&= n\Bigl(b - a (n+1)\bigl((\mp\ri\pi)\lambda/2 +1\bigr)
    (\mp\ri\pi)\Bigr) g^*_{n-1} \\
    & -a z^m\sum_{q=0}^{n-2 } \left((\mp\ri\pi)\lambda
      \begin{pmatrix}
        q\\n+1
      \end{pmatrix}
      +(n+1)
      \begin{pmatrix}
        q\\n
      \end{pmatrix}
    \right) (\mp\ri\pi)^{{n}-q}\log^q z ,\quad \text{for $\theta=\pm
      \omega/2$},
    \nonumber \\
    h_{n-1}&=-\alpha_n \frac{e^{\pm\ri\omega
        (\ell'+1)}}{\ell'+1}\bar{z}^m \sum_{q=0}^{n-1}
    \begin{pmatrix}
      q\\n
    \end{pmatrix}
    \log^q \bar{z}\left(\pm\ri\omega\right)^{n-q}
    \\
    &-b' \bar z^{m} \sum_{q=0}^{n-1}
    \left((\pm\ri\pi)\lambda
      \begin{pmatrix}
        q\\n
      \end{pmatrix}
      + n \begin{pmatrix}
        q\\n-1
      \end{pmatrix}\right)
    (\pm\ri\pi)^{n-1-q}\log^q \bar{z} \nonumber
    \\&
    -a' \bar z^m\sum_{q=0}^{n-1}
    \left((\pm\ri\pi)\lambda
      \begin{pmatrix}
        q\\n+1
      \end{pmatrix}
      +(n+1)\begin{pmatrix}
        q\\n
      \end{pmatrix}\right)
    (\pm\ri\pi)^{{n}-q}\log^q \bar{z}
    ,\quad \text{for
      $\theta=\pm \omega/2$}. \nonumber
  \end{align}
\end{subequations}
For $n=0$, the calculation  is exact, meaning that $f_{-1}$, $g_{-1}$,
and $h_{-1}$ equal zero.
\end{prop}

Before proving the above proposition, we first show that in the case:
$$
  \ell=\lambda-1,\quad \text{and}\quad \alpha_n\neq0,
$$
we can reduce to Proposition~\ref{prop:prelim}. A direct calculation based on formula \eqref{eq:vn*} in the case when $\ell=\lambda-1$ proves that there holds
\begin{lem}
  \label{cor:ell'=1+k+2j}
If $\ell=\lambda-1$ and $\alpha_n\neq0$,
the shadow $(V_n,W_n)$ solution to problem~\eqref{eq:prob4prop} writes
$$
  V_n=\alpha_n v^*_n+U_n,
$$
where $(U_n,W_n)$ satisfies
  \begin{equation}
    \label{eq:prob5prop}
    \left\{
      \begin{aligned}
        \partial_{z}\partial_{\bar z} U_n & = 0,
        \quad & \text{in $\Scal_-$}\,,
        \\
        \partial_{z_+}\partial_{\bar z_+} W_n &= 0,\quad & \text{in
          $\Scal_+$}\,,
        \\
        \partial_{z}  U_n  + \partial_{z_+} W_n &= g_n -
        \alpha_n z^{-1}\frac{{\bar {z}}^{ \ell+1}}{\ell+1}\log^n z,
        \quad & \text{on $\Gcal$}\,,
        \\
        \partial_{\bar z} U_n  + \partial_{\bar z_+} W_n &=
        h_n - \frac{\alpha_n }{n+1}{\bar {z}}^{\ell}\log^{n+1} z,
        \quad & \text{on $\Gcal$}\, ,
      \end{aligned}
    \right.
  \end{equation}
\end{lem}

As a consequence of Proposition~\ref{prop:prelim} and Lemma \ref{cor:ell'=1+k+2j}
we find:

\begin{cor}
Let $\lambda\neq0$, $\ell\in\N$ and $n\in\N$ be chosen. Let $(f_n,g_{n},h_{n})$ be
defined by \eqref{genericsource}. Then the  shadow  $(V_n,W_n)$  generated  by  $(f_n,
g_{n},h_{n})$ belongs to $\Sbb^{\lambda}_{\ell+1,n+1}$. An analytic formula for $(V_n,W_n)$ can be recursively deduced from Proposition~\ref{prop:prelim} and Lemma \ref{cor:ell'=1+k+2j}.
\end{cor}

\begin{proof}
If $\ell\neq\lambda-1$,  we observe that problem~\eqref{vwfn-1} involves source terms whose
  power in $\log z$ is $n-1$, therefore pushing forward the reasoning we
  can decrease the power of the logarithmic terms of the source up to
  zero, and then solve exactly the last problem with no logarithm.

If $\ell=\lambda-1$, we first use Lemma \ref{cor:ell'=1+k+2j} and reduce to the situation of 
Proposition~\ref{prop:prelim} with the difference that a real non-zero coefficient $\gamma_{n+1}$ then appears. Examining the formulas given in Proposition~\ref{prop:prelim}, we can see that even with this non-zero $\gamma_{n+1}$, the solution $(\alpha_n v^*_n+v_n, w_n)$ belongs to $\Sbb^{\lambda}_{\ell+1,n+1}$.
\end{proof}

\begin{proof}[Proof of Proposition~\ref{prop:prelim}]
  Observe that $v^*_n$ satisfies
  \begin{align*}
    & \partial_z\partial_{\bar{z}}v^*_n=f^*_n+\frac{n}{ \ell' +1}f^*_{n-1} ,
    \\
    &
    \partial_z v^*_n=z^{\ell'}\frac{\bar z^{\ell+1}}{\ell+1}
    \log^n z+n\frac{z^{\ell'}}{\ell'+1}\frac{\bar z^{\ell+1}}{\ell+1}
    \log^{n-1} z \quad \text{and} \quad
    \partial_{\bar z} v^*_n=\frac{z^{\ell'+1}}{\ell'+1}\bar z^{\ell} \log^n z.
  \end{align*}
  To prove  the proposition, we  derive a particular solution  $(v, w)$
  for the congruence relation $\equiv$ to the following problem
  \begin{equation}
    \left\{
      \begin{aligned}
        &\partial_{z}\partial_{\bar z} v \equiv 0, \quad \text{in
          $\Scal_-$}\, ,
        \\
        &\partial_{z_+}\partial_{\bar z_+} w \equiv 0, \quad \text{in
          $\Scal_+$}\, ,
        \\
        & \partial_{z} v + \partial_{z_+} w \equiv - \alpha_n
        z^{\ell'}\, \di\frac{\bar z^{\ell+1}}{\ell+1}\, \log^n z +
        g_n, \quad \mbox{on $\Gcal$}\,,
        \\
        & \partial_{\bar z} v + \partial_{\bar z_+} w \equiv -
        \alpha_n \di\frac{z^{\ell'+1}}{\ell'+1}\, {\bar z^{\ell}} \,
        \log^n z + h_n, \quad \mbox{on $\Gcal$}\, .
      \end{aligned}
    \right.
    \label{vw}
  \end{equation}
  Note  that  if $(v_n,w_n)$  given  by  \eqref{vwn}  is a  particular
  solution to  \eqref{vw} for  the congruence relation  $\equiv$, then
  careful calculations show  that the couple $(V_{n-1},W_{n-1})$ given
  by \eqref{VWn-1}  satisfies \eqref{vwfn-1}, which will  end the proof
  of the proposition.

  Derive now  a particular solution  to \eqref{vw} for  the congruence
  relation $\equiv$.  For $z=re^{\pm\ri\omega/2}$, there holds
  \begin{align*}
    z^{\ell'}\, \di\frac{\bar z^{\ell+1}}{\ell+1}\, \log^n
    z&=\frac{e^{\mp\ri\omega(\ell+1)}}{\ell+1}z^m \log^n z,
    \\
    z^{\ell'+1}\, \di\frac{\bar z^{\ell}}{\ell'+1}\, \log^n
    z&=\frac{e^{\pm\ri\omega (\ell'+1)}}{\ell'+1}\bar{z}^m
    \log^n \bar{z}+\frac{e^{\pm\ri\omega (\ell'+1)}}{\ell'+1}\bar{z}^m
    \sum_{q=0}^{n-1}
    \begin{pmatrix}
      q\\n
    \end{pmatrix}
    \log^q \bar{z}\left(\pm\ri\omega\right)^{n-q}.
  \end{align*}
  Hence, denoting by $(c,c')$ the complex numbers given by
  \begin{equation*}
    c= \beta_n-\alpha_n\frac{e^{-\ri\omega(\ell+1)}}{\ell+1},\quad
    c'= \gamma_n -\alpha_n\frac{e^{\ri\omega (\ell'+1)}}{\ell'+1},
  \end{equation*}
  the  transmission conditions  of \eqref{vw}  across $\Gcal$
  reads
  \begin{subequations}
    \begin{align}
      & \partial_{z} v + \partial_{z_+} w \equiv c z^m\log^n z \quad
      &\mbox{for $\theta=\omega/2$}\,,
      \\
      & \partial_{z} v + \partial_{z_+} w \equiv \bar{c} z^m\log^n z
      \quad &\mbox{for $\theta=-\omega/2$}\,,
      \\
      & \partial_{\bar z} v + \partial_{\bar z_+} w \equiv
      c'\bar{z}^m\log^n \bar{z}\quad &\mbox{for
        $\theta=\omega/2$}\,,
      \\
      & \partial_{\bar z} v + \partial_{\bar z_+} w \equiv
      \bar{c'}\bar{z}^m\log^n \bar{z}\quad &\mbox{for
        $\theta=-\omega/2$}\, .
    \end{align}
    \label{vwcmn}
  \end{subequations}
  Referring to  Section~\ref{sec:1rstShadow}, we assume that $(v,w)$
  writes
  \begin{equation}
    \label{u1n}
    \begin{cases}
      \ v(z) & = a z^{\lambda}\log^{n+1} z + a' {\bar z}^{\lambda}\log^{n+1}
      \bar z + b z^{\lambda}\log^n z \, ,
      \\
      \ (-1)^{\lambda}  w(z_+)& =  a z_+^{\lambda}\log^{n+1} z_+ +
       a' {\bar z}_+^{\lambda}\log^{n+1} \bar z_+
      + b'{\bar z}_+^{\lambda} \log^n \bar z_+ \, ,
    \end{cases}
  \end{equation}
  with $\lambda=m+1$.  Then, the following properties hold on $\Gcal$:
  \begin{align*}
    &
    \begin{cases}
      \ \partial_{z} v - a \lambda z^{m} \log^{n+1} z - a (n+1) z^m
      \log^{n} z -
      b \lambda z^{m}\log^n z= bn z^{m}\log^{n-1} z& \in \TSbb^{m}_{n-1}\, ,
      \\
      \ (-1)^{\lambda} \partial_{z_+} w - a \lambda z_+^{m}\log^{n+1} z_+ -
      a (n+1) z_+^m \log^{n} z_+= 0 & \in \TSbb^{m}_{n-1}\, ,
    \end{cases}
    \intertext{and similarly} &
    \begin{cases}
      \ \partial_{\bar{z}} v - a' \lambda \bar{z}^{m}\log^{n+1} \bar{z} -
      a' (n+1) \bar{z}^m \log^{n}\bar{z}=0 & \in \TSbb^{m}_{n-1}\, ,
      \\ 
      \  (-1)^{\lambda} \partial_{\bar{z}_+} w - a' \lambda
      \bar{z}_+^{m}\log^{n+1} \bar{z}_+ - a' (n+1) \bar{z}_+^m
      \log^{n} \bar{z}_+
      - b' \lambda \bar{z}^m_+\log^n \bar{z}_+= b'n
      \bar{z}^m_+\log^{n-1} \bar{z}_+& \in\TSbb^{m}_{n-1}\, .
    \end{cases}
  \end{align*}
  Recall that $z_+=-z$.  Due to the  branch cut on $\R^-$ of $\log$, we
  have  to   distinguish  the  two   branches  of  $\Gcal$   given  by
  $\theta=\omega/2$    and     $\theta=-\omega/2$.     Actually,    for
  $\theta=\pm\omega/2$, since $\theta_+=\pm\omega/2-\pm\pi$, we infer
  for $\tilde{n}\in\{n+1,n\}$
  \begin{align*}
    z^{m}\log^{\tilde{n}} z + (-1)^{\lambda} z_+^{m}\log^{\tilde{n}}
    z_+ = \pm\ri\pi \tilde{n} z^m \log^{\tilde{n}-1} z
    - z^m\sum_{q=0}^{\tilde{n}-2}
    \begin{pmatrix}
      q\\\tilde{n}
    \end{pmatrix}
    (\mp\ri\pi)^{\tilde{n}-q}\log^q z, \quad \text{for
      $\theta=\pm\omega/2$}.
  \end{align*}
  Therefore 
  \begin{equation*}
    \partial_{z} v + \partial_{z_+} w = \pm\ri\pi a \lambda(n+1) z^m
    \log^n z + b \lambda z^{m}\log^n z +g_{n-1},\quad \text{for
      $\theta=\pm\omega/2$},
  \end{equation*}
  where $g_{n-1}$ equals
  \begin{equation*}
    g_{n-1} = b n  z^{m}\log^{n-1} z
    -a z^m\sum_{q=0}^{n-1} \left((\mp\ri\pi)\lambda
      \begin{pmatrix}
        q\\n+1
      \end{pmatrix}
      +(n+1)\begin{pmatrix}
        q\\n
      \end{pmatrix}\right)
    (\mp\ri\pi)^{{n}-q}\log^q z ,\quad \text{for $\theta=\pm \omega/2$}.
  \end{equation*}
  We thus infer   $g_{n-1}\in   \TSbb^{m}_{n-1}$.   Therefore,  the   two  first
  congruence relations in \eqref{vwcmn} are satisfied iff
  \begin{align*}
    & \ri\pi a \lambda(n+1) + b \lambda = c\,,
    \\
    -&\ri\pi a \lambda(n+1) + b \lambda = \bar c\,.
  \end{align*}
  This system  of two equations  has a unique  solution $(a,b)\in\R^2$
  given by
  \begin{align*}
    & a=\dfrac{1}{\pi \lambda(n+1)}\Im(c)\,,\quad
    &b&=\dfrac{1}{\lambda}\Re(c)\,,& %
    \intertext{hence } %
    &a= \dfrac{1}{\pi \lambda(n+1)} \left( \Im\beta_n + \alpha_n
      \dfrac{\sin\omega(\ell+1)}{(\ell+1)} \right)\,,\quad
    &b&=\dfrac{1}{\lambda} \left( \Re\beta_n -\alpha_n
      \dfrac{\cos\omega(\ell+1)}{(\ell+1)} \right)\,.&
  \end{align*}
  Similar computations imply that 
  \begin{equation*}
    \partial_{\bar z} v + \partial_{\bar z_+} w = \mp\ri\pi a'
    \lambda(n+1) \bar z^m \log^n \bar z - b '\lambda \bar z^{m}\log^n \bar z
    +h_{n-1},\quad \text{for $\theta=\pm\omega/2$},
  \end{equation*}
  where $h_{n-1}$ equals
  \begin{align*}
  h_{n-1}&=-\alpha_n
    \frac{e^{\pm\ri\omega (\ell'+1)}}{\ell'+1}\bar{z}^m \sum_{q=0}^{n-1}
    \begin{pmatrix}
      q\\n
    \end{pmatrix}
    \log^q \bar{z}\left( \pm  \ri\omega\right)^{n-q}
    \\
    &-b' \bar z^{m} \sum_{q=0}^{n-1}
    \left((\pm\ri\pi)\lambda
      \begin{pmatrix}
        q\\n
      \end{pmatrix}
      + n \begin{pmatrix}
        q\\n-1
      \end{pmatrix}\right)
    (\pm\ri\pi)^{n-1-q}\log^q \bar{z}
    \\
    & -a' \bar z^m\sum_{q=0}^{n-1}
    \left((\pm\ri\pi)\lambda
      \begin{pmatrix}
        q\\n+1
      \end{pmatrix}
      +(n+1)
      \begin{pmatrix}
        q\\n
      \end{pmatrix}\right)
    (\pm\ri\pi)^{{n}-q}\log^q \bar{z}
    ,\quad \text{for
      $\theta=\pm \omega/2$}.
  \end{align*}
  Therefore,  the third  and  fourth equations  in  \eqref{vwcmn} write  on
  $\Gcal$ as
  \begin{align*}
    -&\ri\pi a' \lambda(n+1) \bar{z}^m \log^n \bar{z} - b' \lambda
    \bar{z}^{m}\log^n \bar{z}\equiv c' \bar{z}^{m}\log^n \bar
    z,\quad \text{for $\theta=\omega/2$}\,,
    \\
    &\ri\pi a' \lambda(n+1) \bar{z}^m \log^n \bar{z} - b' \lambda
    \bar{z}^{m}\log^n \bar{z}\equiv \bar{c'} \bar{z}^{m}\log^n \bar
    z,\quad \text{for $\theta=-\omega/2$}\,.
  \end{align*}
  Thus, the couple $(a',b')$ is unique, in $\R^2$ and satisfies
  \begin{align*}
    &\ri\pi a' \lambda(n+1) + b' \lambda = - c' \,,
    \\
    -&\ri\pi a' \lambda(n+1) + b' \lambda = - \bar{c'} \,,
  \end{align*}
  hence 
  \begin{align*}
    a'= \dfrac 1 {\pi \lambda(n+1)} \left(- \Im \gamma_n + \alpha_n
      \dfrac{ \sin \omega(\ell'+1)}{(\ell'+1)} \right) \, \,,%
    \quad b'= \dfrac 1 {\lambda} \left( - \Re \gamma_n + \alpha_n
      \dfrac{\cos \omega(\ell'+1)}{(\ell'+1)} \right) \,,
  \end{align*}
  and  therefore,  $(v_n,w_n)$ given  by  \eqref{vwn}  is a  particular
  solution  to   \eqref{vw},  and  satisfies  the   assertions  of  the
  proposition by constructions, which ends the proof.
\end{proof}

\subsubsection{The case $\lambda=0$}

We consider now the case $\lambda=0$, which corresponds to a source term
in  $\Tsf^{-2}$,  according  to  Lemma~\ref{L:inv}. As  seen
previously at  Proposition~\ref{prop:k-2}, it is  necessary to increase
the power  of the logarithm terms  of the solution by  1, in comparison
with the  case $\lambda\neq0$. Moreover, observe that a constant term may
appear,  as  in  Proposition~\ref{prop:k-2}  for  instance,  and  it  is
necessary to consider  problem~\eqref{eq:shadowdualPart} to ensure the
continuity of the solution across $\Gcal$. The boundary data are now
$\widetilde{g}_{n}$ and $\widetilde{h}_{n+1}$ defined as
$$
   \forall z\in\Gcal,\quad 
   \widetilde{g}_{n} (z) = \begin{cases} \beta_{n} \log^n z,
        \quad \text{for $\theta=\omega/2$,}
        \\
        \overline{\beta_{n} }\log^n z, \quad \text{for $\theta=-\omega/2$,}
      \end{cases} \quad\mbox{and}\quad
   \widetilde{h}_{n+1} (z) = \begin{cases} \gamma_{n+1} \log^{n+1} z,
        \quad \text{for $\theta=\omega/2$,}
        \\
        \overline{\gamma_{n+1} }\log^{n+1} z, \quad \text{for $\theta=-\omega/2$,}
  \end{cases}
$$
where $n\in\N$, $\beta_n\in\C$, and $\gamma_{n+1}\in\C$. The problem which we want to solve is, instead of \eqref{eq:prob4prop},
  \begin{align}
    \begin{cases}
      \ \partial_{z}\partial_{\bar z} V_{n} = f_{n}, \quad &\text{in
        $\Scal_-$}\,,
      \\
      \ \partial_{z_+}\partial_{\bar z_+} W_{n} = 0,\quad &\text{in
        $\Scal_+$}\,,
      \\
      \ z\partial_{z} V_{n} -\bar z\partial_{\bar z}
      V_{n}-z_+ \partial_{z_+} W_{n} +\bar z_+ \partial_{\bar z_+}
      W_{n} =\widetilde{g}_{n} \quad &\mbox{on $\Gcal$}\,,
      \\
      \ V_{n} -W_{n} =\widetilde{h}_{n+1} \quad &\mbox{on $\Gcal$}\, .
    \end{cases}
    \label{vwfn-1bis}
  \end{align}
We have the following proposition.
\begin{prop}
  \label{prop:prelimbis}
  Let 
  $(\ell,n)\in\N\times\left(\N\setminus\{0\}\right)$. 
  Note      that,      according      to~\eqref{eq:nota},
  $\ell'=-2-\ell$. Define $(v_n,w_n)$ as
  \begin{equation}
    \label{vwnbis}
    \begin{cases}
      & v_n (z) \quad\!\!\! = a \log^{n+2} z + a' \log^{n+2} \bar z +
      b \log^{n+1} z \, ,
      \\
      & w_n (z_+) \,\!\!\! = a \log^{n+2} z_+ +
      a' \log^{n+2} \bar z_+ +  b' \log^{n+1} {\bar z}_+ \, ,
    \end{cases}
  \end{equation}
  where
  \begin{align*}
    &
    \begin{cases}
      a= \dfrac 1 {2 \pi (n+2)} \left( \dfrac 1 {n+1} \left(\Im
          \beta_n+ 2\alpha_n \dfrac{
            \sin \omega(\ell+1)}{(\ell+1)} \right) + \Im
        \gamma_{n+1} \right) \,
      \\
      b= \dfrac 1 2 \left(\dfrac 1 {n+1} \left( \Re \beta_n -
          2\alpha_n \dfrac{\cos\omega(\ell+1)}{(\ell+1)} \right) +
        \Re \gamma_{n+1} \right)\,
    \end{cases},
    \\
    &
    \begin{cases}%
      a'= \dfrac 1 {2 \pi (n+2)} \left( \dfrac 1 {n+1} \left(\Im
          \beta_n + 2\alpha_n \dfrac{
            \sin\omega(\ell+1)}{(\ell+1)} \right) - \Im
        \gamma_{n+1}\right) \,
      \\
      b'= \dfrac 1 2 \left( \dfrac 1 {n+1} \left( \Re \beta_n -
          2\alpha_n \dfrac{\cos\omega(\ell+1)}{(\ell+1)} \right) -
        \Re \gamma_{n+1} \right)\,
    \end{cases}.
  \end{align*} 
  We recall that according to \eqref{eq:vn*}, for $\lambda=0$:
  $$
  v^*_n=-\frac{1}{(\ell+1)^2}(\bar z/z)^{\ell+1}\log^n z,
  $$
  Then, $(\alpha_n v^*_n+v_n, w_n)$ is a particular solution to \eqref{vwfn-1bis} 
  for  the  congruence  relation  $\equiv$. This means that  the  couple  of
  functions $(V_{n-1},W_{n-1})$ defined by
  \begin{align}
    V_{n-1}=V_n-(\alpha_n v^*_n+v_n),\quad
    W_{n-1}=W_n-w_n,\label{VWn-1bis}
  \end{align}
  satisfies the following problem:
  \begin{align}
    \begin{cases}
      \ \partial_{z}\partial_{\bar z} V_{n-1} = f_{n-1}, \quad
      &\text{in $\Scal_-$}\,,
      \\
      \ \partial_{z_+}\partial_{\bar z_+} W_{n-1} = 0,\quad &\text{in
        $\Scal_+$}\,,
      \\
      \ z\partial_{z} V_{n-1} -\bar z\partial_{\bar z}
      V_{n-1}-z_+ \partial_{z_+} W_{n-1} +\bar z_+ \partial_{\bar z_+}
      W_{n-1} = \widetilde{g}_{n-1} \quad &\mbox{on $\Gcal$}\,,
      \\
      \ V_{n-1} -W_{n-1} =\widetilde{h}_n \quad &\mbox{on $\Gcal$}\, ,
    \end{cases}
    \label{vwfn-1ter}
  \end{align}  
where
$$
   (f_{n-1}, \widetilde g_{n-1}, \widetilde h_{n})\in
   \Sbb^{-2,-}_{\ell,n-1}\times\TSbb^{0}_{n-1}\times\TSbb^{0}_{n}.
$$
More precisely we have the explicit formulas
\begin{subequations}
 \begin{align}
    f_{n-1}&=-\dfrac{n\alpha_n }{\ell'+1} f^*_{n-1},
\\ \label{wgn-1}
    \widetilde g_{n-1}&=(n+2)a\sum_{q=0}^{n-1}
    \begin{pmatrix}
      q\\n+1
    \end{pmatrix}
    (\mp\ri\pi)^{n+1-q}\log^q z-(n+2)a'\sum_{q=0}^{n-1}
    \begin{pmatrix}
      q\\n+1
    \end{pmatrix}
    (\pm\ri\pi)^{n+1-q}\log^q \bar z
    \\
    &-(n+1)b'\sum_{q=0}^{n-1}
    \begin{pmatrix}
      q\\n+1
    \end{pmatrix}
    (\pm\ri(\pi-\omega))^{n+1-q}\log^q z%
    - \frac {n \alpha_n e^{\mp i \omega (\ell+1)}}
    {(\ell+1)^2} \log^{n-1} z ,
    \nonumber\\ \label{whn}
    \widetilde{h}_n&=  a  \sum_{q=0}^{n}
    \begin{pmatrix}
      q\\n+2
    \end{pmatrix}
    (\mp\ri\pi)^{n+2-q}\log^q z  + a' \sum_{q=0}^{n}
    \begin{pmatrix}
      q\\n+2
    \end{pmatrix}
    (\pm\ri\pi)^{n+2-q}\log^q \bar z
    \\
    &+b'\sum_{q=0}^{n-1}
    \begin{pmatrix}
      q\\n+1
    \end{pmatrix}
    (\pm\ri(\pi-\omega))^{n+1-q}\log^q  z%
     + \frac {\alpha_n e^{\mp i \omega (\ell+1)}}
    {(\ell+1)^2} \log^n z . \nonumber
  \end{align}
\end{subequations} 
\end{prop}

\begin{proof}
  Writing $z\partial_z v_{\ell,n }-\bar z\partial_{\bar z} v_{\ell,n
    }$ and $z_+\partial_{z_+} w_{ \ell,n }-\bar z_+\partial_{\bar
    z_+} w_{\ell,n }$ and identifying the term in 
  $\log^{n}z$, we infer
  \begin{align*}
    & \pi(n+2)(n+1)(a+a')=\Im\beta_n+%
    2\alpha_n\frac{\sin\omega(\ell+1)}{\ell+1},
    \\
    &(n+1)(b+b')=\Re\beta_n-%
    2\alpha_n\frac{\cos\omega(\ell+1)}{\ell+1}.
  \end{align*}
  Then, writing the condition $v_{\ell,n }+\alpha_nv_n^\ast-w_{\ell, n
    }=\widetilde{h}_{n+1}$ and identifying the term in $\log^{n+1} z$, we
  obtain
  \begin{align*}
    a-a'&=\frac{1}{\pi (n+2)}\Im\gamma_{n+1},
    \\
    b-b'&=\Re\gamma_{n+1},
  \end{align*}
  hence the proposition with $\widetilde g_{n-1}$ and $\widetilde{h}_n$ given by \eqref{wgn-1} and \eqref{whn}.
\end{proof}

For $n=0$, we  solve exactly the following problem
\begin{subequations}
  \label{eq:nvPB}
  \begin{align}
    & \partial_z\partial_{\bar z}V_{\ell,0}=\frac{-\alpha_{\ell,0}}{(\ell+1)^2}(\bar
    z/z)^{\ell+1},
    \\
    &\partial_z\partial_{\bar z}W_{\ell,0}=0,
    \\
    &V_{\ell,0}-W_{\ell,0}=
    \begin{cases}
      \gamma_1 \log z + \gamma_0,\text{if $\theta=\omega/2$},
      \\
      \bar \gamma_1 \log z + \bar \gamma_0,\text{if
        $\theta=-\omega/2$},
    \end{cases}
    \\
    &(z\partial_zV_{\ell,0}-\bar z\partial_{\bar z}V_{\ell,0})-(z_+\partial_{
      z_+}W_{\ell,0}-\bar z_+\partial_{\bar z_+}W_{\ell,0})=
    \begin{cases}
      \beta_0,\text{if $\theta=\omega/2$},
      \\
      \bar\beta_0,\text{if $\theta=\omega/2$},
    \end{cases}
  \end{align}
\end{subequations}
thanks to the following proposition.
\begin{prop} 
  \label{prop:prelim0bis}
  Let $\ell\in\N$.
  Recall that, in this case,
  $$
  v^*_{\ell,0}=-\frac{1}{(\ell+1)^2}(\bar z/z)^{\ell+1}.
  $$
  Define $(v_{\ell,0},w_{\ell,0})$ as
  \begin{equation}
    \label{vw0bis}
    \begin{cases}
      & v_{\ell,0} (z) \quad\!\!\! = a \log^{2} z + a' \log^{2} \bar z +
      b \log z \, ,
      \\
      & w_{\ell,0} (z_+) \,\!\!\! = a \log^2 z_+\, +
      a' \log^{2} \bar z_+\, +  \tilde{b} \log z_++ b' \log {\bar z}_+\,+c,
    \end{cases}
  \end{equation}
  where, setting
  $$
  d=\beta_0-2\alpha_{\ell,0} e^{-\ri\omega(\ell+1)}/(\ell+1),
  $$
  we have
  \begin{align*}
    \begin{cases}
      a=a'+\dfrac{1}{2\pi}\Im\gamma_1=
      \,\dfrac{1}{4\pi}\Im(d+\gamma_1),
      \\[1.5ex]
      b'=\,\dfrac{1}{2}\Re(d-\gamma_1),
      \\[1.5ex]
      \tilde{b}=
      \,\dfrac{\pi-\omega}{\pi}b'-\dfrac{\alpha_{\ell,0}}{\pi(\ell+1)^2}
      \sin\omega(\ell+1) + \Im \gamma_0,
      \\[1.5ex]
      b= \,b'+\tilde{b}+\Re\gamma_1,
      \\[1.5ex]
      c=2\pi a(\pi-\omega)-\dfrac{\alpha_{\ell,0}}{(\ell+1)^2}
      \cos\omega(\ell+1)-\dfrac{\pi-\omega}{2}\Im\gamma_1
      -\Re\gamma_0.
    \end{cases}.
  \end{align*} 
  Then, $(\alpha_{\ell,0}  v^*_{\ell,0}+v_{\ell,0}, w_{\ell,0})$ is a particular  exact solution to
  \eqref{eq:nvPB}.
\end{prop}

\begin{proof}
  The transmission conditions of \eqref{eq:nvPB} across $\Gcal$ writes
  \begin{subequations}
    \begin{align*}
      & (z\partial_{z} v_{\ell,0} -\bar z\partial_{\bar
        z}v_{\ell,0})-(z_+ \partial_{z_+} w_{\ell,0} -\bar z_+\partial_{\bar
        z_+}w_{\ell,0})= \pm 2\ri\pi (a+a')+b-\tilde{b}+b'.
    \end{align*}
  \end{subequations}
  Since
  \begin{equation*}
    \displaystyle d= \beta_{0}-2\alpha_{\ell,0}\frac{e^{-\ri\omega(\ell+1)}}{\ell+1},
  \end{equation*}
  we infer that
  \begin{align*}
    2\ri\pi (a+a') + b -\tilde{b}+b' & = d\,,&
    \\
    -2\ri\pi (a+a') + b -\tilde{b} +b' & = \bar d\,.&
  \end{align*}
  Since
  $$
  v_{\ell,0}-w_{\ell,0}=a(\pm2\ri\pi\log z+\pi^2)+a'(\mp2\ri\pi\log z-2\pi\omega+\pi^2)
  +(b-\tilde{b}-b')\log z\pm\ri(\pi\tilde{b}-(\pi-\omega)b')-c,
  $$
  the continuity across $\Gcal$ implies
  \begin{align*}
    &  \pm2\ri\pi (a-a')+(b-\tilde{b}-b')=
    \begin{cases}
      \gamma_1 ,\text{if $\theta=\omega/2$},
      \\
      \bar  \gamma_1 ,\text{if $\theta=-\omega/2$},
    \end{cases}
    ,\\
    &\pi^2a+\pi(\pi-2\omega)a'\pm\ri(\pi\tilde{b}-(\pi-\omega)b')-c
    =\frac{\alpha_{\ell,0}}{(\ell+1)^2}e^{\mp\ri\omega (\ell+1)}+
    \begin{cases}
      \gamma_0 ,\text{if $\theta=\omega/2$},
      \\
      \bar  \gamma_0 ,\text{if $\theta=-\omega/2$},
    \end{cases},
  \end{align*}
  hence
  $$
  a=a'+\frac{1}{2\pi}\Im\gamma_1,\quad b-\tilde b-b'=\Re\gamma_1,
  $$
  which ends the proof.
\end{proof}

\subsection{Formal derivation of the singularities}

\subsubsection{Principles}

We are now ready to present how to obtain the singularities.

\begin{itemize}
\item First, choose the leading (primal or dual) singularity: $z^k$,
  $k\in\Z$ or $\log z$.

\item      For     $z^k$,      use     Proposition~\ref{prop:kneq-1-2},
  Proposition~\ref{prop:k-1}  or  Proposition~\ref{prop:k-2} depending
  whether   $  k\in\{-1,-2\}$   or   not,  and   for   $\log  z$   use
  Proposition~\ref{prop:k-0},  in  order to  obtain  the first  shadow
  term.   This  shadow  is  composed of  terms  as  $z^{k+2-\ell}\bar
  z^{\ell}\log^n z$, with $n=1$ or $2$ and $\ell\in\{0,1,2\}$.

\item If $k+2\neq-2$, check whether $k+2-\ell=0$, and apply recursively
  Proposition~\ref{prop:prelim} or Corollary~\ref{cor:ell'=1+k+2j} for
  each   term   of  the   shadow.   If   $k+2=-2$,  apply   recursively
  Proposition~\ref{prop:prelimbis}               and              then
  Proposition~\ref{prop:prelim0bis}.

\item Repeat the process till the desired order of the shadow.

\item Sum all the terms.
\end{itemize}

\subsubsection{Generic expression of the shadows at any order}

The  results  of subsection~\ref{PrelimCalc}  leads  to the  following
proposition,  that  provides  the  generic  expression  of  the  primal
singularities of $\Lscr_{\zeta}$.
\begin{prop}
  \label{prop:genform}
  Let $k\in\N$. We recall that according to~\eqref{eq:sk0z}, $\srm^k_0=z^k$. For $j\geq 1$, the
  function $\srm^k_j = (\chi^k_j,\xi^k_j)$ defined by \eqref{eq:primalshadow}, which is the shadow of order $j$ of
  $\srm^k_0$, writes
  \begin{multline}
    \label{vkj}
    \chi^k_j(z) = a_{kj} z^{k+2j}\log^jz + a'_{kj} \bar
    z^{k+2j}\log^j\bar z
    \\
    + \sum_{n=0}^{j-1} \sum_{m=0}^{j-n} b_{kj,nm} z^{k+2j-m}\bar
    z^{m}\log^nz + \sum_{n=0}^{j-1} \sum_{m=0}^{j-n} c_{kj,nm}
    z^{m}\bar z^{k+2j-m}\log^n\bar z,
  \end{multline}
  \begin{multline}
    \xi^k_j(z_+) = (-1)^k a_{kj} z_+^{k+2j}\log^jz_+ + (-1)^k a'_{kj}
    \bar z_+^{k+2j}\log^j\bar z_+
    \\
    + \sum_{n=0}^{j-1} \sum_{m=0}^{j-n} b'_{kj,nm} z_+^{k+2j-m}\bar
    z_+^{m}\log^nz_+ + \sum_{n=0}^{j-1} \sum_{m=0}^{j-n} c'_{kj,nm}
    z_+^{m}\bar z_+^{k+2j-m}\log^n\bar z_+,
  \end{multline}
  where    the   coefficients   $a_{kj}$,    $a'_{kj}$,   $b_{kj,nm}$,
  $b'_{kj,nm}$,    $c_{kj,nm}$,     and    $c'_{kj,nm}$    are    real
  coefficients.  These  coefficients  are  obtained by  the  induction
  process given by Proposition~\ref{prop:prelim}.
\end{prop}

From this proposition, explicit the primal singular functions
$\Sfrak^{k,p}$ of corner asymptotics~\eqref{cornerasympt} can be
easily obtained. 
Actually, according to~\eqref{eq:Skp}, it is sufficient to make
explicit the functions $\sfrak^{k,p}_j$, which are given by

\begin{equation}
  \label{skp}
  \sfrak^{k,p}_j(r,\theta)=
  \begin{cases}
  \Re\left(\chi_j^k(re^{\ri\theta})\right)\delta_{0}^p
      +\Im\left(\chi_j^k(re^{\ri\theta})\right)\delta_{1}^p 
    ,\,&\text{if $|\theta|\leq \omega/2$,}
    \\[1.5ex]
    \Re\left(\xi_j^k(re^{\ri\theta_{+}})\right)\delta_{0}^p
      +\Im\left(\xi_j^k(re^{\ri\theta_{+}})\right)\delta_{1}^p 
    ,\,&\text{elsewhere ,}
    \end{cases}
\end{equation}
according to~\eqref{eq:complexvar},
where  $\delta_{0}^p$  and $\delta_{1}^p$  are  Kronecker symbols.  We
remind that $\theta_+$ is defined by
$$
\theta_+=
\begin{cases}
  \theta-\pi,\,\text{if $\theta\in (0,\pi]$,}
  \\
  \theta+\pi,\,\text{if $\theta\in [-\pi,0)$.}
\end{cases}
$$
Then, we develop relation \eqref{eq:Skp} with respect to the power of
the logarithmic term $\log r$, and we identify the obtained
expressions with the previous expression \eqref{eq:Skp_angular} of
$\Sfrak^{k,p}(r,\theta)$.  

For instance, let us determine
$\Phi^{k,0}_{j,n}$ in the sector $\Scal_-$, for $0\leq n\leq j$.
Using \eqref{vkj}, after tedious calculus,  we obtain: 
\begin{align*}
\Phi^{k,0\, -}_{j,j}(\theta)&= (a_{kj}+a'_{kj})\cos (k+2j)\theta\ ,\quad \text{for $n=j$,}
\intertext{and for any $n\leqslant j-1$, }
  \Phi^{k,0\, -}_{j,n}(\theta)&= 
  \begin{pmatrix}
    n\\j
  \end{pmatrix}
  \theta^{j-n} (a_{kj}+a'_{kj})\cos \left( (k+2j)\theta+\frac\pi2 (j-n)\right)
  \\
  &+\di\sum_{q=0}^n 
  \begin{pmatrix}
    q\\n-q
  \end{pmatrix}
  \theta^{n-2q} \di\sum_{m=0}^{j+q-n}(b_{kj, n-q \, m} +c_{kj, n-q\, m} ) 
  \cos \left( (k+2j-2m)\theta+\frac\pi2 (n-2q)\right) \ .
\end{align*}
A similar reasoning leads to the generic expression of the dual singularities.


\addcontentsline{toc}{section}{References} 
\bibliographystyle{mnachrn}
\bibliography{biblioeddy}

\end{document}